\documentclass{amsart}

\usepackage{amssymb, latexsym, amscd, graphicx, psfrag}
\usepackage{xcolor}
\usepackage[mathscr]{euscript}

\newcommand{\bA}{ \mathbb{A} }
\newcommand{\bC}{ {\mathbb{C}} }

\newcommand{\bP}{\mathbb{P}}

\newcommand{\bT}{\mathbb{T}}

\newcommand{\bZ}{\mathbb{Z}}

\newcommand{\cH}{\mathcal{H}}

\newcommand{\cL}{\mathcal{L}}
\newcommand{\cM}{\mathcal{M}}
\newcommand{\cO}{\mathcal{O}}

\newcommand{\cS}{\mathcal{S}}

\newcommand{\cX}{\mathcal{X}}

\newcommand{\cW}{\mathcal{W}}

\newcommand{\age}{\mathrm{age}}
\newcommand{\Aut}{\mathrm{Aut}}
\newcommand{\CR}{ {\mathrm{CR}} }

\newcommand{\Res}{\mathrm{Res}}

\newcommand{\ev}{\mathrm{ev}}
\newcommand{\val}{ {\mathrm{val}} }

\newcommand{\bu}{\mathbf{u}}

\newcommand{\btau}{\boldsymbol{\tau}}

\newcommand{\Mbar}{\overline{\cM}}

\newtheorem{dummy}{dummy}[section]
\newtheorem{lemma}[dummy]{Lemma}
\newtheorem{theorem}[dummy]{Theorem}

\newtheorem{corollary}[dummy]{Corollary}
\newtheorem{proposition}[dummy]{Proposition}

\theoremstyle{definition}
\newtheorem{definition}[dummy]{Definition}

\theoremstyle{remark}
\newtheorem{remark}[dummy]{Remark}

\numberwithin{equation}{section}

\title[Mirror symmetry for binary dihedral CY3]{Involution-equivariant topological recursion and mirror symmetry for the affine binary dihedral Calabi--Yau threefold}
\author{Bohan Fang}
\address{Bohan Fang, Beijing International Center for Mathematical Research, Peking University, 5 Yiheyuan Road, Beijing 100871, China}
\email{bohanfang@gmail.com}

\author{Zhuoming Lan}
\address{Zhuoming Lan, Beijing International Center for Mathematical Research, Peking University, 5 Yiheyuan Road, Beijing 100871, China}
\email{2606399068@pku.edu.cn}

\author{Jingxiang Ma}
\address{Jingxiang Ma, University of Sheffield, School of Mathematical and Physical Sciences, Hounsfield Road, Sheffield S3 7RH, United Kingdom}
\email{ma.3149@osu.edu}

\date{\today}

\begin{document}

\begin{abstract}
  We prove a closed-string remodeling statement for the affine binary
  dihedral Calabi--Yau orbifold threefold $\cX=[\bC^2/\Gamma]\times\bC$,
  where $\Gamma$ is a binary dihedral subgroup of $SU(2)$.  This target
  lies outside the toric setting of the
  Bouchard--Klemm--Mari\~{n}o--Pasquetti remodeling conjecture: the toric
  mirror curve is replaced by the type-$D_l$ relativistic Toda curve of
  Brini--Ma--Strachan, and the Chekhov--Eynard--Orantin topological
  recursion is replaced by the $\bZ_2$-equivariant topological recursion
  of Giacchetto--Kramer--Lewa\'nski, run in the sign sector of the
  Toda-curve involution with the Prym kernel as its two-point input.  We
  identify the equivariant orbifold quantum cohomology Frobenius manifold
  of $\cX$ with the invariant Jacobian Frobenius structure of the Toda
  curve, and we prove that the B-model $R$-matrix, defined by regularized
  stationary phase, equals the A-side normalized canonical
  Givental--Teleman $R$-matrix {on a chosen simply connected domain;} this
  equality is anchored at the orbifold point through a semistable
  degeneration of the Toda curve.  Comparing the resulting
  Givental--Teleman and Dunin-Barkowski--Orantin--Shadrin--Spitz graph
  sums then identifies, after a parity-twisted leaf substitution, the
  sign-sector recursion with the descendant Gromov--Witten generating
  functions of $\cX$ in the stable range ($2g-2+n>0$ with $n>0$), and
  identifies the recursion free energies with the equivariant
  Gromov--Witten free energies of $\cX$ for $g\geq2$.  {We also match the
  genus-one potential defined by the reduced calibrated tau function with
  the Gromov--Witten potential up to an additive constant.}
\end{abstract}

\maketitle

\tableofcontents


\section{Introduction}

\subsection{Background and motivation}

\subsubsection{Remodeling for toric Calabi--Yau threefolds}

Mirror symmetry relates the Gromov--Witten theory of a Calabi--Yau target to
complex geometry on a mirror space.  For noncompact targets one usually
works equivariantly, and the A-model is the equivariant Gromov--Witten (GW)
theory of the target.  On the B-side, the mirror often encodes the
higher-genus theory in a curve with a recursive structure.  {This viewpoint
underlies the Bouchard--Klemm--Mari\~{n}o--Pasquetti (BKMP) remodeling}
conjecture \cite{bouchard2007remodelingbmodel,
bouchard2008topologicalopenstringsorbifolds}, based on the work of Mari\~no
\cite{marino2006openstringamplitudeslarge}, where the higher-genus
amplitudes in the toric Calabi--Yau threefold setting are read from the
Chekhov--Eynard--Orantin topological recursion
\cite{chekhov2006freeenergytopologicalexpansion,
eynard2007invariantsalgebraiccurvestopological}.

Eynard--Orantin proved the remodeling conjecture for smooth toric Calabi--Yau
threefolds \cite{eynard2012computationopengromovwitten}. Fang--Liu--Zong proved
it for the affine toric Calabi--Yau 3-orbifolds $[\bC^3/G]$
\cite{fang2013allgenusopenclosedmirror} and then for all semi-projective toric
Calabi--Yau 3-orbifolds \cite{fang2019remodelingconjecturetoriccalabiyau}.
The A-model is
written as a Givental--Teleman graph sum
\cite{givental2001semisimplefrobeniusstructureshigher,
teleman2012structure2dsemisimplefield}; the B-model is written as a graph sum
through the theorem of Dunin-Barkowski--Orantin--Shadrin--Spitz (DOSS)
\cite{duninbarkowski2012identificationgiventalformulaspectral}; the mirror
theorem and the $R$-matrix comparison then identify the graph weights.  Thus the
all-genus comparison is reduced to matching vertices, edges, and leaves.

\subsubsection{The non-toric binary-dihedral problem}

The target of this paper is the affine binary-dihedral Calabi--Yau orbifold
threefold
\[
  \cX=[\bC^2/\Gamma]\times\bC,
\]
where $\Gamma$ is a binary dihedral subgroup of the special unitary group
$SU(2)$. {The target is Calabi--Yau because $\Gamma$ preserves
the volume form, and the torus weights below make the action preserve it as well.
The target lies outside the toric setting. We seek a mirror curve and a recursion
that recover its Gromov--Witten theory.}

The answer used here comes from the type-$D_l$ relativistic Toda curve.  The
binary-dihedral group gives rise to the type-$D_l$ case of the ADE McKay correspondence,
and the genus-zero Gromov--Witten theory of the associated surface is governed by the root-system and quantum
McKay results of Bryan--Gholampour and Hu
\cite{bryan2007rootsystemsquantumcohomology,
bryan2008quantummckaycorrespondencepolyhedral,
hu2012quantummckaycorrespondencesingularities}.  Brini--Ma--Strachan construct
the corresponding Toda B-model and identify its
Frobenius structure, in the normalization adapted to the surface
$\bC^2/\Gamma$, with the ADE genus-zero mirror
{\cite{brini2025dubrovindualitymirrorsymmetry}.}
{We incorporate the Calabi--Yau third leg, specialize the existing
finite-group cohomological field theory (CohFT) to $\Gamma$, and prove the
all-genus comparison with the Toda B-model.}

\subsection{The affine binary-dihedral Calabi--Yau orbifold}

Throughout the paper $\Gamma$ is a binary dihedral subgroup of the special
unitary group $SU(2)$ of order $4(l-2)$, with $l\geq4$, and
\[
  \cX=[\bC^2/\Gamma]\times\bC.
\]
The torus $\bT=\bC^\times$ acts on the three factors with weights
$
  (\nu,\nu,-2\nu),
$
so the action is Calabi--Yau.  {Orbifold quantum Riemann--Roch
gives a graph-sum formula for the equivariant Gromov--Witten theory of $\cX$}
\cite{tseng2009orbifoldquantumriemannrochlefschetz,
lan2023twistedequivariantgromovwittentheory}.

At the orbifold point (the origin of the flat orbifold coordinates of
Section~\ref{sec:frobenius}) the Frobenius algebra is the center of the group algebra
$
  Z(\bC(\nu)[\Gamma]),
$
with the threefold-normalized Frobenius metric.  Its primitive idempotents are
indexed by irreducible representations of $\Gamma$, hence the rank is $l+1$, as
expected from the type-$D_l$ McKay data
\cite{jarvis2002orbifoldquantumcohomologyclassifying,
bryan2007rootsystemsquantumcohomology,
hu2012quantummckaycorrespondencesingularities}.
Section~\ref{sec:amodel} records this A-side data, including the normalized
canonical $R$-matrix and the scalar factor contributed by the third leg.

\subsection{The type-$D_l$ relativistic Toda B-model}

\subsubsection{The Toda curve}

The B-model starts from the type-$D_l$ relativistic Toda curve of
Brini--Ma--Strachan.  This construction builds on the extended affine Weyl
mirror symmetry of Brini--van Gemst, where the B-model is obtained from affine
relativistic Toda spectral curves \cite{brini2023mirrorsymmetryextendedaffine};
the underlying extended affine Weyl Frobenius manifolds are those of
Dubrovin--Zhang and Dubrovin--Strachan--Zhang--Zuo
\cite{dubrovin1996extendedaffineweylgroups,
dubrovin2015extendedaffineweylbcd}.
{We use this construction as genus-zero input, with primitive
differential $\phi_D$ and the threefold-normalized invariant Jacobian structure.}

In the normalization of Section~\ref{sec:bmodel}, the compactified curve
$\bar C_\kappa\subset\bP^1_\mu\times\bP^1_\lambda$ is given by
\[
\begin{aligned}
  F(\mu,\lambda;\kappa)
  =
  \lambda\,\mu^{l-2}(\mu^2-1)^2
  -K\prod_{i=1}^l(\mu-\kappa_i)(\mu-\kappa_i^{-1})
  =0,
  \\
  \lambda(w)=K{\prod_{i=1}^l(w-c_i)\over w^2-4},
  \qquad
  w=\mu+\mu^{-1},
  \qquad
  c_i=\kappa_i+\kappa_i^{-1}.
\end{aligned}
\]
The superpotential used for topological recursion is
\[
  \hat x=-2\nu\log\lambda,
\]
and the oscillatory phase is $W=-\hat x=2\nu\log\lambda$.  Both this curve and
this phase convention are fixed throughout the paper.  The involution
\[
  \iota:\mu\longmapsto \mu^{-1}
\]
fixes $\lambda$ and $w$.

\subsubsection{The involution-equivariant sign sector}

The full smooth $\mu$-curve has two
lifts of each quotient critical point, while the A-model Frobenius algebra has
rank $l+1$.  Therefore the B-model comparison cannot use the ordinary
topological recursion on the full upstairs curve.  Instead the recursion is run
{on the representative disks centered at the quotient critical points, with differentials valued in
the sign local system.}

The two-point input is the Prym kernel
\[
  B^-=B^{\rm std}-(\operatorname{id}\times\iota)^*B^{\rm std}.
\]
Since $d\hat x$ is invariant and $\phi_D$ is anti-invariant, the periods which
enter the comparison lie in the sign sector.  Thus the involution $\iota$ forces
us to use the reduced local curve
$
  \cS^\iota_\kappa
$
of Section~\ref{sec:bgraphsum}, which collects this data.  Its stable outputs
$\omega^\iota_{g,n}$ pull back to $\iota$-anti-invariant forms in every external
variable.  We use the $\bZ_2$-equivariant topological recursion of
Giacchetto--Kramer--Lewa\'nski \cite{giacchetto2024newspinhurwitztheory}, applied to the involution $\iota$ and to the
sign sector of the type-$D_l$ relativistic Toda curve.  {Local sign
trivializations identify {its formulation on one representative disk per orbit} with an ordinary local spectral curve,
to which the DOSS theorem applies.}

\subsection{Statement of the main results and outline of the proof}

\subsubsection{Genus-zero mirror theorem}

We first identify the Frobenius structures.  On the A-side we use
the equivariant orbifold quantum cohomology Frobenius manifold of $\cX$.  On the
B-side we use the invariant Jacobian Frobenius structure of the type-$D_l$
relativistic Toda curve, after multiplying the surface metric and cubic tensor by
the Euler factor of the third torus weight.  The affine change of variables is
the centered orbifold-frame coordinate change of Definition~\ref{def:mirror-map}.

The genus-zero comparison is Theorem~\ref{thm:frob-iso}.  It says that, on the
Frobenius domain $\Omega_{\rm Frob}$,
\[
  \Phi:\cM^B|_{\Omega_{\rm Frob}}
  \longrightarrow
  \cM^{\cX}|_{\Phi(\Omega_{\rm Frob})}
\]
is an isomorphism of semisimple Frobenius manifolds, preserving the unit,
product, and metric.  This isomorphism combines the Brini--Ma--Strachan
surface mirror theorem, the Bryan--Gholampour
and Hu ADE quantum McKay input, and the third-leg normalization of
Section~\ref{sec:amodel}
\cite{brini2025dubrovindualitymirrorsymmetry,
bryan2007rootsystemsquantumcohomology,
hu2012quantummckaycorrespondencesingularities}.

\subsubsection{{The $R$-matrix comparison on $\Omega_B$}}

The higher-genus comparison is controlled by $R$-matrices.  On the A-side the
$R$-matrix is the normalized canonical Givental--Teleman calibration of the
twisted $B\Gamma$ CohFT.  On the B-side the matrix $\hat R^B$ is defined by
regularized stationary-phase transforms of the second-kind forms built from the
Prym kernel. {The comparison takes place on the simply connected
domain $\Omega_B$, with fixed thimbles and logarithm branches.}

Theorem~\ref{thm:anchor} proves that, in the chosen normalized canonical frame,
\[
  \hat R^B(\kappa;z)
  =
  \hat R^\cX_{\rm can}(\Phi(\kappa);z),
  \qquad
  \kappa\in\Omega_B.
\]
The orbifold point is used as boundary data for this equality, not as a point at
which a central-fiber B-side $R$-matrix is defined.  The proof of this equality
first shows that both matrices solve the
same Dubrovin equation \cite{dubrovin1994geometry2dtopologicalfield} after the
genus-zero mirror identification, leaving a
diagonal symplectic gauge. {Section~\ref{sec:anchor} removes this
gauge by matching the boundary limits of each canonical column contracted with the flat unit.}

\subsubsection{Stable descendant remodeling}

For descendants, the B-model graph sum must be compared with the shifted
CohFT ancestor graph sum.  The DOSS graph sum is written in dressed leaves
$\theta^k_\alpha$.  Section~\ref{sec:graphsum-remodeling} first rewrites those
leaves in terms of the undressed leaves $\widehat\theta^k_\alpha$, and then
makes the parity-twisted substitution
\[
  \widehat\theta^k_\alpha(q_j)
  \longmapsto
  (-1)^{k+1}\widetilde u_j^{k,\alpha}.
\]
This produces the B-side formal series
$\cW^{B,\Pi}_{g,n}$; the sign in this substitution is the one which cancels the
local graph parities.

The stable descendant theorem is Theorem~\ref{thm:formal-graphsum-remodeling}.
For $n>0$, $2g-2+n>0$, and $\kappa\in\Omega_B$,
\[
  \cW^{B,\Pi}_{g,n}(\widetilde\bu_1,\ldots,\widetilde\bu_n;\kappa)
  =
  \mathcal F^{\cX,{\rm an}}_{g,n}(\bu_1,\ldots,\bu_n;\kappa).
\]
{The equality holds first for polynomial descendant inputs and then in the
{completed ring in the descendant variables.}  The theorem is an analogue of}
descendant remodeling: the ordinary toric
mirror recursion is replaced by the
involution-equivariant sign-sector recursion
\cite{fang2025remodelingconjecturedescendants,
duninbarkowski2012identificationgiventalformulaspectral}.

\subsubsection{{Free energies}}

{The free-energy statement has no ordinary external leaf.
Theorem~\ref{thm:formal-free-energy-remodeling} follows from the one-point graph
comparison, the recursion residue formula, and the shifted-CohFT dilaton equation.
For $g\geq2$ and $\kappa\in\Omega_B$,}
\[
  F_g(\cS^\iota_\kappa)
  =
  F_g^{\mathrm{GW},{\rm an}}(\cX;\kappa).
\]
{{This residue formula} concerns the stable genera $g\geq2$}
\cite{eynard2007invariantsalgebraiccurvestopological,
eynard2009geometricalinterpretationtopologicalrecursion}.
{Proposition~\ref{prop:genus-one-free-energy} proves
$dF_1^B=dF_1^{\mathrm{GW},{\rm an}}$ on the chosen branch, where $F_1^B$
is defined by the reduced calibrated tau function.  Thus the genus-one
potentials agree up to an additive constant.}

\subsubsection{{Proof strategy}}

{The main new step is the boundary analysis at the orbifold point.
The Toda numerator has a cyclotomic factorization there, and the compactified
curve has a nonreduced flat limit. Expanding the target at the two fixed fibers
gives a semistable curve with one main component and two bubbles.
We use its local geometry to compute the coefficientwise boundary limits of
the unit contractions and determine the diagonal gauge of the $R$-matrix.

Once the $R$-matrices agree, the Frobenius isomorphism identifies the vertex
factors, and the edge and leaf identities complete the graph comparison.
The local parity signs cancel the DOSS factor $(-1)^{g-1}$. The results
concern the closed sector: the descendant theorem treats $n>0$ in the stable
range, and the {residue} free-energy theorem treats $g\geq2$.}
{The coefficient of $z$ in the calibrated $R$-matrix gives the
genus-one differential comparison.}

\subsection{{Relation to previous work}}

{We follow the graph comparison used in toric and
descendant remodeling
\cite{fang2019remodelingconjecturetoriccalabiyau,
fang2025remodelingconjecturedescendants}. The new comparison uses the Toda
curve and the sign sector of the equivariant recursion of
Giacchetto--Kramer--Lewa\'nski \cite{giacchetto2024newspinhurwitztheory},
with the orbifold boundary calculation fixing its calibration.
Related work connects ADE and Toda geometries with Chern--Simons theory
on spherical Seifert manifolds
\cite{borot2015chernsimonstheorysphericalseifert}; here we identify the recursion
with the equivariant Gromov--Witten theory of $[\bC^2/\Gamma]\times\bC$.}

\subsection{Overview of the paper}

Section~\ref{sec:amodel} develops the equivariant orbifold GW theory of
$\cX$, reduces it to a twisted theory of $B\Gamma$, and records the A-model
$R$-matrix.  Section~\ref{sec:bmodel} introduces the type-$D_l$ logarithmic
Toda curve, the orbifold locus, and the semistable degeneration.
Section~\ref{sec:bgraphsum} constructs the reduced, involution-equivariant B-model
graph sum.  Section~\ref{sec:frobenius} proves the Frobenius-manifold
isomorphism and the flat algebra extension through the pole-cancellation
locus.  Section~\ref{sec:anchor} proves {the $R$-matrix comparison on $\Omega_B$} by orbifold-point gauge fixing.  Finally,
Section~\ref{sec:graphsum-remodeling} proves stable descendant remodeling and the
{free-energy comparisons, with genus one treated up to an additive constant.}
{Appendix~\ref{app:boundary-characters} records the explicit boundary label dictionary and
character- and sector-basis formulas, which are supplementary to the main proof.}

\subsection{AI disclosure}
\label{subsec:rethlas-ai}

The mathematical content of the first version of this paper was developed with the assistance of a system coordinating multiple Rethlas reasoning agents \cite{ju2026automatedconjectureresolutionformal}. The system was designed following a suggestion from Jihao Liu and was inspired by Danus \cite{liu2026danusorchestratingmathematicalreasoning}. The authors provided the mathematical setup, a strategy for solving the problem, and guidance throughout the process. In particular, they supplied observations on the degeneration of the mirror curve and constraints on the fundamental bidifferential.

For the present version, the authors thoroughly reviewed and manually revised the entire manuscript following its initial posting on arXiv. They made substantial revisions to the mathematical presentation throughout the paper. The authors take full responsibility for the content and correctness of the entire paper.

\subsection*{Acknowledgments}

The first-named author thanks Chiu-Chu Melissa Liu, Song Yu, and Zhengyu Zong for valuable mathematical discussions. He is grateful to Bin Dong and Jihao Liu for their help in designing the Rethlas-based agent system used in this work, and to the AI4Math project at Peking University for providing a stimulating and productive environment for AI--math interaction. He also thanks Yin Wu for discussions on the use of AI and agent harnesses. The second-named author thanks Xiaowen Hu and Zhengyu Zong for helpful mathematical discussions. The third-named author thanks Andrea Brini for helpful discussions.

The work of BF is partially supported by the National Key R\&D Program of China (2023YFA1009803) and by NSFC grants 12125101 and 11890661. JM was partially supported by a PhD studentship from the EPSRC Doctoral Training Partnership (EP/W524360/1).


\section{Equivariant orbifold Gromov--Witten theory}
\label{sec:amodel}

{In this section, we specialize the finite-group orbifold Gromov--Witten theory of
\cite{jarvis2002orbifoldquantumcohomologyclassifying,
lan2023twistedequivariantgromovwittentheory} to the binary dihedral target $\cX$.
We recall its reduction to a twisted theory of $B\Gamma$ and record the
orbifold-point Frobenius algebra and normalized canonical Givental--Teleman
$R$-matrix in the threefold normalization used in the later comparisons.}

Throughout, $\Gamma\subset SU(2)$ is the binary dihedral group of order $4(l-2)$ with $l\geq 4$, and
\[
  \cX=[\bC^2/\Gamma]\times\bC,
\]
where $\Gamma$ acts on the first factor through its defining two-dimensional representation on $V=\bC^2$ and
trivially on the last factor $\bC$. We let the torus $\bT=\bC^\times$ act with weights
$(\nu,\nu,-2\nu)$, so that the action is Calabi--Yau. 

\subsection{Reduction to a twisted theory of $B\Gamma$}

Let $F_V\to B\Gamma$ be the vector bundle associated with the
defining representation $V$ of $\Gamma$, and let $F_{\mathbf 1}$
be the trivial line bundle.  We equip them with $\bT$-weights $\nu$
and $-2\nu$, respectively, and set
\[
  F_\rho := (F_V)_\nu\oplus (F_{\mathbf 1})_{-2\nu},
  \qquad \rho=V\oplus\mathbf 1.
\]
We write
\[
  e_1:=e_\bT(V_\nu\oplus\cO_{-2\nu})=-2\nu^3
\]
for the Euler factor on the identity sector; the subscript $1$ refers
to the identity element of $\Gamma$.
Let
\[
  \pi:\mathcal C_{g,n}\longrightarrow
  \overline{\cM}_{g,n}(B\Gamma),
  \qquad
  f:\mathcal C_{g,n}\longrightarrow B\Gamma,
\]
be the universal curve and the universal map.  Following
\cite{lan2023twistedequivariantgromovwittentheory}, define the equivariant $K$-theory class
\[
  F_{\rho,g,n}
  :=R^\bullet\pi_*f^*F_\rho
  =R^0\pi_*f^*F_\rho-R^1\pi_*f^*F_\rho
  \in K_\bT\bigl(\overline{\cM}_{g,n}(B\Gamma)\bigr).
\]

\begin{definition}\label{def:loc-twisted-bg}
Under the localized state-space identification
\begin{equation}\label{eq:amodel:state-space-identification}
  H^*_{\CR,\bT}(\cX)\otimes_{\bC[\nu]}\bC(\nu)
  \cong
  H^*_{\CR}(B\Gamma;\bC)\otimes_\bC\bC(\nu),
\end{equation}
and with $\bar\psi_j$ denoting the descendant cotangent class at the
$j$-th marking on $\overline{\cM}_{g,n}(B\Gamma)$,
the $\bT$-equivariant descendant correlators of $\cX$, for insertions
$\beta_j\in H^*_{\CR,\bT}(\cX)\otimes_{\bC[\nu]}\bC(\nu)$, are defined by
\begin{equation}\label{eq:amodel:localization}
  \left\langle\prod_{j=1}^n\tau_{a_j}(\beta_j)\right\rangle^{\cX,\bT}_{g,n}
  :=
  \int_{\overline{\cM}_{g,n}(B\Gamma)}
  \prod_{j=1}^n\bar\psi_j^{a_j}\,\ev_j^*\beta_j\,
  \frac{1}{e_\bT(F_{\rho,g,n})}.
\end{equation}
Equivalently,
\[
  \frac{1}{e_\bT(F_{\rho,g,n})}
  =
  \frac{e_\bT(R^1\pi_*f^*F_\rho)}
       {e_\bT(R^0\pi_*f^*F_\rho)}.
\]
\end{definition}

{
For $2g-2+n>0$, let
\[
  \operatorname{st}:\overline{\cM}_{g,n}(B\Gamma)
  \longrightarrow\Mbar_{g,n}
\]
be the morphism that forgets the map to $B\Gamma$ and stabilizes the underlying
coarse marked curve. For insertions $\beta_j$ in the localized state space above,
define the unshifted equivariant Gromov--Witten CohFT classes by
\begin{equation}\label{eq:amodel:unshifted-cohft}
  \Omega_{g,n}(\beta_1,\ldots,\beta_n)
  :=\operatorname{st}_*\left(
    \frac{\prod_{j=1}^n\ev_j^*\beta_j}{e_\bT(F_{\rho,g,n})}
  \right)
  \in H^*(\Mbar_{g,n};\bC)\otimes_\bC\bC(\nu).
\end{equation}
}

For $[h]\in\operatorname{Conj}(\Gamma)$, let
$\mathbf 1_{[h]}$ denote the fundamental class of the corresponding
sector of $IB\Gamma$.  Following the age normalization of
\cite{lan2023twistedequivariantgromovwittentheory}, set
\[
  \overline{\mathbf 1}_{[h]}
  :=\nu^{-\operatorname{age}_V(h)}\mathbf 1_{[h]}.
\]
The trivial summand has age zero, while
\[
  \operatorname{age}_V(1)=0,
  \qquad
  \operatorname{age}_V(h)=1\quad(h\neq 1).
\]

{
We distinguish the natural localization identification
\eqref{eq:amodel:state-space-identification}, which identifies the unnormalized
sector classes, from the age-normalized map
\[
  J:H^*_{\CR}(B\Gamma;\bC)\otimes\bC(\nu)
  \longrightarrow H^*_{\CR,\bT}(\cX)\otimes_{\bC[\nu]}\bC(\nu),
  \qquad
  J(\mathbf1_{[h]})=\overline{\mathbf1}_{[h]}.
\]
For every descendant index $a\geq0$, choose scalar formal coordinates
$u_{a,[h]}$ and put
\[
  \mathbf u_a=\sum_{[h]}u_{a,[h]}\mathbf1_{[h]},
  \qquad
  \widehat{\mathbf u}_a=J(\mathbf u_a)
  =\sum_{[h]}u_{a,[h]}\overline{\mathbf1}_{[h]}.
\]
Thus $\mathbf u_a$ is a $B\Gamma$ input and $\widehat{\mathbf u}_a$ is its
transport to $\cX$; their coefficients in these respective bases agree.
In the unnormalized $\cX$-sector basis, the coefficient of
$\mathbf1_{[h]}$ in $\widehat{\mathbf u}_a$ is instead
$\nu^{-\age_V(h)}u_{a,[h]}$. Write
\[
  \mathbf u(z)=\sum_{a\geq0}\mathbf u_a z^a,
  \qquad
  \widehat{\mathbf u}(z)=J(\mathbf u(z)).
\]

Let $\mathcal D^{\cX}(\widehat{\mathbf u},\hbar)$ and
$\mathcal D^{B\Gamma}(\mathbf u,\hbar)$ denote the total descendant
potentials of $\cX$ and the untwisted theory of $B\Gamma$, respectively.
If $\eta_{B\Gamma}$ is the standard orbifold Poincar\'e pairing, then
\begin{equation}\label{eq:amodel:metric-scaling}
  J^*\eta^\cX=e_1^{-1}\eta_{B\Gamma}.
\end{equation}
This scalar relation uses $J$. Under the natural localization identification,
the unnormalized sector pairing is instead
\[
  \eta^\cX(\mathbf1_{[h]},\mathbf1_{[k]})
  =\frac{\delta_{[k],[h^{-1}]}}{|C_\Gamma(h)|e_h},
  \qquad
  e_h=\begin{cases}-2\nu^3,&h=1,\\-2\nu,&h\ne1.\end{cases}
\]
In the transported variables, quantum Riemann--Roch gives
\begin{equation}\label{eq:amodel:qrr-potential}
  \mathcal D^{\cX}(\widehat{\mathbf u},\hbar)
  =\operatorname{Quant}_{\eta^\cX}
  \bigl(\hat R^{\cX}_{\rm orb}\bigr)
  \left[\mathcal D^{B\Gamma}
  (J^{-1}\widehat{\mathbf u},e_1\hbar)\right].
\end{equation}
Here the expression in brackets is regarded as a function of the
$\cX$ input $\widehat{\mathbf u}$, and
$\operatorname{Quant}_{\eta^\cX}$ denotes Givental quantization with
pairing $\eta^\cX$ and the usual dilaton shift by its unit.
The variable $e_1\hbar$ supplies the factor $e_1^{g-1}$ in genus $g$.
In these age-normalized sector coordinates, the positive-power terms of
quantum Riemann--Roch give the normalized series
\[
  \hat R^{\cX}_{\rm orb}(z)=I+R_1z+R_2z^2+\cdots.
\]
}

\subsection{The orbifold-point Frobenius algebra}

At the orbifold point the quantum product specializes to the classical
equivariant Chen--Ruan product.  Its Frobenius form carries a scalar
rescaling that enters the normalized canonical basis used below.

\begin{lemma}\label{lem:frob}
In the age-normalized sector basis
$\{\overline{\mathbf 1}_{[h]}\}_{[h]\in\operatorname{Conj}(\Gamma)}$
defined above, the orbifold-point Frobenius algebra is
\begin{equation}\label{eq:amodel:orbifold-frobenius}
  \bigl(H^*_{\CR,\bT}(\cX)\otimes_{\bC[\nu]}\bC(\nu),\star,\eta^\cX\bigr)
  \cong
  \bigl(Z(\bC(\nu)[\Gamma]),\cdot,e_1^{-1}\eta_{B\Gamma}\bigr),
\end{equation}
The isomorphism sends
\[
  \overline{\mathbf 1}_{[h]}
  \longmapsto \sum_{g\in[h]}g,
\]
and $\eta_{B\Gamma}$ on the right denotes the classifying-space
Frobenius form transported by this isomorphism.
{In particular, the algebra is semisimple over
$\bC(\nu)$, with primitive idempotents indexed by the irreducible representations
$V_\alpha\in\mathrm{Irr}(\Gamma)$.}
\end{lemma}
\begin{proof}
In the age-normalized sector basis, {the orbifold-point product twisted by the inverse Euler class} reduces to the untwisted product of $B\Gamma$.
By the Jarvis--Kimura description
\cite{jarvis2002orbifoldquantumcohomologyclassifying}, the latter is
the class-algebra product on $Z(\bC[\Gamma])$.  After extending scalars
to $\bC(\nu)$ and using \eqref{eq:amodel:metric-scaling}, this gives
\eqref{eq:amodel:orbifold-frobenius}.  Semisimplicity follows from that
of the finite group algebra.
\end{proof}

{
For later use, the primitive central idempotents are
\begin{equation}\label{eq:amodel:primitive-idempotents}
  e_\alpha=\frac{d_\alpha}{|\Gamma|}
    \sum_{h\in\Gamma}\chi_\alpha(h^{-1})h,
  \qquad d_\alpha=\dim V_\alpha.
\end{equation}
Character orthogonality gives
\begin{equation}\label{eq:amodel:idempotent-norms}
  \eta^\cX(e_\alpha,e_\beta)
  =\delta_{\alpha\beta}
    \frac{d_\alpha^2}{|\Gamma|^2e_1}.
\end{equation}
Therefore the normalized canonical vectors are
\begin{equation}\label{eq:amodel:normalized-canonical-vector}
  \widetilde e_\alpha
  =\frac{|\Gamma|\sqrt{e_1}}{d_\alpha}e_\alpha,
  \qquad
  \eta^\cX(\widetilde e_\alpha,\widetilde e_\beta)
  =\delta_{\alpha\beta}.
\end{equation}
Choose $\sqrt{e_1}$ in a quadratic extension of $\bC(\nu)$.  The
vectors $\{\widetilde e_\alpha\}$ then form a normalized canonical frame;
the character formula below is independent of this choice and has
entries in $\bC(\nu)$.
}

\subsection{The Givental--Teleman $R$-matrix}

{Consider the localized Gromov--Witten cohomological field theory
(CohFT) of $B\Gamma$ obtained by inserting
$e_\bT(F_{\rho,g,n})^{-1}$ on $\overline{\cM}_{g,n}(B\Gamma)$,
where $F_{\rho,g,n}=R^\bullet\pi_*f^*F_\rho$, as in
Definition~\ref{def:loc-twisted-bg}. Orbifold quantum Riemann--Roch
determines its normalized $R$-matrix $\hat R^{\cX}_{\rm orb}(z)$ at the
orbifold point. Givental--Teleman reconstruction recovers this CohFT
from its underlying topological field theory and this $R$-matrix,
with the associated unit-preserving translation}
\cite{givental2001semisimplefrobeniusstructureshigher,
teleman2012structure2dsemisimplefield}. {We compute it in the normalized
canonical frame $\{\widetilde e_\alpha\}$.} For each $h\in\Gamma$, let
$a_{h,1},a_{h,2}\in[0,1)$ be the eigenangles of $h$ on $V$, so that
the eigenvalues are $e^{2\pi i a_{h,1}}$ and
$e^{2\pi i a_{h,2}}$.  For $h\ne1$ they can be ordered as
$(a_h,1-a_h)$ with $a_h\in(0,1)$, while
$(a_{1,1},a_{1,2})=(0,0)$.

\begin{proposition}\label{prop:rmatrix}
In the normalized canonical basis indexed by
$\alpha,\beta\in\mathrm{Irr}(\Gamma)$, the Givental--Teleman
$R$-matrix at the orbifold point of the localized twisted theory in
Definition~\ref{def:loc-twisted-bg} is
\begin{equation}\label{eq:amodel:rmatrix}
  \hat R^{\cX}_{\rm orb}(z)_{\alpha\beta}
  =\frac{1}{|\Gamma|}\sum_{h\in\Gamma}\chi_\alpha(h^{-1})\chi_\beta(h)\,\Theta_h(z)\,E(z),
\end{equation}
where $B_m(x)$ denotes the $m$-th Bernoulli polynomial and
\begin{gather*}
  \Theta_h(z)=\exp\Big(\sum_{t\geq1}\frac{(-1)^t
  \sum_{j=1}^2B_{t+1}(a_{h,j})}{t(t+1)}
  \Big(\frac{z}{\nu}\Big)^t\Big),\\
  E(z)=\exp\Big(\sum_{t\geq1}\frac{(-1)^{t+1}B_{t+1}(0)}{t(t+1)}
  \Big(\frac{z}{2\nu}\Big)^t\Big).
\end{gather*}
\end{proposition}

\begin{proof}
In the age-normalized sector basis, with the metric and genus-variable
rescalings of
\eqref{eq:amodel:qrr-potential}--\eqref{eq:amodel:metric-scaling},
finite-group quantum Riemann--Roch
\cite{lan2023twistedequivariantgromovwittentheory,
tseng2009orbifoldquantumriemannrochlefschetz} {is diagonal in the
twisted-sector basis:}
\[
  \overline{\mathbf 1}_{[h]}
  \longmapsto
  S_h(z)\overline{\mathbf 1}_{[h]},
  \qquad
  S_h(z):=\Theta_h(z)E(z).
\]
The exponents defining $S_h$ have zero constant term, and $h$ and
$h^{-1}$ have the same unordered eigenangle pair.  Therefore
\begin{equation}\label{eq:amodel:sector-symmetry}
  S_h(0)=1,
  \qquad
  S_{h^{-1}}(z)=S_h(z).
\end{equation}

Set
\[
  \epsilon_{[h]}
  :=\sqrt{e_1|C_\Gamma(h)|}\,
    \overline{\mathbf 1}_{[h]}.
\]
{
Lemma~\ref{lem:frob} and character orthogonality give
\begin{align}
  \eta^\cX(\epsilon_{[h]},\epsilon_{[k]})
  &=\delta_{[k],[h^{-1}]},\label{eq:amodel:normalized-sector-pairing}\\
  \widetilde e_\alpha
  &=\sum_{[h]\in\operatorname{Conj}(\Gamma)}
    \frac{\chi_\alpha(h^{-1})}{\sqrt{|C_\Gamma(h)|}}\,
    \epsilon_{[h]},\label{eq:amodel:character-transform}\\
  \epsilon_{[h]}
  &=\sum_{\beta\in\mathrm{Irr}(\Gamma)}
    \frac{\chi_\beta(h)}{\sqrt{|C_\Gamma(h)|}}\,
    \widetilde e_\beta.\label{eq:amodel:inverse-character-transform}
\end{align}
Therefore
\[
\begin{aligned}
  \hat R^{\cX}_{\rm orb}(z)\widetilde e_\beta
  &=
  \sum_{[h]\in\operatorname{Conj}(\Gamma)}
  \frac{\chi_\beta(h^{-1})S_h(z)}
       {\sqrt{|C_\Gamma(h)|}}\,\epsilon_{[h]}\\
  &=
  \sum_\alpha
  \left(
    \sum_{[h]\in\operatorname{Conj}(\Gamma)}
    \frac{\chi_\beta(h^{-1})\chi_\alpha(h)}
         {|C_\Gamma(h)|}\,S_h(z)
  \right)\widetilde e_\alpha\\
  &=
  \sum_\alpha
  \left(
    \frac1{|\Gamma|}\sum_{h\in\Gamma}
    \chi_\alpha(h^{-1})\chi_\beta(h)S_h(z)
  \right)\widetilde e_\alpha,
\end{aligned}
\]
where the last equality uses
$|[h]|/|\Gamma|=|C_\Gamma(h)|^{-1}$ and
\eqref{eq:amodel:sector-symmetry}, proving \eqref{eq:amodel:rmatrix}.
}
\end{proof}

\begin{corollary}[Third-leg factorization]\label{cor:thirdleg}
Since $E(z)$ is independent of $h$, formula
\eqref{eq:amodel:rmatrix} gives
\begin{equation}\label{eq:amodel:third-leg-factor}
  \hat R^{\cX}_{\rm orb}(z)
  =E(z)\,\hat R^{A,\mathrm{surf}}(z),
\end{equation}
where $\hat R^{A,\mathrm{surf}}(z)$ is the normalized canonical
$R$-matrix of the $V_\nu$-twisted theory of $B\Gamma$.
\end{corollary}

\begin{remark}[Metric comparison]\label{rem:amodel:metric-comparison}
Let $\eta^{[\bC^2/\Gamma]}$ denote the Frobenius metric of the
$V_\nu$-twisted theory of $B\Gamma$ in the age-normalized sector
basis.  Its identity-sector Euler factor is
$e_1^{\mathrm{surf}}=\nu^2$.  Since the threefold Euler factor is the
$e_1$ defined above,
\begin{equation}\label{eq:amodel:surface-metric-scaling}
  \eta^\cX=(-2\nu)^{-1}\eta^{[\bC^2/\Gamma]},
\end{equation}
because $e_1^{\mathrm{surf}}/e_1=(-2\nu)^{-1}$.
\end{remark}

\begin{lemma}\label{lem:unitarity}
The matrix $\hat R^{\cX}_{\rm orb}$ satisfies
$\hat R^{\cX}_{\rm orb}(0)=I$ and the symplectic unitarity relation
\begin{equation}\label{eq:amodel:unitarity}
  \hat R^{\cX}_{\rm orb}(z)\,
  \hat R^{\cX}_{\rm orb}(-z)^{\mathsf T}=I,
\end{equation}
in the normalized canonical metric.
\end{lemma}

\begin{proof}
The first identity in \eqref{eq:amodel:sector-symmetry} and character
orthogonality give
$\hat R^{\cX}_{\rm orb}(0)_{\alpha\beta}=\delta_{\alpha\beta}$.
The reflection identity $B_n(1-x)=(-1)^nB_n(x)$ makes
the sum $\sum_{j=1}^2B_{t+1}(a_{h,j})$ vanish for even $t$ when
$h\ne1$; for $h=1$ the same conclusion follows from
$B_{2m+1}(0)=0$ for $m\geq1$.  The same identity removes the even-$t$
terms from $\log E(z)$.  Thus $\log S_h(z)$ contains only odd
powers of $z$, and
\[
  S_h(z)S_h(-z)=1.
\]
If $[k]=[h^{-1}]$, then
$S_k(-z)=S_{h^{-1}}(-z)=S_h(-z)$.  The diagonal action in the sector
basis and \eqref{eq:amodel:normalized-sector-pairing} therefore give
\[
\eta^\cX\bigl(
  \hat R^{\cX}_{\rm orb}(z)\epsilon_{[h]},
  \hat R^{\cX}_{\rm orb}(-z)\epsilon_{[k]}
\bigr)
=S_h(z)S_k(-z)\delta_{[k],[h^{-1}]}
=\delta_{[k],[h^{-1}]}.
\]
In the normalized canonical frame, this identity reads
$\hat R^{\cX}_{\rm orb}(z)^{\mathsf T}
\hat R^{\cX}_{\rm orb}(-z)=I$.  Replacing $z$ by $-z$ shows that
$\hat R^{\cX}_{\rm orb}(-z)^{\mathsf T}$ is a left inverse of
$\hat R^{\cX}_{\rm orb}(z)$.  Since its constant term is $I$, the
latter is invertible as a formal matrix, and \eqref{eq:amodel:unitarity}
follows.
\end{proof}


\section{The B-model: spectral curve, orbifold locus, and nodal degeneration}
\label{sec:bmodel}

{We recall the Toda curve of type $D_l$
\cite{brini2025dubrovindualitymirrorsymmetry} and study its degeneration
at the orbifold locus. The semistable limit has three rational components.
Their limiting bidifferential and ramification points supply the boundary data
for the graph sums and the calibration comparison.}

\subsection{The type-$D_l$ relativistic Toda curve and its conventions}

\begin{definition}[The type-$D_l$ relativistic Toda family]\label{def:toda-curve}
For $l\geq 4$ and parameters
\[
  \kappa=(\kappa_1,\ldots,\kappa_l;K)\in(\bC^\times)^{l+1},
  \qquad K=\kappa_{l+1},
\]
let $\bar C_\kappa\subset\bP^1_\mu\times\bP^1_\lambda$ be the compactification of the type-$D_l$
relativistic Toda curve
\begin{equation}\label{eq:bmodel:toda-curve}
\begin{aligned}
  F(\mu,\lambda;\kappa)=\lambda\,\mu^{l-2}(\mu^2-1)^2
  -K\prod_{i=1}^l(\mu-\kappa_i)(\mu-\kappa_i^{-1})=0.
\end{aligned}
\end{equation}
Equivalently,
\begin{equation}\label{eq:bmodel:lambda-w}
  \lambda(w)=K\frac{\prod_{i=1}^l(w-c_i)}{w^2-4},
  \qquad w=\mu+\mu^{-1},\qquad c_i=\kappa_i+\kappa_i^{-1}.
\end{equation}
We use the locally defined function
\begin{equation}\label{eq:bmodel:superpotential}
  \hat x=-2\nu\log\lambda,
  \qquad
  d\hat x=-2\nu\,{d\lambda\over\lambda}.
\end{equation}
The type-$D_l$ primitive differential is
\begin{equation}\label{eq:bmodel:primitive-differential}
  \phi_D=\sqrt{\nu}\,{d\mu\over\mu},
  \qquad
  \phi_D^2=\nu\left(\frac{d\mu}{\mu}\right)^2.
\end{equation}
The residue metric and cubic tensor associated with
$\phi_D$ are normalized in Section~\ref{sec:frobenius}. The local conjugate
coordinate in the Chekhov--Eynard--Orantin recursion of
Section~\ref{sec:bgraphsum} {is scaled so that the vertex factors in the
recursion graph sum agree with those determined by the threefold-normalized
Frobenius metric; see \eqref{eq:bgraphsum:gaussian-normalization}.}
The involution $\iota:\mu\mapsto\mu^{-1}$ fixes $\lambda$; on each smooth
fiber we use the standard genus-zero Bergman kernel
\[
  B^{\rm std}=\frac{d\mu_1\,d\mu_2}{(\mu_1-\mu_2)^2}.
\]
\end{definition}

We pass freely between the upstairs equation
\eqref{eq:bmodel:toda-curve} and its quotient form
\eqref{eq:bmodel:lambda-w}. Since the normalization of $\bar C_\kappa$ is
rational on the smooth locus, $B^{\rm std}$ is the unique fundamental
bidifferential and carries no period normalization. The components appearing
in the singular fibers below have rational normalizations as well.

\subsection{The orbifold locus}

\begin{definition}[Orbifold locus and orbifold point]\label{def:orbifold-point}
Let
\begin{equation}\label{eq:bmodel:orbifold-exponents}
  \zeta=e^{2\pi i/(4(l-2))},
  \qquad
  \mathbf n=(n_1,\ldots,n_l)=(2l-4,\,2l-5,\,2l-7,\,\ldots,\,3,\,1,\,0).
\end{equation}
For $1\leq k\leq l$ set $\kappa_k^{\rm orb}=\zeta^{-n_k}$. For $K\in\bC^\times$, write
\begin{equation}\label{eq:bmodel:orbifold-point}
  \kappa^{\rm orb}(K)=(\kappa_1^{\rm orb},\ldots,\kappa_l^{\rm orb};K).
\end{equation}
The vector $\mathbf n$ in \eqref{eq:bmodel:orbifold-exponents} is the orbifold exponent vector. It is
determined by the type-$D_l$ Dynkin marks, equivalently by the dimensions of the irreducible
representations under the McKay correspondence.
The \emph{orbifold locus} is
$\{\kappa^{\rm orb}(K):K\in\bC^\times\}$. We call each
$\kappa^{\rm orb}(K)$ an \emph{orbifold point} and abbreviate it to
$\kappa^{\rm orb}$ when $K$ is understood.
\end{definition}

\subsection{The degeneration at the orbifold locus}

Specializing $\bar C_\kappa$ at the orbifold locus inside
$S=\bP^1_\mu\times\bP^1_\lambda$ produces a nonreduced divisor containing
the vertical fibers over $\mu=\pm1$ with multiplicity two. Expanding the
ambient family along these fibers separates their multiplicities into two
rational bubbles and yields a reduced three-component nodal curve.

Write $m=l-2$, let
$D_\sigma=\{\mu=\sigma\}\subset S$ for $\sigma=\pm1$, and let
\[
  A=[\bP^1_\mu\times\{\lambda_0\}],\qquad
  B=[\{\mu_0\}\times\bP^1_\lambda]
\]
be the section and fiber classes for the projection to $\bP^1_\mu$.

\begin{lemma}[Cyclotomic collapse and the nonreduced flat limit]\label{lem:hilbert-limit}
At $\kappa^{\rm orb}(K)$ the product in \eqref{eq:bmodel:toda-curve} has the cyclotomic factorization
\begin{equation}\label{eq:bmodel:cyclotomic-collapse}
  \prod_{i=1}^l(\mu-\kappa_i^{\rm orb})(\mu-(\kappa_i^{\rm orb})^{-1})
  =(\mu^2-1)^2\big(\mu^{2m}+1\big).
\end{equation}
Hence $\bar C_\kappa$, of class $A+2lB$ in $S$, has equation at $\kappa^{\rm orb}(K)$
\[
  (\mu^2-1)^2\big(\lambda\mu^{m}-K(\mu^{2m}+1)\big)=0,
\]
and its flat limit in $S$ is the
nonreduced divisor
\begin{equation}\label{eq:bmodel:hilbert-limit}
  C^{\rm flat}_0=\Sigma^{\rm orb}+2D_++2D_-,\qquad
  \Sigma^{\rm orb}:\ \lambda=K(\mu^m+\mu^{-m}).
\end{equation}
\end{lemma}

\begin{proof}
By Definition~\ref{def:orbifold-point}, $\kappa_1^{\rm orb}=-1$ and $\kappa_l^{\rm orb}=1$. These are
the two $\iota$-fixed collisions and produce the factor $(\mu^2-1)^2$. {The remaining $l-2$ pairs of zeros of $\lambda$ exchanged by $\iota$} exhaust the roots of $\mu^{2m}=-1$. This gives the factorization
\eqref{eq:bmodel:cyclotomic-collapse}. The Toda equation then gives the
nonreduced central divisor \eqref{eq:bmodel:hilbert-limit}. The cycle-class identity
is
\[
  A+2lB=(A+2mB)+2B+2B,
\]
which is $[\Sigma^{\rm orb}]+2[D_+]+2[D_-]$.
\end{proof}

Choose a transverse arc $\kappa(t)$ through
$\kappa^{\rm orb}(K)$ {along which the two colliding pairs of zeros exchanged by $\iota$ are represented by}
$\kappa_\sigma(t)=\sigma e^{a_\sigma t+O(t^2)}$, with
$a_\sigma\neq0$, and such that $\bar C_{\kappa(t)}$ is smooth for
$t\neq0$. We require the arc to avoid all other discriminant and
pole-cancellation loci. Set
\begin{equation}\label{eq:bmodel:expanded-target}
  \mathfrak S=\operatorname{Bl}_{(D_+\cup D_-)\times\{0\}}\big(S\times\bA^1_t\big).
\end{equation}
For $t\neq0$ one has $\mathfrak S_t\cong S$, while the central fiber is the normal-crossing surface
$\mathfrak S_0=S_0\cup_{D_+}E_+\cup_{D_-}E_-$ with
\[
  E_\sigma=\bP_{D_\sigma}\big(\cO_{D_\sigma}\oplus N_{D_\sigma/S}\big)
  \cong\bP^1_\lambda\times\bP^1_{y_\sigma},
\]
the normal bundle of a vertical fiber being trivial. {The resulting family is the standard
expanded degeneration along $D_+\cup D_-$ \cite{li2001degenerationstablemorphismsrelative}.}

\begin{proposition}[Semistable degeneration]\label{prop:semistable}
{The smooth Toda fibers have strict transform
\[
  \mathfrak C=\overline{\bigcup_{t\neq0}\bar C_{\kappa(t)}}\subset\mathfrak S.
\]
This is a semistable genus-zero family with reduced nodal central fiber}
\begin{equation}\label{eq:bmodel:central-curve}
  \bar C_0=\Sigma^{\rm orb}\cup\Sigma_+\cup\Sigma_-,
\end{equation}
a tree of three rational components with two separating nodes. The main component has equation
$\lambda^{\rm orb}=K(\mu^m+\mu^{-m})$, and each bubble $\Sigma_\sigma\subset E_\sigma$ is the conic
\begin{equation}\label{eq:bmodel:bubble-conic}
  \lambda=2K\sigma^m(1-y_\sigma^2).
\end{equation}
The deck transformation induced by $\iota:\mu\mapsto\mu^{-1}$ is $y_\sigma\mapsto-y_\sigma$. The node
$p_\sigma$ lies at $(\mu,\lambda)=(\sigma,2K\sigma^m)$. With plumbing parameter
$a=a_\sigma t$, the bubble chart is
\begin{equation}\label{eq:bmodel:plumbing-chart}
  \mu=\sigma+\frac{a}{y_\sigma},
\end{equation}
so the node is $y_\sigma=0$. The finite ramification point of $\hat x=-2\nu\log\lambda$ on the bubble
is $y_\sigma=0$, hence it coincides with the node preimage on the normalization.
\end{proposition}

\begin{proof}
Work near $D_\sigma$ and in the blow-up chart
\[
  y_\sigma=\frac{a_\sigma t}{\mu-\sigma},\qquad (\mu-\sigma)y_\sigma=a_\sigma t.
\]
The two inverse roots colliding at $\sigma$ contribute
\[
  (\mu-\sigma e^{a_\sigma t})(\mu-\sigma e^{-a_\sigma t})
  =(\mu-\sigma)^2-a_\sigma^2t^2+O(t^3,(\mu-\sigma)t^2).
\]
At $t=0$, the coefficient of the $\lambda$-term after division by
$(\mu-\sigma)^2$ is
\[
  \left.
  \frac{\mu^m(\mu^2-1)^2}{(\mu-\sigma)^2}
  \right|_{\mu=\sigma}
  =
  \left.\mu^m(\mu+\sigma)^2\right|_{\mu=\sigma}
  =4\sigma^m.
\]
Likewise, \eqref{eq:bmodel:cyclotomic-collapse} gives
\[
  \left.
  \frac{(\mu^2-1)^2(\mu^{2m}+1)}{(\mu-\sigma)^2}
  \right|_{\mu=\sigma}
  =
  \left.(\mu+\sigma)^2(\mu^{2m}+1)\right|_{\mu=\sigma}
  =8.
\]
In the blow-up chart,
\[
  \frac{a_\sigma t}{\mu-\sigma}=y_\sigma,
\]
so the colliding pair divided by $(\mu-\sigma)^2$ restricts to
$1-y_\sigma^2$ on the exceptional divisor.
After removing the common exceptional factor $(\mu-\sigma)^2$
from the pullback of \eqref{eq:bmodel:toda-curve}, the strict
transform meets the exceptional divisor $\mu-\sigma=0$ in
\[
  4\sigma^m\lambda-8K(1-y_\sigma^2)=0,
\]
which is \eqref{eq:bmodel:bubble-conic}. On the component $S_0=\{y_\sigma=0\}$ of the central fiber,
the reduced equation is $\lambda=K(\mu^m+\mu^{-m})$.

At $\mu-\sigma=y_\sigma=0$, the derivative of the strict-transform
equation with respect to $\lambda$ is $4\sigma^m\neq0$. Eliminating
$\lambda$ gives a smooth total space over the local model
$t=(\mu-\sigma)y_\sigma/a_\sigma$, whose central fiber has the ordinary
node $(\mu-\sigma)y_\sigma=0$. The complementary blow-up chart gives the
same bubble with coordinate $1/y_\sigma$. Away from $\mu=\pm1$, the
choice of the arc excludes further collisions and hence further exceptional
components. The strict transform is a Cartier divisor in the smooth total
space \eqref{eq:bmodel:expanded-target}, and its defining equation does not
vanish identically on any component of the ambient central fiber; thus $t$
is a non-zero-divisor on
$\mathcal O_{\mathfrak C}$ and $\mathfrak C\to\bA^1_t$ is flat. Properness
of the expanded target and constancy of the arithmetic genus now show that
the central fiber is the semistable genus-zero tree
\eqref{eq:bmodel:central-curve}.
\end{proof}

\begin{remark}[Contraction and cycle classes]\label{rem:contraction}
Contracting $E_\sigma$ to $D_\sigma$ maps each bubble with
degree two. The resulting stable-map limit in $S$ is
$\Sigma^{\rm orb}\cup(\Sigma_+\to D_+)\cup(\Sigma_-\to D_-)$, with cycle pushforward
\[
  A+2lB\ \rightsquigarrow\ \big(A+2mB\big)+2B+2B
  \quad\Longleftrightarrow\quad [\Sigma^{\rm orb}]+2[D_+]+2[D_-].
\]
Thus the two bubbles account for the multiplicity-two
vertical components in the flat limit of Lemma~\ref{lem:hilbert-limit}.
\end{remark}

\begin{lemma}[Componentwise nodal kernel and its uniqueness]\label{lem:componentwise-nodal-kernel}
Let
\[
  \widetilde C_0=\Sigma^{\rm orb}\sqcup\Sigma_+\sqcup\Sigma_-
\]
be the normalization of $\bar C_0$. The unique admissible fundamental bidifferential, defined by its
normalization data on $\widetilde C_0$ {with residues summing to zero at the two preimages of each node}, is the direct-sum kernel
\begin{equation}\label{eq:bmodel:nodal-kernel}
  B_0=B^{\rm std}_{\Sigma^{\rm orb}}\oplus B^{\rm std}_{\Sigma_+}\oplus B^{\rm std}_{\Sigma_-},
\end{equation}
with vanishing cross-component terms, and $B^{\rm std}$ degenerates to $B_0$ as
$\kappa\to\kappa^{\rm orb}(K)$.
\end{lemma}

\begin{proof}
We use the normalization description. An admissible fundamental bidifferential on the nodal curve is a
collection of symmetric bidifferentials on pairs of components of $\widetilde C_0$ with the standard
principal part $ds_1ds_2/(s_1-s_2)^2$ along the diagonal of each component, no other poles except the
allowed node-preimage poles of the dualizing sheaf, {and the condition that the residues at the two preimages of each node sum to zero.} The
kernel $B_0$ in \eqref{eq:bmodel:nodal-kernel} has this diagonal principal
part and no node residues, so it is admissible.

Two admissible kernels differ by a symmetric global section of
$\omega_{\bar C_0}\boxtimes\omega_{\bar C_0}$. Since $\bar C_0$ is a tree of rational curves,
$H^0(\bar C_0,\omega_{\bar C_0})=0$: on each bubble a dualizing differential has at most one simple
pole, at the node preimage, and the residue theorem forces its residue to vanish; the matching
condition then kills the possible residues on the main component as well. Hence the difference of two
admissible kernels is zero, proving uniqueness.

It remains to identify the smooth limit. {Convergence on the main component is immediate
away from the node preimages.} If both points lie on the same bubble, write
$\mu_i=\sigma+a/y_i$. Then
\[
  \frac{d\mu_1d\mu_2}{(\mu_1-\mu_2)^2}
  =\frac{dy_1dy_2}{(y_1-y_2)^2}.
\]
If one point is on the main component and one on a bubble, the numerator has one factor of $a$ and
the denominator has a nonzero limit on compact subsets of the main component away from the node, so the
cross-component term is $O(a)$. If the two points are on different bubbles, the denominator has a
nonzero limit and the numerator is $O(a^2)$. Thus all cross-component terms vanish and the
same-component kernels are the standard rational kernels.
\end{proof}

{
The ramification points on the normalization of the central curve provide the boundary data used in
the graph-sum and $R$-matrix comparisons.
}

\begin{lemma}[{Ramification points on the central normalization}]\label{lem:thimble-count}
{The finite ramification points of $\hat x$ on the normalization of $\bar C_0$ are the
$2l-4$ points $\mu^{2(l-2)}=1$ on $\Sigma^{\rm orb}$ together with the point $y_\sigma=0$ on each
bubble. Under $\iota:\mu\mapsto\mu^{-1}$, the $2l-6$ nonfixed points on the main component form
$l-3$ orbits, while the two remaining main-component points and the two bubble points are fixed.
Thus the number of $\iota$-orbits is}
\[
  (l-3)+2+2=l+1.
\]
\end{lemma}

\begin{proof}
On $\Sigma^{\rm orb}$,
\[
  \lambda^{\rm orb}=K(\mu^m+\mu^{-m}),
  \qquad
  d\lambda^{\rm orb}=Km\mu^{-m-1}(\mu^{2m}-1)d\mu,
\]
so the finite ramification points are $\mu^{2m}=1$, a set of size $2l-4$. On a bubble,
$\lambda=2K\sigma^m(1-y_\sigma^2)$, so $d\lambda=-4K\sigma^m y_\sigma dy_\sigma$ and the only
finite ramification point is $y_\sigma=0$; the other ramification of the double cover is over
{$\lambda=\infty$ and is not part of the finite ramification locus considered here.}

The involution pairs the $2l-6$ nonfixed main points into
$l-3$ orbits and fixes the two main points $\mu=\pm1$; the two bubble
{ramification points are fixed as well. Thus the set of $\iota$-orbits of finite ramification
points has cardinality $(l-3)+2+2=l+1$,} agreeing with the rank of
the flat B-model algebra family in
Proposition~\ref{prop:bmodel-flat-algebra-family}.
\end{proof}

{
{The boundary limits of the canonical columns contracted with the flat unit
determine the residual gauge on $\Omega_B$} in Section~\ref{sec:anchor}.
}

\begin{remark}[The limiting second-kind form on the bubble]\label{rem:bubble-measure}
In the bubble chart \eqref{eq:bmodel:plumbing-chart} ($a=a_\sigma t$, node at
$y_\sigma=0$), the type-$D_l$
primitive differential has the rescaled limit
\begin{equation}\label{eq:bmodel:bubble-measure}
  {\phi_D\over a}\ \longrightarrow\ -\,{\sqrt{\nu}\over\sigma}\,{dy_\sigma\over y_\sigma^2},
\end{equation}
a second-kind form with double pole and zero residue at the node.
{
For the smoothing direction, Lemma~\ref{lem:bubble-beta} considers the
primary differential $2\nu\,\partial_a\log\lambda\,\phi_D$, whose limit is
$(4\nu\sqrt\nu/\sigma)\,dy_\sigma/(1-y_\sigma^2)$.
That lemma computes its normalized sectorial period and identifies the
asymptotic series with the factor associated with the eigenangle pair $(1/4,3/4)$.
}
\end{remark}

\begin{remark}[Role of rationality]\label{rem:why-componentwise-kernel}
The uniqueness in
Lemma~\ref{lem:componentwise-nodal-kernel} depends on the separating nodes
and the rationality of every component. Together these conditions give
$b_1=0$ for the dual graph and $H^0(\bar C_0,\omega_{\bar C_0})=0$, so no
third-kind correction appears in the limiting kernel. For degenerations
with positive first Betti number, period-normalized node terms may survive;
this occurs, for example, for the mirror curve of $K_{\bP^2}$ and is
tracked in the toric remodeling argument
\cite{fang2019remodelingconjecturetoriccalabiyau}.
\end{remark}


\section{The reduced B-model graph sum}
\label{sec:bgraphsum}

{
In this section, we construct the rank-$(l+1)$ local spectral curve used in the
all-genus comparison from the global Toda spectral curve introduced in
Section~\ref{sec:bmodel}. On a simply connected
working domain, the sign local system, centered odd coordinate, and Prym
bidifferential define the $\bZ_2$-equivariant topological recursion and its dilaton,
leaf, and edge data. Regularized stationary phase and the Prym double-Laplace
identity determine the symplectic B-model $R$-matrix, while the DOSS correspondence
packages the stable correlators into a decorated graph sum. We also record the
limits of these data on the semistable orbifold fiber for the later gauge fixing
and A/B graph-sum comparison.}

{We now fix the normalizations underlying this reduction. On the smooth $\mu$-double cover, write}
\[
  \iota(\mu)=\mu^{-1},
  \qquad
  w=\mu+\mu^{-1}.
\]
The invariant Frobenius algebra consists of functions of $w$. Let $\phi_D$ be the type-$D_l$
primitive differential normalized by
\[
  \iota^*\phi_D=-\phi_D,
  \qquad
  \phi_D^2=\nu\left(\frac{d\mu}{\mu}\right)^2.
\]
The period forms are $f(w)\phi_D$. For
\[
  \hat x=-2\nu\log\lambda,
\]
the threefold-normalized pairing is represented by {the sum over critical orbits of the lifted residues of}
$\phi_D^2/d\hat x$. {
We apply the local $\bZ_2$-equivariant topological recursion of
\cite[Sec.~8]{giacchetto2024newspinhurwitztheory} with trivial $d\hat x$-character and sign
$\widetilde y$- and bidifferential characters. Its reduction to one representative of each critical
orbit, followed by the standard normalization of the bidifferential, gives the recursion used below.
We denote its stable forms by $\omega^\iota_{g,n}$. The DOSS theorem
\cite{duninbarkowski2012identificationgiventalformulaspectral}, in the form recalled in
\cite[Sec.~2.3 and Sec.~9.1]{giacchetto2024newspinhurwitztheory}, expresses them as a Givental graph
sum.}

\subsection{The B-model working open set}

We work in a simply connected domain on which the critical points, logarithm branches, and thimbles vary
without monodromy.

Let
\begin{equation}\label{eq:bgraphsum:critical-polynomial}
\begin{aligned}
  P(w;\kappa)=\prod_{i=1}^l(w-c_i),\qquad c_i=\kappa_i+\kappa_i^{-1},
  \\
  N(w;\kappa)=(w^2-4)P'(w;\kappa)-2wP(w;\kappa).
\end{aligned}
\end{equation}
The roots of $N$ in \eqref{eq:bgraphsum:critical-polynomial} are the critical
points of the quotient function $\lambda(w)$, because
\[
  \lambda(w)=K\frac{P(w;\kappa)}{w^2-4},\qquad K=\kappa_{l+1}.
\]

\begin{definition}[B-model working open set]\label{def:bmodel-open}
Fix a nonzero scale $K_0$ and a small analytic neighborhood $\mathcal U_B$ of
\[
  (\kappa_1^{\rm orb},\ldots,\kappa_l^{\rm orb};K_0)
\]
in the B-model parameter space. We remove the following bad locus from $\mathcal U_B$:
\begin{enumerate}
\item the discriminant of the type-$D_l$ curve and the discriminant of $N(w;\kappa)$, ensuring that
  the roots of $N$ are simple and the quotient superpotential is Morse at each critical point;
\item the pole-cancellation locus through the orbifold point, whose local branches are the loci on which {a pair of zeros
  $\{\kappa_i,\kappa_i^{-1}\}$ of $\lambda$ exchanged by $\iota$} cancels a pole at $w=\pm2$; these are
  the local branches of Definition~\ref{def:pole-cancellation-locus}, and include
  $\{\kappa_1=-1\}$ and $\{\kappa_l=1\}$;
\item the Maxwell locus on the chosen logarithmic cover, on which two quotient critical values of
  $W=-\hat x=2\nu\log\lambda$ coincide;
\item the Stokes walls for the chosen sector of $e^{W/z}$, so that the reduced Lefschetz thimbles are
  locally constant;
\item the cuts on which the chosen branches of $\log\lambda$, $\log\mu$, or
  $\log_{\rm orb}\kappa_i$ (the branch fixed in Section~\ref{sec:frobenius}) cease to be single-valued.
\end{enumerate}
After shrinking $\mathcal U_B$ if necessary, choose a connected component
$\mathcal C_B$ of the complement of the bad locus whose closure contains the orbifold point.
A \emph{B-model working open set} is a simply connected domain
\begin{equation}\label{eq:bgraphsum:working-open}
  \Omega_B\subset\mathcal C_B
\end{equation}
whose closure also contains the orbifold point.
On $\Omega_B$ we fix once and for all a holomorphic ordering
$w_1,\ldots,w_{l+1}$ of the roots of $N(w;\kappa)$, local branches $\mu_\alpha(w)$ of
$\mu+\mu^{-1}=w$, critical values $u_\alpha=W(w_\alpha)$, local coordinates $s_\alpha$, reduced
Lefschetz thimbles $\Gamma_\alpha$, and all logarithm branches.
\end{definition}

\subsection{{The GKL-reduced local spectral curve}}

{
The ordered zeros of $N(w;\kappa)$ are indexed by
\begin{equation}\label{eq:bgraphsum:label-set}
  I=\{1,\ldots,l+1\},
  \qquad N(w_\alpha(\kappa);\kappa)=0\quad(\alpha\in I).
\end{equation}
The ordering fixed in Definition~\ref{def:bmodel-open} makes the roots $w_\alpha$ holomorphic on
$\Omega_B$. Put $Q_\kappa=\bar C_\kappa/\iota$ and
$Q_\kappa^\circ=Q_\kappa\setminus\{w=2,w=-2\}$, and let
$\pi:\bar C_\kappa\to Q_\kappa$ be the quotient map. Choose a small disk
$U_\alpha\subset Q_\kappa^\circ$ about $w_\alpha$ and one of its lifts
$U_{\alpha,+}$, with local branch $\mu_\alpha(w)$; set
$U_{\alpha,-}=\iota(U_{\alpha,+})$. We denote the lifted critical points by
$p_{\alpha,\pm}$ and use $p_\alpha=p_{\alpha,+}$ on the chosen representative disk.

The local functions are
\[
  \hat x(w)=-2\nu\log\lambda(w)
\]
and the centered logarithm
\begin{equation}\label{eq:bgraphsum:centered-y}
  \widetilde y_\alpha(w)={\sqrt{-1}\over 2\nu}\,
  \bigl(\log\mu_\alpha(w)-\log\mu_\alpha(w_\alpha)\bigr).
\end{equation}
Fix the choice of $\sqrt{-1}$. Centering removes the additive ambiguity in
$\log(\mu^{-1})=-\log\mu+2\pi i n$ and makes the extension to the other lift exactly odd:
\[
  \widetilde y_{\alpha,-}(\iota q)=-\widetilde y_{\alpha,+}(q).
\]
The two-point function is the Prym kernel
\begin{equation}\label{eq:bgraphsum:prym-kernel}
  B^-:=B^{\rm std}-(\operatorname{id}\times\iota)^*B^{\rm std}.
\end{equation}
Since $B^{\rm std}(\mu_1,\mu_2)=d\mu_1\,d\mu_2/(\mu_1-\mu_2)^2$, one has
\begin{equation}\label{eq:bgraphsum:prym-kernel-explicit}
  B^-=
  \left(
  \frac1{(\mu_1-\mu_2)^2}
  +\frac1{(\mu_1\mu_2-1)^2}
  \right)d\mu_1\,d\mu_2,
\end{equation}
and $\iota_1^*B^-=\iota_2^*B^-=-B^-$.

Set
\[
  B_{\rm av}:=\frac12B^-.
\]
Choose the lifted Morse coordinates equivariantly, so that
$s_{\alpha,-}(\iota q)=s_{\alpha,+}(q)$. On the disjoint union of the $2(l+1)$ lifted disks, the data
$(\hat x,\widetilde y,B_{\rm av})$ form a local $\bZ_2$-spectral curve in the sense of
\cite[Defs.~8.5 and~8.10]{giacchetto2024newspinhurwitztheory}. Indeed, the action on these disks is
free, and the characters of $d\hat x$, $\widetilde y$, and $B_{\rm av}$ are respectively
\[
  \chi=1,\qquad \upsilon=\operatorname{sgn},\qquad
  \beta=\operatorname{sgn}.
\]
Thus $\chi\upsilon=\beta$, and $B_{\rm av}$ is the sign-character average of
\cite[Eq.~(8.2)]{giacchetto2024newspinhurwitztheory}. The lifted disks avoid the fixed points
$\mu=\pm1$ of the global involution and therefore satisfy the local definition.

Following \cite[Def.~8.7]{giacchetto2024newspinhurwitztheory}, choose the disks
$U_{\alpha,+}$ as orbit representatives and identify them with $U_\alpha$ by $\pi$.
{We use $B^-=2B_{\rm av}$, with the standard principal part along the diagonal on each
representative disk, and make the following definition.}
\begin{definition}[Reduced local B-model curve]\label{def:reduced-local-curve}
For $\kappa\in\Omega_B$, the \emph{reduced local B-model curve} is
\begin{equation}\label{eq:bgraphsum:reduced-curve}
  \cS^\iota_\kappa
  =\bigl(\{U_{\alpha,+}\}_{\alpha\in I},\hat x,
  \{\widetilde y_{\alpha,+}\}_{\alpha\in I},B^-\bigr).
\end{equation}
At $p_\alpha$ choose the Morse coordinate $s_\alpha$ by
\begin{equation}\label{eq:bgraphsum:morse-coordinate}
  \hat x=-u_\alpha+s_\alpha^2,
  \qquad u_\alpha=W(w_\alpha),
\end{equation}
and write $q\mapsto\bar q$ for the local conjugation $s_\alpha\mapsto-s_\alpha$.
This local conjugation is distinct from the global involution $\iota$.
\end{definition}

Let $\omega^{\rm av}_{g,n}$ be the reduced correlator with $B_{\rm av}$ and let
$\omega^{G,{\rm av}}_{g,n}$ be the full equivariant correlator. Proposition~8.8 and
Remark~2.2 of \cite{giacchetto2024newspinhurwitztheory}, combined exactly as in their Lemma~9.1,
give
\begin{equation}\label{eq:bgraphsum:averaged-rescaling}
  \omega^\iota_{g,n}
  =2^{3g-3+2n}\omega^{\rm av}_{g,n}
  =2^{g+n-1}\omega^{G,{\rm av}}_{g,n}.
\end{equation}
The first equality is the rescaling $B_{\rm av}\mapsto B^-$ on the representative disks; the second
also uses the orbit factor $2^{2g-2+n}$ of \cite[Prop.~8.8]{giacchetto2024newspinhurwitztheory}.
By \cite[Lem.~8.6]{giacchetto2024newspinhurwitztheory}, the extension of the reduced correlators to
the other lifted disks is anti-invariant in every variable.

Equivalently, the reduced forms are sections of the sign local system $\cL_{\rm sgn}$ on
$Q_\kappa^\circ$, or of the sign line on the quotient stack
$[\bar C_\kappa/\langle\iota\rangle]$. In this language
\[
  (\iota_j)^*(\pi^{\times n})^*\omega^\iota_{g,n}
  =-(\pi^{\times n})^*\omega^\iota_{g,n},
  \qquad 1\leq j\leq n,
\]
{On $Q_\kappa^\circ\times Q_\kappa^\circ$, the kernel $B^-$ descends as a section of
\[
  \bigl((K_{Q_\kappa^\circ}\otimes\cL_{\rm sgn})
  \boxtimes(K_{Q_\kappa^\circ}\otimes\cL_{\rm sgn})\bigr)(2\Delta),
\]
where $\Delta$ is the diagonal of this open product.}
{The quotient corresponding to the trivial character instead uses}
\[
  B^+=B^{\rm std}+(\operatorname{id}\times\iota)^*B^{\rm std}
  =\pi^*\left(\frac{dw_1\,dw_2}{(w_1-w_2)^2}\right),
\]
whereas the period forms $f(w)\phi_D$ lie in the $B^-$ sector.

The reduced thimbles live on the logarithmic cover $\widetilde Q_{\kappa,\lambda}$ of
$Q_\kappa^\circ$ on which $W=2\nu\log\lambda$ is single-valued. Equivalently,
\[
  \Gamma_\alpha\in
  H^{\rm rd}_1\bigl(\widetilde Q_{\kappa,\lambda};\pi^*\cL_{\rm sgn}^\vee\bigr),
  \qquad \operatorname{Re}(W/z)\ll0,
\]
for the exponential connection $d-dW/z$. Replacing an orbit representative changes the sign frame.
Writing $D_{\rm sgn}=\operatorname{diag}(\varepsilon_\alpha)$, with
$\varepsilon_\alpha\in\{\pm1\}$, this sends
\[
  \hat R^B\longmapsto D_{\rm sgn}\hat R^BD_{\rm sgn},\qquad
  \check B_{k,l}\longmapsto D_{\rm sgn}\check B_{k,l}D_{\rm sgn},
\]
and changes each one-point datum by the corresponding sign. Lemma~8.6 of
\cite{giacchetto2024newspinhurwitztheory} then makes every contracted graph weight independent of
the representatives.}

{To compare the linear coefficient of the local conjugate coordinate with the
threefold Frobenius normalization, put}
\[
  a_\alpha=\left.\frac{d\log\mu_\alpha}{ds_\alpha}\right|_{s_\alpha=0}.
\]
The quotient primitive idempotent is represented by $1$ on both lifted critical points
$p_{\alpha,+}$ and $p_{\alpha,-}=\iota p_{\alpha,+}$. {The normalization by sums of residues over each orbit in}
Lemma~\ref{lem:quotient-residue-normalization} gives
\[
  \frac1{\Delta^\cX_\alpha}
  ={1\over(2\nu)^3}
  \sum_{p_{\alpha,\pm}}\Res_{p_{\alpha,\pm}}
  \frac{\phi_D^2}{d\hat x}
  ={a_\alpha^2\over8\nu^2}.
\]
Since
$\widetilde y_\alpha={\sqrt{-1}\over2\nu}(\log\mu_\alpha-\log\mu_\alpha(w_\alpha))$, we also have
\[
  \widetilde y_\alpha(s_\alpha)-\widetilde y_\alpha(-s_\alpha)
  =2\check h^\alpha_1s_\alpha+O(s_\alpha^3),
  \qquad \check h^\alpha_1={\sqrt{-1}\over2\nu}\,a_\alpha.
\]
With the same choice of square roots as in the normalized canonical frame, these two calculations give
\begin{equation}\label{eq:bgraphsum:gaussian-normalization}
  \left(\frac{\check h^\alpha_1}{\sqrt{-2}}\right)^2
  ={1\over\Delta^\cX_\alpha},
\end{equation}
{which identifies the squared vertex factor with the corresponding idempotent norm.}

\subsection{{The $\bZ_2$-equivariant topological recursion}}

{
The $\bZ_2$-equivariant topological recursion of
\cite[Eqs.~(8.12)--(8.13)]{giacchetto2024newspinhurwitztheory}, {after the reduction to one representative disk per orbit in} \cite[Def.~8.7]{giacchetto2024newspinhurwitztheory}, is the ordinary local topological recursion on
$\cS^\iota_\kappa$.
Its unstable inputs in our normalization are
\begin{equation}\label{eq:bgraphsum:unstable-inputs}
  \omega^\iota_{0,1}|_{U_{\alpha,+}}=\widetilde y_\alpha\,d\hat x,
  \qquad
  \omega^\iota_{0,2}=B^-,
\end{equation}
and its local kernel is
\begin{equation}\label{eq:bgraphsum:recursion-kernel}
  K_\alpha(q_0,q)
  =\frac{\frac12\displaystyle\int_{\bar q}^{q}B^-(q_0,\cdot)}
  {\bigl(\widetilde y_\alpha(q)-\widetilde y_\alpha(\bar q)\bigr)d\hat x(q)}.
\end{equation}
For $2g-2+n>0$, the forms $\omega^\iota_{g,n}$ are defined by
\cite[Eqs.~(8.12)--(8.13)]{giacchetto2024newspinhurwitztheory}, with the residue sum restricted to
$\alpha\in I$ as prescribed by their Definition~8.7. The orientation in
\eqref{eq:bgraphsum:recursion-kernel} fixes the sign convention used in the graph sum below.

Lemma~8.6 of \cite{giacchetto2024newspinhurwitztheory} extends the correlators anti-invariantly to
the other lifts, while the second equality of \eqref{eq:bgraphsum:averaged-rescaling} identifies them
{with the formulation on both lifted disks in each orbit.}}

\subsection{The reduced $R$-matrix and the local weights}

The odd part of $\widetilde y_\alpha$ supplies the dilaton coefficients:
\begin{equation}\label{eq:bgraphsum:dilaton-coefficients}
  \widetilde y_\alpha(s_\alpha)-\widetilde y_\alpha(-s_\alpha)
  =2\sum_{k\geq1}\frac{2^{k-1}}{(2k-1)!!}\check h^\alpha_k s_\alpha^{2k-1}.
\end{equation}
The linear coefficient $\check h^\alpha_1$ gives the vertex factor
$\check h^\alpha_1/\sqrt{-2}$, with the normalization
\eqref{eq:bgraphsum:gaussian-normalization}.

The \emph{bare second-kind forms} are defined from the Prym kernel by
\begin{equation}\label{eq:bgraphsum:bare-second-kind}
  \vartheta^k_\alpha(q)=
  -\frac{(2k-1)!!}{2^k}
  \Res_{q'\to p_\alpha}B^-(q,q')s_\alpha(q')^{-2k-1},
  \qquad k\geq0.
\end{equation}
On $U_\alpha$, this form has a pole of order $2k+2$ at $p_\alpha$ and zero residue; its pullback
upstairs has the corresponding image pole at $\iota p_\alpha$. The $\mu$-cover is rational, so no
$A$-period normalization is needed for {these forms with values in the sign local system.} The ordinary graph leaf is
\begin{equation}\label{eq:bgraphsum:ordinary-leaf}
  \theta^k_\alpha=\frac1{\sqrt{-2}}\,\vartheta^k_\alpha.
\end{equation}

The primary form $\vartheta^0_\beta$ has a double pole at $p_\beta$, so its thimble Laplace transform
is understood as a regularized formal stationary-phase expansion.
Its pullback to the rational $\mu$-cover has only the two corresponding second-kind poles
$p_{\beta,+}$ and $p_{\beta,-}$, with zero residue at both, and is therefore exact. Choose a rational
primitive $\xi^0_\beta$ and pass to its anti-invariant part
\[
  \xi_\beta={1\over2}\bigl(\xi^0_\beta-\iota^*\xi^0_\beta\bigr).
\]
Then $d\xi_\beta$ is the pullback of $\vartheta^0_\beta$, and anti-invariance removes the constant
ambiguity. Define
\begin{equation}\label{eq:bgraphsum:regularized-primary-transform}
  \int_{\Gamma_\alpha}^{\rm reg}e^{-\hat x/z}\vartheta^0_\beta
  :=\frac1z\int_{\Gamma_\alpha}e^{-\hat x/z}\xi_\beta\,d\hat x .
\end{equation}

\begin{definition}[The reduced B-model $R$-matrix]\label{def:b-rmatrix}
The \emph{B-model $R$-matrix} $\hat R^B(z)$ is the formal stationary-phase series defined by
\begin{equation}\label{eq:bgraphsum:b-rmatrix}
  \hat R^B_{\beta\alpha}(z)
  \sim
  \frac{\sqrt z\,e^{-u_\alpha/z}}{2\sqrt\pi}
  \int_{\Gamma_\alpha}^{\rm reg}e^{-\hat x/z}\vartheta^0_\beta,
  \qquad z\to0.
\end{equation}
{Column $\alpha$ corresponds to the reduced Lefschetz thimble $\Gamma_\alpha$ and its critical
value $u_\alpha$, while row $\beta$ corresponds to the primary second-kind form $\vartheta^0_\beta$.
The regularization \eqref{eq:bgraphsum:regularized-primary-transform} and normalization
\eqref{eq:bgraphsum:b-rmatrix} give $\hat R^B(0)=I$.}
\end{definition}

The edge coefficients come from the double-Laplace transform of the Prym kernel. Near
$(p_\alpha,p_\beta)$, write
\begin{equation}\label{eq:bgraphsum:kernel-expansion}
  B^-(q_1,q_2)=
  \left(\frac{\delta_{\alpha\beta}}{(s_\alpha-s_\beta)^2}
  +\sum_{k,l\geq0}B^{\alpha\beta}_{k,l}s_\alpha^ks_\beta^l\right)
  ds_\alpha ds_\beta,
\end{equation}
and set
\begin{equation}\label{eq:bgraphsum:edge-coefficients}
  \check B^{\alpha\beta}_{k,l}
  =\frac{(2k-1)!!(2l-1)!!}{2^{k+l+1}}B^{\alpha\beta}_{2k,2l}.
\end{equation}

{
\begin{proposition}[Prym double-Laplace factorization]\label{prop:prym-double-laplace}
The edge coefficients \eqref{eq:bgraphsum:edge-coefficients} of $B^-$ and the
stationary-phase matrix of Definition~\ref{def:b-rmatrix} satisfy
\[
  \sum_{k,l\geq0}\check B^{\alpha\beta}_{k,l}z_1^kz_2^l
  =\frac{\delta_{\alpha\beta}-\sum_{\gamma\in I}
  \hat R^B_{\gamma\alpha}(z_1)\hat R^B_{\gamma\beta}(z_2)}{z_1+z_2}.
\]
In particular,
\[
  \hat R^B(0)=I,
  \qquad
  \hat R^B(z)\hat R^B(-z)^{\mathsf T}=I.
\]
\end{proposition}

\begin{proof}
Definition~\ref{def:reduced-local-curve} is the GKL reduced curve of
\cite[Def.~8.7]{giacchetto2024newspinhurwitztheory}, with its bidifferential rescaled from
{$B_{\rm av}$ to $B^-$, which has the standard principal part along the diagonal on each
representative disk.} It is therefore an ordinary local spectral curve
on the $l+1$ representative disks. The critical values are pairwise distinct on $\Omega_B$, and
$d\hat x=-2\nu\,d\lambda/\lambda$ and $B^-$ are meromorphic on the compact $\mu$-cover.

The argument of \cite[Sec.~9.1]{giacchetto2024newspinhurwitztheory} now applies verbatim: after the
GKL reduction, the Eynard--DOSS graph calculation is local, while the edge/$R$-matrix identity only
uses the compact curve and the meromorphic differential $d\hat x$. In the local coordinate
$\zeta_\alpha=\sqrt2s_\alpha$, {the differential obtained by evaluating the kernel at a critical point in}
\cite[Eq.~(2.23)]{giacchetto2024newspinhurwitztheory} is related to our primary second-kind form by
\[
  \vartheta^0_\beta
  =-\sqrt2\left.
  {B^-(\,\cdot\,,q')\over d\zeta_\beta(q')}
  \right|_{q'=p_\beta}.
\]
Consequently the Gaussian transform denoted $R^{-1}$ in
\cite[Eq.~(2.25)]{giacchetto2024newspinhurwitztheory} agrees with $\hat R^B$ after removal of
the critical exponential. The regularization
\eqref{eq:bgraphsum:regularized-primary-transform} is the integration-by-parts form of the same
formal Gaussian transform. {After the rescaling $t=s_\alpha/\sqrt z$, with a compatible
square-root branch and orientation, integration along the chosen Lefschetz thimble gives the same
formal stationary-phase series as integration along the real $t$-axis with weight $e^{-t^2}$.}

With this dictionary, the local Cauchy identity
\cite[Eq.~(3-66)]{duninbarkowski2012identificationgiventalformulaspectral} becomes
\[
  \sum_{k,l\geq0}\check B^{\alpha\beta}_{k,l}z_1^kz_2^l
  =\frac{\delta_{\alpha\beta}-\sum_{\gamma\in I}
  \hat R^B_{\gamma\alpha}(z_1)\hat R^B_{\gamma\beta}(z_2)}{z_1+z_2};
\]
{the factors in \eqref{eq:bgraphsum:edge-coefficients} arise from the Gaussian moments
together with the prefactor of the double-Laplace transform.}
{Thus the identity is indexed by $\alpha,\beta\in I$.}

Near $p_\alpha$ one has
$\vartheta^0_\alpha=-ds_\alpha/s_\alpha^2+O(1)ds_\alpha$. Hence
\[
  \int_{\Gamma_\alpha}^{\rm reg}e^{-\hat x/z}\vartheta^0_\alpha
  =2\sqrt\pi\,z^{-1/2}e^{u_\alpha/z}(1+O(z)),
\]
whereas for $\beta\ne\alpha$ the same regularized integral is
$O(z^{1/2}e^{u_\alpha/z})$. Definition~\ref{def:b-rmatrix} therefore gives
$\hat R^B_{\beta\alpha}(z)=\delta_{\alpha\beta}+O(z)$.

Finally, the factorization in matrix notation reads
\[
  (z_1+z_2)\sum_{k,l\geq0}\check B_{k,l}z_1^kz_2^l
  =I-(\hat R^B(z_1))^{\mathsf T}\hat R^B(z_2).
\]
Setting $(z_1,z_2)=(z,-z)$ gives
$(\hat R^B(z))^{\mathsf T}\hat R^B(-z)=I$. Since $\hat R^B(0)=I$, the formal matrix is invertible;
transposing and taking the unique inverse gives the equivalent convention
$\hat R^B(z)\hat R^B(-z)^{\mathsf T}=I$.
\end{proof}}

{The relation obtained by contraction with the flat unit for} $\check h^\alpha_k$ will be proved in
Section~\ref{sec:anchor}.

\subsection{The DOSS graph sum}

\begin{theorem}[{DOSS graph sum for the $\bZ_2$-equivariant topological recursion}]\label{thm:doss}
For $\kappa\in\Omega_B$, let $\mathsf G^B_{g,n}(I)$ be the set of isomorphism classes of connected
{stable B-model decorated graphs $\vec{\Gamma}$ equipped with a vertex coloring
$\alpha:V(\vec{\Gamma})\to I$ and with the following data and local weights.}
\begin{enumerate}
\item \emph{Ordinary leaves.}
The ordinary leaves form the numbered set
$L^o(\vec{\Gamma})=\{1,\ldots,n\}$. An ordinary leaf $j$ is attached to a vertex $v(j)$, inherits the
{color $\alpha(j)=\alpha(v(j))\in I$,} and has a height $k(j)\in\bZ_{\geq0}$. Its local weight in the
$j$-th external variable is
\[
  \mathsf L_B^o(j)=\theta^{k(j)}_{\alpha(j)}(q_j).
\]

\item \emph{Dilaton leaves.}
The dilaton leaves form an arbitrary finite unnumbered set $L^1(\vec{\Gamma})$. A dilaton leaf
$\ell$ is attached to a vertex $v(\ell)$ and {inherits the color
$\alpha(\ell)=\alpha(v(\ell))\in I$;} it has a height $k(\ell)\in\bZ_{\geq2}$. Its local weight is
\[
  \mathsf L_B^1(\ell)
  =-\frac{\check h^{\alpha(\ell)}_{k(\ell)}}{\sqrt{-2}}.
\]

\item \emph{Edges.}
Each edge $e\in E(\vec{\Gamma})$ is an unordered pair $e=\{e',e''\}$ of half-edges. If $e'$ is incident
to $v'$ and $e''$ to $v''$, then
\[
  \alpha(e')=\alpha(v'),\qquad
  \alpha(e'')=\alpha(v''),\qquad
  k(e'),k(e'')\in\bZ_{\geq0},
\]
and the local edge weight is
\[
  \mathsf E_B(e)
  =\check B^{\alpha(e'),\alpha(e'')}_{k(e'),k(e'')}.
\]

\item \emph{Vertices.}
{Each vertex $v\in V(\vec{\Gamma})$ has a genus $g(v)\in\bZ_{\geq0}$ and color
$\alpha(v)\in I$.}
Let $H(v)$ be the set of all flags incident to $v$: edge half-edges, ordinary leaves, and dilaton
leaves. Set $\val(v)=|H(v)|$. The stability condition and the local vertex weight are respectively
\[
  2g(v)-2+\val(v)>0,
  \qquad
  \mathsf V_B(v)
  =
  \left(\frac{\check h^{\alpha(v)}_1}{\sqrt{-2}}\right)^{2-2g(v)-\val(v)}
  \left\langle\prod_{h\in H(v)}\tau_{k(h)}\right\rangle_{g(v)}.
\]
\end{enumerate}
The total genus condition is
\[
  g(\vec{\Gamma})
  =\sum_{v\in V(\vec{\Gamma})}g(v)+|E(\vec{\Gamma})|-|V(\vec{\Gamma})|+1
  =g.
\]

{For $\kappa\in\Omega_B$, $n>0$, and $2g-2+n>0$, the
$\bZ_2$-equivariant topological recursion gives}
\begin{equation}\label{eq:bgraphsum:doss-expansion}
  \omega^\iota_{g,n}=
  \sum_{\vec{\Gamma}\in\mathsf G^B_{g,n}(I)}
  \frac{\mathrm{wt}_B(\vec{\Gamma})}{|\Aut(\vec{\Gamma})|}.
\end{equation}
For $\vec{\Gamma}\in\mathsf G^B_{g,n}(I)$, the B-model graph weight is the product of the local
weights above, together with the global sign $(-1)^{g(\vec{\Gamma})-1}$:
\begin{equation}\label{eq:bgraphsum:doss-weight}
\begin{aligned}
  \mathrm{wt}_B(\vec{\Gamma})
  ={}&(-1)^{g(\vec{\Gamma})-1}
  \prod_{v\in V(\vec{\Gamma})}\mathsf V_B(v)
  \prod_{e\in E(\vec{\Gamma})}\mathsf E_B(e) \\
  &\quad\cdot
  \prod_{j\in L^o(\vec{\Gamma})}\mathsf L_B^o(j)
  \prod_{\ell\in L^1(\vec{\Gamma})}\mathsf L_B^1(\ell).
\end{aligned}
\end{equation}
The group $\Aut(\vec{\Gamma})$ consists
of automorphisms of the decorated graph that fix the numbered ordinary leaves and preserve the genus,
{vertex coloring, and height decorations;} it may permute indistinguishable dilaton leaves.
\end{theorem}

{
\begin{proof}
By \cite[Def.~8.7]{giacchetto2024newspinhurwitztheory},
$\cS^\iota_\kappa$ is an ordinary local spectral curve on the representative disks. The application
of the Eynard--DOSS correspondence after this reduction is exactly the argument of
\cite[Sec.~9.1]{giacchetto2024newspinhurwitztheory}; see also their Theorem~2.11 and
\cite[Thm.~3.7]{duninbarkowski2012identificationgiventalformulaspectral}. The hypotheses follow from
Definition~\ref{def:bmodel-open}: the critical points are nondegenerate, their values $u_\alpha$ are
pairwise distinct, the odd part of $\widetilde y_\alpha$ has nonzero linear term, and $B^-$ has the
standard diagonal singularity on each representative disk. Translating the GKL/DOSS formula through
\eqref{eq:bgraphsum:dilaton-coefficients}--\eqref{eq:bgraphsum:edge-coefficients} gives the
displayed weights. The factors $1/\sqrt{-2}$ and the global sign are the toric-remodeling convention
of \cite[Sec.~6.3]{fang2019remodelingconjecturetoriccalabiyau}; compare also
\cite{fang2016eynardorantinrecursionequivariantmirror}.
\end{proof}}

\subsection{Semistable boundary data at the orbifold point}

By Lemma~\ref{lem:componentwise-nodal-kernel}, $B^{\rm std}$ converges componentwise on compact
subsets of the normalizations away from the nodes. The involution extends to each component of the
semistable normalization, and pullback by $\operatorname{id}\times\iota$ commutes with the limit.
Therefore
\begin{equation}\label{eq:bgraphsum:boundary-kernel}
  B^-_0=B^-_{\Sigma^{\rm orb}}\oplus B^-_{\Sigma_+}\oplus B^-_{\Sigma_-},
  \qquad
  B^-_C=B^{\rm std}_C-(\operatorname{id}\times\iota_C)^*B^{\rm std}_C.
\end{equation}

On the main component one has
\[
  \lambda^{\rm orb}=K(\mu^m+\mu^{-m}),\qquad m=l-2.
\]
Near the bubble over $\mu=\sigma$, the plumbing chart is
\begin{equation}\label{eq:bgraphsum:boundary-bubble-chart}
  \mu=\sigma+\frac{a}{y_\sigma},
  \qquad
  \left.\lambda\right|_{a=0}=2K\sigma^m(1-y_\sigma^2),
\end{equation}
where the second equality is the equation on the central bubble.
{Let $\widetilde y_{\sigma,a}$ be the centered coordinate
\eqref{eq:bgraphsum:centered-y} with its center continued along either
critical point near $\mu=\sigma$, and extend it oddly to the partner lift.
On compact subsets of the bubble away from the node,
\[
  \log\mu-\log\sigma=O(a),
  \qquad
  \log\mu_{\rm crit}-\log\sigma=O(a^{1/2}).
\]
The second estimate follows from the critical-point calculation in
Lemma~\ref{lem:coefficientwise-parity-limit}. On the chosen square-root cover,
\begin{equation}\label{eq:bgraphsum:bubble-y-degeneration}
  \widetilde y_{\sigma,a}-\iota^*\widetilde y_{\sigma,a}
  =2\widetilde y_{\sigma,a}=O(a^{1/2}).
\end{equation}
The fiberwise differential $d\widetilde y_{\sigma,a}$ is $O(a)$ on these compact subsets.}
{Thus the anti-invariant part of the centered coordinate degenerates locally uniformly on compact
subsets of the central bubble away from the node.

Remark~\ref{rem:bubble-measure} records the rescaled second-kind form on the bubble.
Lemma~\ref{lem:bubble-beta} instead evaluates the normalized period of the
primary differential for the smoothing direction. The ramification points in
Lemma~\ref{lem:thimble-count} and the unit-contraction limits of
Section~\ref{sec:anchor} determine the normalization on $\Omega_B$.
}


\section{The Frobenius-manifold isomorphism near the orbifold point}
\label{sec:frobenius}

{In this section, we identify the Frobenius manifolds that enter the remodeling comparison
and establish the flat algebra extension through the pole-cancellation locus.}

\subsection{The domain $\Omega_{\rm Frob}$ and the invariant Jacobian}

We refine the neighborhood notation used in Definition~\ref{def:bmodel-open}. We choose a local
branch chart
$\widetilde{\mathcal U}_B$ over $\mathcal U_B$, containing the orbifold point
$\kappa^{\rm orb}(K_0)$, on which the logarithm branches and the ramified base coordinates used
below are fixed.
Set
\begin{equation}\label{eq:frob:critical-polynomial}
  P(w;\kappa)=\prod_{i=1}^l(w-c_i),
  \qquad
  N(w;\kappa)=(w^2-4)P'(w;\kappa)-2wP(w;\kappa).
\end{equation}

\begin{definition}[Pole-cancellation locus]\label{def:pole-cancellation-locus}
The \emph{pole-cancellation locus} in $\widetilde{\mathcal U}_B$ is
\[
  \mathcal Z_{\rm pc}
  =
  \left\{
    \kappa\in\widetilde{\mathcal U}_B:
    P(2;\kappa)P(-2;\kappa)=0
  \right\}.
\]
Equivalently, it is the locus on which
$c_i=\kappa_i+\kappa_i^{-1}$ equals $2$ or $-2$ for some $i$, or, equivalently,
$\kappa_i=1$ or $\kappa_i=-1$. Along this locus the factor $w-c_i$ cancels one of the poles of
\[
  \lambda(w)=\kappa_{l+1}{P(w;\kappa)\over w^2-4}.
\]
In the chosen ramified chart, the relevant local branches over $w=2$ and $w=-2$ are defined by
$\varepsilon_+=0$ and $\varepsilon_-=0$, respectively.
\end{definition}

Let $\Omega_{\rm Frob}$ be a simply connected open subset of
$\widetilde{\mathcal U}_B$ whose closure contains $\kappa^{\rm orb}(K_0)$, obtained by removing the
discriminant of the compactified type-$D_l$ curve, the discriminant of $N(w;\kappa)$, and
$\mathcal Z_{\rm pc}$. We choose $\Omega_{\rm Frob}$ to contain the B-model working
open set $\Omega_B$ of Definition~\ref{def:bmodel-open}:
\[
  \Omega_B\subset\Omega_{\rm Frob}.
\]
Thus the type-$D_l$ curve is smooth on $\Omega_{\rm Frob}$, and $N(w;\kappa)$ has $l+1$ simple
roots.

On $\Omega_{\rm Frob}$ we use the same branches of the logarithms as in Section~\ref{sec:bgraphsum}.
We keep the notation
\[
  \widehat P(w;\kappa)=\kappa_{l+1}P(w;\kappa)
\]
for the numerator of the quotient function $\lambda(w)=\widehat P(w;\kappa)/(w^2-4)$.
The symbol $\log_{\rm orb}$ denotes the logarithm branch fixed on
$\widetilde{\mathcal U}_B$ and normalized at the orbifold value. In particular,
\begin{equation}\label{eq:frob:orbifold-log-coordinates}
  \btau_i(\kappa)=\log_{\rm orb}\kappa_i,
  \qquad
  \btau_i^{\rm orb}=\log_{\rm orb}\kappa_i^{\rm orb}=-n_i\log\zeta,
\end{equation}
where
\[
  \zeta=e^{2\pi i/(4(l-2))},
  \qquad
  \mathbf n=(n_1,\ldots,n_l)=(2l-4,2l-5,2l-7,\ldots,3,1,0).
\]
We write
\begin{equation}\label{eq:frob:centered-log-coordinates}
  \delta\btau_i=\btau_i-\btau_i^{\rm orb}.
\end{equation}
The point $\kappa^{\rm orb}(K_0)$ is a point of the chosen chart $\widetilde{\mathcal U}_B$, but not a
point of $\Omega_{\rm Frob}$.

We now define the smooth B-side algebra on this domain. Let
\[
  F(\mu,\lambda;\kappa)
  =
  \lambda\,\mu^{l-2}(\mu^2-1)^2
  -\kappa_{l+1}\prod_{i=1}^l(\mu-\kappa_i)(\mu-\kappa_i^{-1})
\]
be the type-$D_l$ curve equation. Since we have removed the pole-cancellation locus, the points over
$\mu=\pm1$ remain punctures and $\partial_\lambda F$ is invertible on the affine curve. Thus the
critical-point condition $d\log\lambda=0$ is equivalent to $\mu\partial_\mu F=0$ upstairs and to
$N(w;\kappa)=0$ downstairs.

\begin{definition}[Smooth invariant Jacobian]\label{def:smooth-invariant-jacobian}
For $\kappa\in\Omega_{\rm Frob}$, the \emph{smooth invariant Jacobian} is
\begin{equation}\label{eq:frob:smooth-jacobian}
  H_B^\iota(\kappa)
  =
  \left(
  {\bC(\nu)[\mu^{\pm1},\lambda^{\pm1}]
  \over
  \langle F,\mu\partial_\mu F\rangle}
  \right)^{\bZ_2},
\end{equation}
where $\bZ_2$ acts by the involution $\iota:\mu\mapsto\mu^{-1}$ and fixes $\lambda$.
Equivalently, in the invariant coordinate $w=\mu+\mu^{-1}$, there is an isomorphism
\[
  H_B^\iota(\kappa)\cong\bC(\nu)[w]/\langle N(w;\kappa)\rangle.
\]
These algebras are the fibers of the locally free rank-$(l+1)$ algebra bundle
\[
  \cH_B^{\rm sm}
  =
  \left(
    \mathcal O_{\Omega_{\rm Frob}}[w]/\langle N(w;\kappa)\rangle
  \right)\otimes_{\bC}\bC(\nu),
\]
whose product is induced by multiplication in the quotient. The Kodaira--Spencer representative of a
vector field $X$ on $\Omega_{\rm Frob}$ is
\begin{equation}\label{eq:frob:ks-map}
  X\longmapsto [\,2\nu\,\partial_X\log\lambda\,]\in H_B^\iota(\kappa).
\end{equation}
Fiberwise, these representatives define a bundle map
$\mathrm{KS}:T\Omega_{\rm Frob}\to\cH_B^{\rm sm}$.
\end{definition}

\subsection{The residue metric and cubic tensor}

We consider the residue metric and cubic tensor associated with the type-$D_l$ relativistic Toda
curve of Definition~\ref{def:toda-curve} and fix their normalization for use below. On the smooth type-$D_l$ curve
\[
  \lambda(w)=\kappa_{l+1}{\prod_{i=1}^l(w-c_i)\over w^2-4},
  \qquad w=\mu+\mu^{-1},
\]
let
\[
  \phi_D^2=\nu\left({d\mu\over\mu}\right)^2.
\]
We first consider the residue forms
\begin{equation}\label{eq:frob:raw-metric}
  \eta_0(X,Y)
  =\sum_{\tilde p:\,d\log\lambda(\tilde p)=0}
  \Res_{\tilde p}
  {\partial_X\log\lambda\,\partial_Y\log\lambda\over d\log\lambda}\,\phi_D^2,
\end{equation}
and
\begin{equation}\label{eq:frob:raw-cubic}
  c_0(X,Y,Z)
  =\sum_{\tilde p:\,d\log\lambda(\tilde p)=0}
  \Res_{\tilde p}
  {\partial_X\log\lambda\,\partial_Y\log\lambda\,\partial_Z\log\lambda
  \over d\log\lambda}\,\phi_D^2.
\end{equation}
Here derivatives are taken at fixed $\mu$; we call this the fixed-$\mu$ lift. The
sums in \eqref{eq:frob:raw-metric} and \eqref{eq:frob:raw-cubic} are upstairs
{on the smooth $\mu$-curve, following Dubrovin's residue construction of Hurwitz Frobenius manifolds
\cite[Lecture~5]{dubrovin1994geometry2dtopologicalfield}, with superpotential $\log\lambda$.}

\begin{definition}[Normalization of the residue metric and cubic tensor]\label{def:residue-normalization}
We define the metric and cubic tensor used below by
\begin{equation}\label{eq:frob:threefold-tensors}
  \eta^{B,\cX}=-{1\over4\nu^2}\eta_0,
  \qquad
  c^{B,\cX}=-{1\over2\nu}c_0.
\end{equation}
\end{definition}

{To express these tensors as sums of residues over critical orbits, let}
\[
  \hat x=-2\nu\log\lambda,
  \qquad
  W=-\hat x=2\nu\log\lambda.
\]
For a vector field $X$ on $\Omega_{\rm Frob}$, set
\[
  f_X=\partial_X\log\lambda,
  \qquad
  c_X^B=2\nu f_X=\partial_XW.
\]

\begin{lemma}[{Residue formulas as sums over critical orbits}]\label{lem:quotient-residue-normalization}
For $\kappa\in\Omega_{\rm Frob}$, the metric and cubic tensor can equivalently be written as
\begin{equation}\label{eq:frob:quotient-residue-metric}
  \eta^{B,\cX}(X,Y)
  ={1\over(2\nu)^3}\sum_{p_\alpha}\Res^{\iota}_{p_\alpha}
  {c_X^B c_Y^B\,\phi_D^2\over d\hat x},
\end{equation}
and
\begin{equation}\label{eq:frob:quotient-residue-cubic}
  c^{B,\cX}(X,Y,Z)
  ={1\over(2\nu)^3}\sum_{p_\alpha}\Res^{\iota}_{p_\alpha}
  {c_X^B c_Y^B c_Z^B\,\phi_D^2\over d\hat x}.
\end{equation}
The sum is over the $l+1$ quotient critical points. The notation $\Res^{\iota}$ means {the sum of the lifted residues over the orbit;} for a free critical orbit it is the sum of the two equal upstairs
residues.
\end{lemma}

\begin{proof}
The invariant differential in the logarithmic residue is unchanged by the involution
$\mu\mapsto\mu^{-1}$. We deliberately use {the sum of the lifted residues over the orbit,} rather than the ordinary
residue of a descended quotient differential. Thus a free critical orbit contributes the sum of the
two lifted residues. This accounts for the orbit factor which would be missed by evaluating at only
one lift.

The remaining constants come from the paper's phase convention. Since
\[
  d\hat x=-2\nu\,d\log\lambda,
  \qquad
  c_X^B=2\nu\,\partial_X\log\lambda,
\]
we have
\[
  {1\over(2\nu)^3}{c_X^B c_Y^B\over d\hat x}
  =-{1\over4\nu^2}{\partial_X\log\lambda\,\partial_Y\log\lambda\over d\log\lambda},
\]
and similarly
\[
  {1\over(2\nu)^3}{c_X^B c_Y^B c_Z^B\over d\hat x}
  =-{1\over2\nu}{\partial_X\log\lambda\,\partial_Y\log\lambda\,\partial_Z\log\lambda
  \over d\log\lambda}.
\]
Substituting these identities into {the residue formulas summed over critical orbits} gives
$\eta^{B,\cX}=-\eta_0/(4\nu^2)$ and $c^{B,\cX}=-c_0/(2\nu)$, as in
\eqref{eq:frob:threefold-tensors}.
\end{proof}

\subsection{{The Kodaira--Spencer map and residue tensors}}

The Kodaira--Spencer representative of $X$ is the class
$c_X^B=[\,2\nu\,\partial_X\log\lambda\,]$ in the smooth invariant Jacobian. We now compare
the Hurwitz residue tensors of the preceding subsection with the fiberwise algebra structure of
Definition~\ref{def:smooth-invariant-jacobian}.

\begin{proposition}[Comparison of tensors under the Kodaira--Spencer map]
\label{prop:ks-tensor-comparison}
For $\kappa\in\Omega_{\rm Frob}$, the Kodaira--Spencer map identifies
$T_\kappa\Omega_{\rm Frob}$ with $H_B^\iota(\kappa)$. Transport quotient multiplication to the
tangent space by
\[
  X\star_{\rm Jac}Y
  =\mathrm{KS}^{-1}(c_X^Bc_Y^B).
\]
Then
\[
  \eta^{B,\cX}(X\star_{\rm Jac}Y,Z)=c^{B,\cX}(X,Y,Z).
\]
Consequently, $\star_{\rm Jac}$ is the product determined by the Hurwitz residue tensors,
and
\begin{equation}\label{eq:frob:ks-product-comparison}
  c_{X\circ Y}^B=c_X^Bc_Y^B.
\end{equation}
\end{proposition}

\begin{proof}
Let $u_\alpha=W(p_\alpha)$ be the critical values. Differentiating at fixed $\mu$ gives
\[
  X(u_\alpha)=\partial_XW(p_\alpha)=c_X^B(p_\alpha),
\]
because $dW(p_\alpha)=0$. The critical values are local coordinates on $\Omega_{\rm Frob}$,
so evaluation at the critical points shows that the Kodaira--Spencer map is an isomorphism.

Put $U=X\star_{\rm Jac}Y$, so that $c_U^B=c_X^Bc_Y^B$. {The formulas summed over critical orbits}
\eqref{eq:frob:quotient-residue-metric} and \eqref{eq:frob:quotient-residue-cubic} give
\[
  \eta^{B,\cX}(U,Z)
  ={1\over(2\nu)^3}\sum_{p_\alpha}\Res^\iota_{p_\alpha}
    {c_X^Bc_Y^Bc_Z^B\,\phi_D^2\over d\hat x}
  =c^{B,\cX}(X,Y,Z).
\]
Since the Hurwitz metric is nondegenerate on $\Omega_{\rm Frob}$, this gives
$U=X\circ Y$ and proves \eqref{eq:frob:ks-product-comparison}.
\end{proof}

Accordingly, multiplication by a vector field $X$ is represented on the invariant Jacobian by
\begin{equation}\label{eq:frob:product-operator}
  C_X^B=\operatorname{mult}_{c_X^B}.
\end{equation}
The scale vector field
\[
  e={\kappa_{l+1}\over2\nu}\partial_{\kappa_{l+1}}
\]
satisfies $c_e^B=1$ and is therefore the unit.

\begin{definition}[B-side Frobenius manifold]\label{def:bside-frobenius-manifold}
On $\Omega_{\rm Frob}$, let $\cM^B$ be the semisimple B-side Dubrovin--Frobenius
manifold whose tangent algebra is $H_B^\iota(\kappa)$, whose product is the invariant
Jacobian product of Definition~\ref{def:smooth-invariant-jacobian}, and whose metric and cubic tensor
are $\eta^{B,\cX}$ and $c^{B,\cX}$ in \eqref{eq:frob:threefold-tensors}.
\end{definition}

If $e_\alpha$ is a primitive idempotent of $\cM^B$, then
\begin{equation}\label{eq:frob:idempotent-metric}
  \eta^{B,\cX}(e_\alpha,e_\beta)
  ={\delta_{\alpha\beta}\over\Delta^{\cX}_\alpha}.
\end{equation}
When no confusion is possible in Sections~\ref{sec:anchor} and~\ref{sec:graphsum-remodeling}, the
symbol $\Delta_\alpha$ means the quantity in
\eqref{eq:frob:idempotent-metric}.
The Frobenius canonical coordinate attached to $e_\alpha$ is the critical value of $W$:
\begin{equation}\label{eq:frob:canonical-coordinate}
  u_\alpha=W(p_\alpha)=2\nu\log\lambda(p_\alpha).
\end{equation}
The differential identity
\[
  X(u_\alpha)=c_X^B(p_\alpha)
\]
was established in the proof of Proposition~\ref{prop:ks-tensor-comparison}.
{Equation~\eqref{eq:frob:canonical-coordinate} fixes the canonical-coordinate convention used below.}
We will use \eqref{eq:frob:canonical-coordinate} again in {the calibration on $\Omega_B$ in}
Section~\ref{sec:anchor}.

\subsection{The A-model Frobenius manifold and the affine change of variables}

Let
\[
  H_A=H^*_{\CR,\bT}(\cX)\otimes_{\bC[\nu]}\bC(\nu)
\]
be the localized equivariant Chen--Ruan cohomology space from Section~\ref{sec:amodel}. Let
$\mathcal U_A\subset H_A$ be a sufficiently small analytic neighborhood of the origin in the
age-normalized orbifold flat coordinates, and let $\cM^{\cX}$ denote the genus-zero primary
$\bT$-equivariant orbifold quantum-cohomology Frobenius manifold of $\cX$ on $\mathcal U_A$.

We now write the affine change of variables used in Theorem~\ref{thm:frob-iso}. After shrinking the
chosen branch chart if necessary, the formulas below define a holomorphic map
\[
  \Phi:\widetilde{\mathcal U}_B\longrightarrow\mathcal U_A.
\]
This map comes from Hu's
quantum McKay theorem, which relates the root coordinates for the quantum cohomology of the type-$D_l$
resolution to the orbifold Chen--Ruan coordinates of $[\bC^2/\Gamma]$
\cite{hu2012quantummckaycorrespondencesingularities}. By construction, it sends
$\kappa^{\rm orb}(K_0)$ to the origin of $H_A$.

\begin{definition}[The affine change of variables]\label{def:mirror-map}
Write the B-model parameters as
\[
  \kappa=(\kappa_1,\ldots,\kappa_l;K),
  \qquad K=\kappa_{l+1}.
\]
Recall the logarithmic coordinates $\btau_i$ and their centered differences $\delta\btau_i$ from
\eqref{eq:frob:orbifold-log-coordinates} and \eqref{eq:frob:centered-log-coordinates}. Write
\[
  \delta\btau
  =(\delta\btau_1,\ldots,\delta\btau_l)^{\mathsf T}.
\]

Let $G=(G_{ij})_{1\leq i,j\leq l}$ be the type-$D_l$ resolution matrix
\begin{equation}\label{eq:frob:G-matrix}
  G_{ij}=\delta_{ij}-\delta_{i,j+1}+\delta_{i,l-1}\delta_{j,l},
  \qquad
  G^{\mathsf T}G=C,
\end{equation}
where $C$ is the $D_l$ Cartan matrix. Define the shifted resolution root-coordinate vector
$y=(y_1,\ldots,y_l)^{\mathsf T}$ by
\begin{equation}\label{eq:frob:resolution-coordinates}
  y=G^{-1}\delta\btau,
  \qquad\text{equivalently}\qquad
  \delta\btau=Gy.
\end{equation}
Thus $y$ is obtained from the centered logarithmic B-model coordinates by the constant change of
basis $G^{-1}$; the scale parameter $K$ is treated separately below.

Let $\chi_1,\ldots,\chi_l$ be the nontrivial irreducible characters corresponding to the $D_l$ nodes,
and index the columns of the matrix $\mathsf L$ below by the nontrivial conjugacy classes $[g]$ of $\Gamma$. Define
\begin{equation}\label{eq:frob:L-matrix}
  \mathsf L_{j,[g]}
  ={|[g]|\over|\Gamma|}\sqrt{2-\chi_V(g)}\,\overline{\chi_j(g)},
\end{equation}
with the same choice of square roots as in the Bryan--Gholampour coordinate formula, in Hu's
continuation convention. Character orthogonality implies that $\mathsf L$ is invertible: a linear
combination of nontrivial irreducible characters vanishing on all nonidentity conjugacy classes would
be supported at the identity, hence would be a multiple of the regular character, impossible without a
trivial-character component. We set
\[
  \mathsf H=\mathsf L^{-1}
\]
and define the centered Hu coordinates before age normalization by
\begin{equation}\label{eq:frob:raw-hu-map}
  (x^{\rm Hu}_{[g]})_{[g]\ne[1]}
  =i\,\mathsf H\,C\,y
  =i\,\mathsf H\,C\,G^{-1}\delta\btau
  =i\,\mathsf H\,G^{\mathsf T}\delta\btau.
\end{equation}
These variables are dual to the conjugacy-sector basis before age normalization. Since Section~\ref{sec:amodel} uses the
age-normalized classes $\overline{\mathbf 1}_{[g]}=\nu^{-1}\mathbf 1_{[g]}$ for $g\neq1$, the paper's
twisted orbifold flat coordinates are
\begin{equation}\label{eq:frob:age-normalized-map}
  x_{[g]}=\nu\,x^{\rm Hu}_{[g]},
  \qquad [g]\neq[1].
\end{equation}
For the scale direction we use
\begin{equation}\label{eq:frob:scale-coordinate}
  x_{l+1}=2\nu\log{K\over K_0},
\end{equation}
so the unit vector is $\partial_{x_{l+1}}$. Then the scaled orbifold point
$(\kappa_1^{\rm orb},\ldots,\kappa_l^{\rm orb};K_0)$ maps to the origin $0\in\mathcal U_A$. In the ordering of the
nonidentity conjugacy classes fixed by the columns of $\mathsf L$, the affine map in the chosen
logarithmic coordinates is
\begin{equation}\label{eq:frob:affine-map}
  \Phi(\kappa_1,\ldots,\kappa_l;K)
  =\bigl((x_{[g]})_{[g]\ne[1]},x_{l+1}\bigr)
  \in\mathcal U_A,
\end{equation}
with components given by \eqref{eq:frob:raw-hu-map},
\eqref{eq:frob:age-normalized-map}, and \eqref{eq:frob:scale-coordinate}.
Because $G$ and $\mathsf L$ are invertible and the logarithm branches are fixed on
$\widetilde{\mathcal U}_B$, these formulas define a holomorphic local coordinate isomorphism after
shrinking the chosen charts if necessary.
\end{definition}

\begin{remark}[Comparison with Hu's formula]\label{rem:hu-coordinate-check}
The raw variables $x^{\rm Hu}$ in Definition~\ref{def:mirror-map} are obtained by inverting the linear part
of Hu's type-$D$ quantum McKay coordinate transformation in the conjugacy-sector basis before age
normalization. The
variables $x_{[g]}=\nu x^{\rm Hu}_{[g]}$ are the corresponding age-normalized Chen--Ruan coordinates
used in Section~\ref{sec:amodel}. In particular, the vector $\btau^{\rm orb}$ is the affine
root-of-unity shift in Hu's formula, and the centered coordinates $x_{[g]}$ vanish at the Chen--Ruan
origin.

{Indeed, retain Hu's indexing by conjugacy classes. In the resulting basis the root-coordinate
expression uses the conjugate character matrix.} More explicitly, set
\[
  A_{[g],j}=\sqrt{2-\chi_V(g)}\,\chi_j(g).
\]
With the branch
$\sqrt{\chi_V(g)-2}=i\sqrt{2-\chi_V(g)}$, the Bryan--Gholampour direct coordinate transformation is
$x^{\rm Hu}=iAy$. On the other hand, Hu's formula groups the class variables through
\[
  \mathsf L_{j,[g]}
  ={|[g]|\over|\Gamma|}\sqrt{2-\chi_V(g)}\,\overline{\chi_j(g)}.
\]
Character orthogonality, together with the McKay relation for the defining representation, gives
\[
  \mathsf L A=C.
\]
Thus the nonconstant part of Hu's root expression is $\mathsf Lx^{\rm Hu}$ and must be compared with
$iCy$ below.

On the resolution side a root $\beta=\sum_jb_j\alpha_j$ evaluates on the resolution coordinate $y$ as
\[
  \beta(y)=b^{\mathsf T}Cy.
\]
With the square-root branch fixed above, equality of the root-coordinate expressions for all
coefficient vectors $b$ requires
\[
  \mathsf L x^{\rm Hu}=i C y.
\]
Since $y=G^{-1}\delta\btau$ and $C=G^{\mathsf T}G$, this gives
\[
  x^{\rm Hu}=i\,\mathsf L^{-1}CG^{-1}\delta\btau
  =i\,\mathsf L^{-1}G^{\mathsf T}\delta\btau .
\]
{The final multiplication by $\nu$ is the change of coordinates induced by the age normalization
of the basis used in the A-model section.}

It remains to check the affine origin. Let
\[
  \mathbf n=(2l-4,2l-5,2l-7,\ldots,3,1,0)^{\mathsf T}
\]
be the orbifold exponent vector, so that
\[
  \btau^{\rm orb}= -{2\pi i\over4(l-2)}\,\mathbf n .
\]
A direct multiplication by the matrix $G^{\mathsf T}$ gives
\[
  G^{\mathsf T}\mathbf n
  =(1,2,\ldots,2,1,1)^{\mathsf T},
\]
{the vector of Dynkin marks, equivalently the coefficients of the highest root, for $D_l$.} Hence
\[
  G^{\mathsf T}\btau^{\rm orb}
  =-{2\pi i\over|\Gamma|}\,G^{\mathsf T}\mathbf n,
  \qquad |\Gamma|=4(l-2),
\]
which is Hu's root-of-unity shift after the same square-root branch choice, because
\[
  {2\pi\over|\Gamma|}G^{\mathsf T}\mathbf n
  +iG^{\mathsf T}\bigl(\btau-\btau^{\rm orb}\bigr)
  =iG^{\mathsf T}\btau .
\]
Subtracting $\btau^{\rm orb}$ is therefore the centering which sends the orbifold point to
$x=0$.
\end{remark}

\subsection{The Frobenius-manifold isomorphism}

We can now state the genus-zero input in the normalization used by the rest of the paper.

\begin{theorem}[Frobenius-manifold isomorphism on $\Omega_{\rm Frob}$]\label{thm:frob-iso}
The restriction of the affine map in Definition~\ref{def:mirror-map} is an isomorphism of semisimple
Frobenius manifolds
\begin{equation}\label{eq:frob:mirror-isomorphism}
  \Phi\big|_{\Omega_{\rm Frob}}:\cM^B\big|_{\Omega_{\rm Frob}}
  \xrightarrow{\ \sim\ }
  \cM^{\cX}\big|_{\Phi(\Omega_{\rm Frob})}.
\end{equation}
{It identifies the invariant Jacobian product and residue metric
with the equivariant orbifold quantum product and pairing of $\cX$.
The unit is preserved, and all tensors use the threefold normalization
of Definition~\ref{def:residue-normalization}.}
\end{theorem}

\begin{proof}
Proposition~\ref{prop:ks-tensor-comparison} identifies the B-side Hurwitz product with multiplication
in the invariant Jacobian under the Kodaira--Spencer map. The discussion following that proposition
also identifies the scale vector field with the unit. Thus Definition~\ref{def:bside-frobenius-manifold}
puts the B-model Frobenius structure in precisely the form used in the identifications below.

Brini--Ma--Strachan identify the Dubrovin-dual Frobenius manifold of the type-$D_l$ relativistic
Toda curve with the equivariant quantum-cohomology Frobenius manifold of the $D_l$ minimal resolution
\cite{brini2025dubrovindualitymirrorsymmetry}. In the notation of
Definition~\ref{def:mirror-map}, their logarithmic branch coordinates are $\btau_i$ and the
resolution root coordinates are $y=G^{-1}\delta\btau$; see
\eqref{eq:frob:G-matrix}--\eqref{eq:frob:resolution-coordinates}. The resolution-side structure is
the root-system Frobenius structure computed by Bryan--Gholampour
\cite{bryan2007rootsystemsquantumcohomology}. Hu's quantum McKay theorem continues this structure to
the Chen--Ruan frame of the binary dihedral quotient
\cite{hu2012quantummckaycorrespondencesingularities}. Definition~\ref{def:mirror-map} writes the
resulting composite coordinate change as $\Phi$, while Remark~\ref{rem:hu-coordinate-check} verifies
the affine centering, character-basis transformation, and age normalization. Consequently, $\Phi$
identifies the products and units of $\cM^B$ and $\cM^{\cX}$.

{It remains to compare the scalar normalizations. In the
Brini--Ma--Strachan surface normalization, the cubic tensor is $c_0$ and
the unit is $e=(2\nu)^{-1}\partial_{\log K}$. Since
$\partial_e\log\lambda=(2\nu)^{-1}$, contraction with the unit gives
\[
  c^{\rm surf}=c_0,\qquad
  \eta^{\rm surf}(X,Y)=c_0(X,Y,e)=\frac1{2\nu}\eta_0(X,Y).
\]
Adjoining the equivariant line of weight $-2\nu$ multiplies both tensors
by $(-2\nu)^{-1}$, as in Remark~\ref{rem:amodel:metric-comparison}.
Consequently,
\[
  \eta^{B,\cX}=-\frac1{4\nu^2}\eta_0,\qquad
  c^{B,\cX}=-\frac1{2\nu}c_0,
\]
exactly as in Definition~\ref{def:residue-normalization}. Thus $\Phi$
preserves the metric as well as the product and unit, proving
\eqref{eq:frob:mirror-isomorphism}.}
\end{proof}

\begin{proposition}[Coordinate consequences of the isomorphism]\label{prop:mirror-match}
Under the isomorphism of Theorem~\ref{thm:frob-iso}, the canonical idempotent pairings agree in the
threefold normalization:
\begin{equation}\label{eq:frob:mirror-idempotent-metric}
  \eta^{\cX}(e_\alpha,e_\beta)
  =\eta^{B,\cX}(e_\alpha,e_\beta)
  ={\delta_{\alpha\beta}\over\Delta^{\cX}_\alpha}.
\end{equation}
The canonical coordinates on the B-side are the critical values
$u_\alpha=W(p_\alpha)$ of \eqref{eq:frob:canonical-coordinate}; under the isomorphism, they are the
canonical coordinates attached to the corresponding A-side idempotents.
\end{proposition}

\begin{proof}
The pairing identity is the metric part of Theorem~\ref{thm:frob-iso}. Since the isomorphism
preserves multiplication, it identifies the primitive idempotents and hence their canonical
coordinates. Their B-side description is the intrinsic Kodaira--Spencer calculation in
Proposition~\ref{prop:ks-tensor-comparison}.
\end{proof}

\subsection{The orbifold neighborhood and the B-model algebra extension}
\label{subsec:pole-cancellation}

The affine map satisfies
\[
  \Phi\bigl(\kappa^{\rm orb}(K_0)\bigr)=0.
\]
Although $\kappa^{\rm orb}(K_0)\notin\Omega_{\rm Frob}$ and hence
$0\notin\Phi(\Omega_{\rm Frob})$, the point $0$ lies in the closure of
$\Phi(\Omega_{\rm Frob})$. By the Bryan--Gholampour--Hu formula for the genus-zero Gromov--Witten
potential of $[\bC^2/\Gamma]$, its third derivatives are holomorphic near the orbifold origin; the
unit direction is recovered from the string equation
\cite{bryan2007rootsystemsquantumcohomology,hu2012quantummckaycorrespondencesingularities}.

For the threefold $\cX=[\bC^2/\Gamma]\times\bC$, the additional trivial line of $\bT$-weight
$-2\nu$ contributes the inverse equivariant Euler factor
\[
  e_{\bT}(\bC_{-2\nu})^{-1}=(-2\nu)^{-1}.
\]
Consequently, both the equivariant orbifold Poincar\'e pairing and the genus-zero primary three-point
tensor of $[\bC^2/\Gamma]$ are multiplied by $(-2\nu)^{-1}$. Since they are rescaled by the same
nonzero factor, the quantum product is unchanged, and the resulting Frobenius pairing is the one
computed in Lemma~\ref{lem:frob}. Hence, after shrinking $\mathcal U_A$, the A-side Frobenius
manifold $\cM^{\cX}$ is holomorphic on $\mathcal U_A$, including at $0$. The Frobenius algebra at the
origin is
\begin{equation}\label{eq:frob:aside-special-fiber}
  T_0\mathcal U_A\cong Z\bigl(\bC(\nu)[\Gamma]\bigr)
\end{equation}
with the form of Lemma~\ref{lem:frob}, and the formal completion of $\cM^{\cX}$ at $0$ is the
formal equivariant orbifold quantum-cohomology Frobenius structure. After shrinking
$\widetilde{\mathcal U}_B$ if necessary, $\Phi$ takes values in $\mathcal U_A$, and its restriction
over $\Omega_{\rm Frob}$ is the isomorphism of Theorem~\ref{thm:frob-iso}.

This A-side continuation does not define the ordinary logarithmic residue construction on the
singular B-model central fiber. On the B-side, we instead construct a finite flat family of Frobenius
algebras over the full chosen chart $\widetilde{\mathcal U}_B$. Away from the pole-cancellation locus
$\mathcal Z_{\rm pc}$, it agrees with the invariant Jacobian algebra; along the two local branches,
we replace the corresponding rank-two factors by the elementary modifications described below.

Along $\mathcal Z_{\rm pc}$, {a pair of zeros $\{\kappa_i,\kappa_i^{-1}\}$ of $\lambda$ exchanged by $\iota$} collides upstairs at the $\iota$-fixed point
$\mu=\pm1$. Ordinary critical collisions are kept as finite local residue algebras. At an $\iota$-fixed point after pole
cancellation, the local summand is replaced by the stabilizer-algebra summand used below.

Let
\[
  \mathcal A=\mathcal O_{\widetilde{\mathcal U}_B}[w]/\langle N(w;\kappa)\rangle .
\]
Its restriction to $\Omega_{\rm Frob}$ is the smooth invariant
Jacobian family $\cH_B^{\rm sm}$ of Definition~\ref{def:smooth-invariant-jacobian}.

\begin{proposition}[B-model algebra extension]\label{prop:bmodel-flat-algebra-family}
After possibly shrinking $\widetilde{\mathcal U}_B$, there exists a locally free analytic family
of rank $l+1$ commutative $\bC(\nu)$-algebras over the chosen chart, which we denote by $\cH_B$.
Its restriction to $\Omega_{\rm Frob}$ is the ordinary invariant Jacobian algebra bundle underlying
$\cM^B$. Moreover, the Kodaira--Spencer map extends to a holomorphic bundle isomorphism
\[
  \widetilde{\mathrm{KS}}:
  T\widetilde{\mathcal U}_B\xrightarrow{\ \sim\ }\cH_B.
\]
\end{proposition}

\begin{proof}
\emph{Step 1: the fiber of $\mathcal A$ over the orbifold point.}
At fixed $\kappa$, $\mathcal A$ reduces to the $\bC(\nu)$ module
\[
  \bC(\nu)[w]/\langle N(w;\kappa)\rangle
\]
and is free of rank $l+1$, since $N(w;\kappa)$ has degree $l+1$ and leading coefficient
$l-2\neq0$.
At $\kappa^{\rm orb}(K_0)$ one has
\[
  P(w;\kappa^{\rm orb})=(w^2-4)Q(w),
  \qquad
  N(w;\kappa^{\rm orb})=(w^2-4)^2Q'(w),
\]
and
\[
  Q(w)=2T_m\!\left({w\over2}\right),
  \qquad m=l-2.
\]
Here $T_m$ and $U_{m-1}$ are the Chebyshev polynomials of the first and second kinds, respectively.
Thus
\[
  Q'(w)=m\,U_{m-1}\!\left({w\over2}\right),
\]
so $w=\pm2$ are double roots and the remaining $m-1=l-3$ roots are simple. The two summands
\[
  \bC(\nu)[w]/\langle(w-2)^2\rangle,
  \qquad
  \bC(\nu)[w]/\langle(w+2)^2\rangle.
\]
therefore give nonreduced factors of $\mathcal A$ at the orbifold point.

\emph{Step 2: construction of $\cH_B$.}
We now perform local elementary modifications that construct $\cH_B$. We describe the construction at $w=2$; the construction at $w=-2$ is identical.

Let $t=w-2$ be the local coordinate on the $(w)$-line. Near the orbifold point, the corresponding branch of the pole-cancellation locus is $\kappa_l=1$. On the chosen logarithmic base chart, set
\[
 \varepsilon_+=\log_{\mathrm{orb}}\kappa_l,
\]
so that
\[
\kappa_l=e^{\varepsilon_+},\qquad
c_l=2\cosh\varepsilon_+,
\]
and the reduced pole-cancellation branch is $\varepsilon_+=0$. Choose coordinates $\mathbf s$
along this branch, centered at the orbifold point and including $\varepsilon_-$ among them, so that
$(\mathbf s,\varepsilon_+)$ is a local coordinate system on the base chart.
Thus $\mathbf s=0$ is a slice transverse to the divisor $\{\varepsilon_+=0\}$. Set
$\delta_+=c_l-2$.

Along the hypersurface $\varepsilon_+=0$ in the base chart, the monic numerator polynomial has the form
\[
  P(2+t;\mathbf s,0)=tS_+(t,\mathbf s,0),
  \qquad S_+(0,0,0)\neq0.
\]
Substituting this into $N=(w^2-4)P'-2wP$ gives, along this hypersurface, the exact identity
\[
  N(2+t;\mathbf s,0)
  =t^2\bigl((4+t)\partial_tS_+(t,\mathbf s,0)-S_+(t,\mathbf s,0)\bigr).
\]
The factor in parentheses is a unit near the orbifold point: at the central point
$S_+(t,0,0)=(4+t)Q(2+t)$, so its value at $t=0$ is $16Q'(2)\neq0$.
After shrinking the chart, we choose a holomorphic extension, still denoted $S_+$, such that
\[
  P(2+t;\mathbf s,\varepsilon_+)=(t-\delta_+)S_+(t,\mathbf s,\varepsilon_+).
\]
Weierstrass preparation therefore gives a monic quadratic factor
\[
  q_+(t;\mathbf s,\varepsilon_+)
  =t^2+A_+(\mathbf s,\varepsilon_+)t+B_+(\mathbf s,\varepsilon_+)
\]
supported near $w=2$; below we denote this factor by $N_+=q_+$. {The pair of zeros that cancels the pole} enters through $2\cosh\varepsilon_+$, so
$A_+$ and $B_+$ are even in $\varepsilon_+$; moreover
$A_+(\mathbf s,0)=B_+(\mathbf s,0)=0$. On the transverse slice $\mathbf s=0$, the expansion of
$N$ from $P=(t-\delta_+)S_+$ has the form
\[
  N(2+t;0,\varepsilon_+)=n_{20} t^2+n_{02}\varepsilon_+^2
  +O(t^3,t\varepsilon_+^2,\varepsilon_+^4),
\]
where
\[
  n_{20}=16Q'(2)\neq0,
  \qquad
  n_{02}=16Q(2)\neq0.
\]
Indeed, the value of $n_{20}$ follows from the unit computed above, while at $t=0$ one has
$N(2;0,\varepsilon_+)=4\delta_+S_+(0,0,\varepsilon_+)$ and
$\delta_+=\varepsilon_+^2+O(\varepsilon_+^4)$.
Thus the discriminant vanishes to order two at the orbifold point. After shrinking the
chosen chart and choosing the square root of the resulting unit on the ramified chart, we may write
\[
  A_+^2-4B_+=4\varepsilon_+^2a_+(\mathbf s,\varepsilon_+)^2,
  \qquad a_+(0,0)\neq0.
\]
With $b_+=-A_+/2$ this gives the exact relative factorization
\[
  q_+(t;\mathbf s,\varepsilon_+)
  =(t-b_+(\mathbf s,\varepsilon_+))^2
  -\varepsilon_+^2a_+(\mathbf s,\varepsilon_+)^2.
\]
Let $V\subset\widetilde{\mathcal U}_B$ be the sufficiently small analytic base chart used above, and write
\[
  \mathcal O
  =
  \mathcal O_V\otimes_{\bC}\bC(\nu).
\]
Let
\[
  \mathcal A_+
  =
  \mathcal O[t]/\langle q_+(t;\mathbf s,\varepsilon_+)\rangle,
\]
and let $u_+$ be the class of $1$.

The map from the logarithmic base chart to the invariant $c_l$-chart is locally the double cover
\[
  \varepsilon_+\longmapsto
  \delta_+=c_l-2
  =4\sinh^2\!\left(\frac{\varepsilon_+}{2}\right),
\]
with deck involution
\[
  \sigma_+:\varepsilon_+\longmapsto-\varepsilon_+.
\]
It is ramified along the pole-cancellation branch $\varepsilon_+=0$. Since $b_+$ and $a_+^2$ are $\sigma_+$-invariant and $a_+$ is a unit, we choose $a_+$ to be $\sigma_+$-invariant. Define
\[
  \mathcal H_{B,+}
  =
  \mathcal O u_+\oplus\mathcal O\xi_+
  \subset
  \mathcal A_+[\varepsilon_+^{-1}],
  \qquad
  \xi_+
  =
  \frac{t-b_+}{\varepsilon_+a_+}.
\]
Then
\[
  \xi_+^2=u_+,
  \qquad
  \sigma_+(u_+)=u_+,
  \qquad
  \sigma_+(\xi_+)=-\xi_+.
\]
We refer to $u_+$ and $\xi_+$ as the \emph{even} and \emph{odd} part of the modified rank-two summand.

Over $\varepsilon_+\neq0$ the split primitive idempotents are
\[
  e_{+,+}={u_++\xi_+\over2},
  \qquad
  e_{+,-}={u_+-\xi_+\over2}.
\]
The same argument at $w=-2$, using the branch $\kappa_1=-e^{\varepsilon_-}$ and local coordinate
$t_-=w+2$, gives a monic quadratic Weierstrass factor $N_-$ supported near $w=-2$, the quotient
\[
  \mathcal A_- = \mathcal O[t_-]/\langle N_-\rangle,
\]
and a modified rank-two factor $\mathcal H_{B,-}$ with generator $\xi_-$ and relation
$\xi_-^2=u_-$.

Choose disjoint Weierstrass neighborhoods of $w=2$, $w=-2$, and the residual roots.
After shrinking the chosen chart, the pairwise resultants of $N_+$, $N_-$, and the residual factor
$N_{\rm res}$ are units, and
\[
  N=N_+N_-N_{\rm res}.
\]
Set
\[
  \mathcal A_{\rm res}=\mathcal O[w]/\langle N_{\rm res}\rangle.
\]
The Chinese remainder theorem gives
\[
  \mathcal A\cong \mathcal A_+\oplus\mathcal A_-\oplus\mathcal A_{\rm res}.
\]
We set
\[
  \cH_B=\mathcal H_{B,+}\oplus\mathcal H_{B,-}\oplus\mathcal A_{\rm res}.
\]
Each summand $\mathcal H_{B,\pm}$ is locally free of rank two, while $\mathcal A_{\rm res}$ is locally free of rank $l-3$. Hence $\cH_B$ is locally free of rank
\[
  2+2+(l-3)=l+1.
\]
We equip $\cH_B$ with the direct-sum algebra structure, using the multiplication inherited from $\mathcal A_{\rm res}$ on the residual summand and, for $p\in\{+,-\}$, the $\mathcal O$-bilinear multiplication
\[
    u_p^2=u_p,
    \qquad
    u_p\xi_p=\xi_p,
    \qquad
    \xi_p^2=u_p
\]
on $\mathcal H_{B,p}$. Over $\varepsilon_p\neq0$, this multiplication agrees with that of $\mathcal A_p$: indeed,
\[
    t_p=b_p+\varepsilon_pa_p\xi_p,
    \qquad
    e_{p,+}=\frac{u_p+\xi_p}{2},
    \qquad
    e_{p,-}=\frac{u_p-\xi_p}{2},
 \]
and $e_{p,+},e_{p,-}$ are the two orthogonal primitive idempotents of the split quadratic factor. Consequently,
\[
    \cH_B\big|_{V\setminus\mathcal Z_{\rm pc}}
    \cong
    \mathcal A\big|_{V\setminus\mathcal Z_{\rm pc}},
\]
and the restriction of $\cH_B$ to $\Omega_{\rm Frob}$ is the ordinary smooth invariant Jacobian algebra bundle $\cH_B^{\rm sm}$.

\emph{Step 3: the modified Kodaira--Spencer map.}
It remains to extend the bundle identification induced by $\Phi$ on $\Omega_{\rm Frob}$ to the
modified frame. Define the modified
Kodaira--Spencer map on the smooth locus by
\[
  \widetilde{\mathrm{KS}}(X)=\bigl[2\nu\,\partial_X\log\lambda\bigr]\in\cH_B.
\]
For the pole-cancellation vector field $\partial_{\varepsilon_+}$ at $w=2$ we have
\[
  2\nu\,\partial_{\varepsilon_+}\log\lambda
  =-{4\nu\sinh\varepsilon_+\over w-2\cosh\varepsilon_+}.
\]
In the modified factor this denominator is
\[
  w-2\cosh\varepsilon_+
  =(b_+-\delta_+)+\varepsilon_+a_+\xi_+.
\]
Both $b_+$ and $\delta_+$ are even functions of $\varepsilon_+$ vanishing on this hypersurface, hence
$b_+-\delta_+=O(\varepsilon_+^2)$. Thus, in the basis $\{u_+,\xi_+\}$,
\[
  -{4\nu\sinh\varepsilon_+\over (b_+-\delta_+)+\varepsilon_+a_+\xi_+}
  =-4\nu\sinh\varepsilon_+
  {(b_+-\delta_+)-\varepsilon_+a_+\xi_+
  \over (b_+-\delta_+)^2-\varepsilon_+^2a_+^2}.
\]
Both coefficients are holomorphic: $\sinh\varepsilon_+/\varepsilon_+$ is a unit, while
$b_+-\delta_+$ is divisible by $\varepsilon_+^2$. The even coefficient vanishes on this hypersurface
and the odd coefficient has the nonzero limit
\[
  -{4\nu\over a_+(0,0)}.
\]
Equivalently, at the split roots $w_\pm=2+b_+\pm\varepsilon_+a_+$ the two values tend to
$\mp4\nu/a_+(0,0)$. If $X$ is tangent to the $w=2$ branch of the pole-cancellation locus, then
$2\nu\partial_X\log\lambda$ is regular at $w=2$; expanding it in
$t=b_++\varepsilon_+a_+\xi_+$ shows that its odd coefficient is divisible by $\varepsilon_+$. The
calculation at $w=-2$ is the same, with $\varepsilon_-$ and $\xi_-$. Moreover the $w=2$
pole-cancellation vector field is regular at the opposite pole-cancellation summand over $w=-2$, and
the same holds with the two signs interchanged. Hence the two pole-cancellation vector fields have
zero odd limit in the opposite summands, so the odd part of the limiting Kodaira--Spencer matrix is
diagonal with two nonzero entries.

It remains to check the submatrix formed by the residual simple factors together with the even
directions of the two pole-cancellation summands. At the orbifold point the residual roots are the
$m-1$ roots $r_a$ of $U_{m-1}(w/2)$. For $2\le j\le l-1$, consider {the $m=l-2$ pairs}
\[
  \left\{
  \kappa_j^{\rm orb},
  (\kappa_j^{\rm orb})^{-1}
  \right\}
\]
{of zeros of $\lambda$ exchanged by $\iota$, away from its fixed points. Their images under the quotient map} $\mu\mapsto w=\mu+\mu^{-1}$ are
\[
  c_j^{\rm orb}
  =
  \kappa_j^{\rm orb}
  +
  (\kappa_j^{\rm orb})^{-1}.
\]
Set
\[
  \mathsf T=\{r_1,\ldots,r_{m-1},2,-2\}.
\]
These are $m+1$ distinct points, and none is equal to any $c_j^{\rm orb}$: the $c_j^{\rm orb}$
are the simple roots of $Q$, the $r_a$ are the roots of $Q'$, and
$Q(2)Q(-2)\neq0$.

The scale vector gives the constant class $1$. {For the coordinates of the interior pairs of zeros, indexed by}
$2\le j\le l-1$, set $\btau_j=\log_{\rm orb}\kappa_j$; then
\[
  2\nu\,\partial_{\btau_j}\log\lambda
  =2\nu\,{\kappa_j^{-1}-\kappa_j\over w-c_j}.
\]
At the orbifold point the constants
$\gamma_j=(\kappa_j^{\rm orb})^{-1}-\kappa_j^{\rm orb}$ are nonzero {for all interior pairs of zeros.} We claim that the
$m+1$ functions
\[
  1,
  \qquad
  {\gamma_j\over w-c_j^{\rm orb}},\quad 2\le j\le l-1,
\]
are linearly independent after evaluation on $\mathsf T$. Indeed, if
\[
  b_0+\sum_{j=2}^{l-1}{b_j\gamma_j\over w-c_j^{\rm orb}}
\]
vanishes at all points of $\mathsf T$, then after multiplication by
the same central factor, written as $Q(w)=\prod_{j=2}^{l-1}(w-c_j^{\rm orb})$, we obtain a polynomial
of degree at most $m$ with the
$m+1$ distinct roots in $\mathsf T$. Hence it is zero. Evaluating this zero polynomial at
$w=c_j^{\rm orb}$ gives
\[
  b_j\gamma_j\prod_{k\ne j}(c_j^{\rm orb}-c_k^{\rm orb})=0,
\]
so all $b_j$ vanish, and then $b_0=0$. Thus the evaluation matrix with rows indexed by
$\mathsf T$ and columns
\[
  1,\qquad {\gamma_j\over w-c_j^{\rm orb}},\quad 2\le j\le l-1,
\]
is invertible {and gives the block for the residual simple factors together with the two even
pole-cancellation directions.}

Order the target frame first by the two odd directions $\xi_+,\xi_-$ and
then by the two even directions $u_+,u_-$ together with the residual factors.
Order the source frame first by $\partial_{\varepsilon_+},\partial_{\varepsilon_-}$ and then by the
scale direction and the coordinates $\partial_{\btau_j}$,
$2\le j\le l-1$. The preceding calculations show that the modified
Kodaira--Spencer matrix extends holomorphically to the orbifold point
and that its value there is block lower triangular. Its odd diagonal
block is diagonal with two nonzero entries, while {its diagonal block corresponding to the residual and even directions} is the invertible evaluation matrix
computed above. Hence the matrix is invertible at the orbifold point.
After shrinking the chosen chart, it follows that
\[
  \widetilde{\mathrm{KS}}:
  T\widetilde{\mathcal U}_B\longrightarrow\cH_B
\]
is a holomorphic bundle isomorphism.
\end{proof}

\begin{corollary}[Extension of the Frobenius structure identification]\label{cor:frob-identification-extension}
The isomorphism induced by $\Phi$ extends to an isomorphism of Frobenius-algebra bundles
\[
  \Phi_{\rm ext}:
  \cH_B\xrightarrow{\ \sim\ }\Phi^*T\mathcal U_A.
\]
In particular, at the orbifold point there is an isomorphism of Frobenius algebras
\[
  \cH_B(\kappa^{\rm orb})\cong Z\bigl(\bC(\nu)[\Gamma]\bigr)
\]
with the metric in the threefold normalization.
\end{corollary}

\begin{proof}
Let
\[
  d\Phi:T\widetilde{\mathcal U}_B\longrightarrow \Phi^*T\mathcal U_A
\]
be the differential of the affine change of variables. By Remark~\ref{rem:hu-coordinate-check}, $d\Phi$ is a
holomorphic isomorphism on the same chosen chart. We define
\[
  \Phi_{\rm ext}=d\Phi\circ\widetilde{\mathrm{KS}}^{-1}:
  \cH_B\longrightarrow \Phi^*T\mathcal U_A .
\]
On $\Omega_{\rm Frob}$, this agrees with the identification induced by $\Phi$. The differences
\[
  \Phi_{\rm ext}(X\circ_B Y)-\Phi_{\rm ext}(X)\circ_A\Phi_{\rm ext}(Y),
  \qquad
  \Phi_{\rm ext}(\mathbf 1_B)-
  \mathbf 1_A
\]
are holomorphic tensors in the modified frame. They vanish on the nonempty open subset
$\Omega_{\rm Frob}$ of the chosen chart; by the identity theorem they vanish on the chosen chart.

Equip $\cH_B$ with the pullback metric
\[
  \eta^{B,\cX}_{\rm ext}=\Phi_{\rm ext}^*\eta^{\cX}.
\]
By Theorem~\ref{thm:frob-iso}, its restriction to $\Omega_{\rm Frob}$ is the B-model residue metric
$\eta^{B,\cX}$, so it is the unique holomorphic extension of that metric. Finally,
$\Phi(\kappa^{\rm orb}(K_0))=0$, and \eqref{eq:frob:aside-special-fiber} identifies
$T_0\mathcal U_A$ with $Z(\bC(\nu)[\Gamma])$ in the threefold normalization. Specializing
$\Phi_{\rm ext}$ at the orbifold point gives the asserted Frobenius-algebra isomorphism.
\end{proof}


\section{{The $R$-matrix comparison on $\Omega_B$}}
\label{sec:anchor}

{In this section, we identify the B-model calibration on {the domain $\Omega_B$} with the A-model
Givental--Teleman calibration. Coefficientwise boundary data at $\kappa^{\rm orb}$ determine the
residual gauge of {this comparison on $\Omega_B$}.}

{We first identify the stationary-phase matrix with the normalized B-model calibration and construct
the shifted A-model calibration formally. We compare the two matrices and remove the residual
diagonal gauge using {their common boundary value after contraction with the flat unit}. This also proves convergence of the
A-side coefficients in the base parameters.}

We keep the notation of Sections~\ref{sec:bgraphsum} and~\ref{sec:frobenius}. Thus
\[
  W=-\hat x=2\nu\log\lambda,\qquad u_\alpha=W(p_\alpha),\qquad
  U=\operatorname{diag}(u_\alpha),
\]
and $\hat R^B$ is in the normalized canonical convention
\[
  \hat R^B(0)=I,\qquad
  \hat R^B(z)\hat R^B(-z)^{\mathsf T}=I.
\]

\subsection{The B-model $R$-matrix on $\Omega_B$}

Definition~\ref{def:b-rmatrix} defines the B-model $R$-matrix $\hat R^B$ by regularized Laplace
transforms of the bare second-kind forms $\vartheta^0_\beta$. We rewrite that transform using the
evaluated Prym kernel, obtaining the matrix $Q$ below.

{Unless an analytic integral is explicitly specified, all Laplace transforms in this section
are understood as regularized formal stationary-phase expansions along the chosen Lefschetz
thimbles. Differentiation and integration by parts are interpreted coefficientwise in these
expansions. The chosen Morse coordinates, square-root branches, and contour orientations fix
the leading stationary-phase terms.}

Let
\[
  \widetilde u_\alpha:=\hat x(p_\alpha)=-u_\alpha,\qquad
  \widetilde U=\operatorname{diag}(\widetilde u_\alpha)=-U.
\]
Near $p_\alpha$ choose the Morse coordinate $s_\alpha$ of Section~\ref{sec:bgraphsum}, so that
\[
  \hat x=\widetilde u_\alpha+s_\alpha^2.
\]
We also write
\[
  \zeta_\alpha=\sqrt2\,s_\alpha,\qquad
  \hat x=\widetilde u_\alpha+{\zeta_\alpha^2\over2}.
\]
The evaluated Prym kernel is
\[
  \mathcal B_\beta(q)
  :=
  \left.
  {B^-(q,q')\over d\zeta_\beta(q')}
  \right|_{q'=p_\beta}.
\]
{
For $\alpha\ne\beta$ put
\[
  (\widetilde\Gamma)_{\alpha\beta}
  :=
  \left.
  {B^-(q,q')\over d\zeta_\alpha(q)d\zeta_\beta(q')}
  \right|_{(q,q')=(p_\alpha,p_\beta)},
  \qquad
  (\widetilde\Gamma)_{\alpha\alpha}:=0.
\]

\begin{lemma}[Logarithmic Rauch formula for the Prym kernel]\label{lem:logarithmic-rauch}
On $\Omega_B$, let $\partial_{\widetilde u_\gamma}$ denote variation of the critical value
$\widetilde u_\gamma$, with the other critical values fixed. On differentials, this derivative is
induced by the horizontal lift that keeps the local values of $\hat x$ at the observation points
fixed. Then
\begin{equation}
\label{eq:anchor:rauchs-formula}
  \partial_{\widetilde u_\gamma}B^-(q_1,q_2)
  =
  \Res_{p=p_\gamma}
  {B^-(q_1,p)B^-(q_2,p)\over d\hat x(p)}.
\end{equation}
The residue is taken at the chosen representative $p_\gamma$ of the critical orbit, with the
convention $B^-=B^{\rm std}-\iota_2^*B^{\rm std}$. Moreover, for $\gamma\ne\beta$,
\begin{equation}
\label{eq:anchor:kernel-variation}
  \partial_{\widetilde u_\gamma}\mathcal B_\beta
  =
  (\widetilde\Gamma)_{\gamma\beta}\mathcal B_\gamma.
\end{equation}
\end{lemma}

\begin{proof}
Apply the ordinary Rauch formula to the meromorphic covering $\lambda$; see
\cite[Thm.~3.1]{kokotov2008hierarchyhurwitz} or
\cite[Sec.~3.3, Eqs.~(18)--(19)]{shramchenko2008riemannhilbertirregular}.
On $\Omega_B$, the poles of $\lambda$ have fixed orders, and its finite ramification points
$p_{\gamma,+},p_{\gamma,-}$ are simple. Put
\[
  v_\gamma=\lambda(p_{\gamma,+})=\lambda(p_{\gamma,-}),
  \qquad
  \widetilde u_\gamma=-2\nu\log v_\gamma.
\]
Locally, the branch values of the separately labelled simple ramification points are deformation
coordinates for the covering, even when some branch values agree. Restrict the ordinary Rauch
formula to the family preserving $\iota$, where the two branch values in each orbit vary together.
Since $\partial_{\widetilde u_\gamma}v_\gamma=-v_\gamma/(2\nu)$, this gives
\[
  \partial_{\widetilde u_\gamma}B^{\rm std}(q_1,q_2)
  =-{v_\gamma\over2\nu}
  \sum_{\epsilon=\pm}\Res_{p=p_{\gamma,\epsilon}}
  {B^{\rm std}(q_1,p)B^{\rm std}(q_2,p)\over d\lambda(p)}.
\]
Here keeping $\lambda$ fixed is equivalent to keeping a local branch of $\hat x$ fixed.
Changing the logarithm branch adds a locally constant multiple of $-4\pi i\nu$ to $\hat x$
and does not change the lift. At either critical point, $\lambda(p)-v_\gamma$ vanishes to
second order, so replacing $v_\gamma$ by $\lambda(p)$ inside the residue does not change it.
The identity
\[
  {1\over d\hat x}=-{\lambda\over2\nu\,d\lambda}
\]
therefore yields
\[
  \partial_{\widetilde u_\gamma}B^{\rm std}(q_1,q_2)
  =\sum_{\epsilon=\pm}\Res_{p=p_{\gamma,\epsilon}}
  {B^{\rm std}(q_1,p)B^{\rm std}(q_2,p)\over d\hat x(p)}.
\]

The involution is independent of the base and preserves $\hat x$ locally, so the horizontal
lift commutes with $\iota$. Thus we may apply $1-\iota_2^*$ to this identity.
Pull the residue at $p_{\gamma,-}$ back to the disk about $p_{\gamma,+}$ using $\iota$;
all differentials at the partner point in the following display are understood by this pullback.
For a non-fixed orbit $\{p,\iota p\}$, the sum of the two cover residues after anti-invariant
projection has numerator
\[
\begin{aligned}
  &B^{\rm std}(q_1,p)\bigl(B^{\rm std}(q_2,p)-B^{\rm std}(q_2,\iota p)\bigr)  \\
  &\quad
    +B^{\rm std}(q_1,\iota p)\bigl(B^{\rm std}(q_2,\iota p)-B^{\rm std}(q_2,p)\bigr)  \\
  &\hspace{2cm}
  =
  \bigl(B^{\rm std}(q_1,p)-B^{\rm std}(q_1,\iota p)\bigr)
  \bigl(B^{\rm std}(q_2,p)-B^{\rm std}(q_2,\iota p)\bigr).
\end{aligned}
\]
The product on the right is $B^-(q_1,p)B^-(q_2,p)$. Since $\iota^*d\hat x=d\hat x$,
the two residues give \eqref{eq:anchor:rauchs-formula} with no additional factor.

For $\gamma\ne\beta$, the relation
$\hat x=\widetilde u_\beta+\zeta_\beta^2/2$ shows that the horizontal lift preserves
$\zeta_\beta$. It therefore preserves the section $p_\beta=\{\zeta_\beta=0\}$ and
$d\zeta_\beta|_{p_\beta}$. Evaluation in the second variable commutes with differentiation.
Evaluating \eqref{eq:anchor:rauchs-formula} at $(p_\beta,d\zeta_\beta)$ and using
$d\hat x=\zeta_\gamma d\zeta_\gamma$ at $p_\gamma$ proves
\eqref{eq:anchor:kernel-variation}.
\end{proof}
}

{We use $a$ for a flat-coordinate index. Let $T_a$ be the corresponding flat cohomology class
and $\partial_{x_a}$ the corresponding flat vector field under the mirror isomorphism of
Theorem~\ref{thm:frob-iso}; in particular,
$\partial_{x_{\mathrm{unit}}}=\partial_{x_{l+1}}$.}

\begin{proposition}[The B-model $R$-matrix on $\Omega_B$]\label{prop:smooth-chamber-calibration}
On $\Omega_B$, let $Q(z)$ denote the regularized formal stationary-phase expansion
\[
  Q_{\beta\alpha}(z)
  \sim
  -{\sqrt z\over\sqrt{2\pi}}
  \int_{\Gamma_\alpha}^{\rm reg}
  e^{-(\hat x-\widetilde u_\alpha)/z}\mathcal B_\beta,
  \qquad z\to0,
\]
in the chosen sector for $\Gamma_\alpha$. Then $Q=\hat R^B$ as formal series. The matrix
$Q$ satisfies $Q(0)=I$,
\[
  Q(z)Q(-z)^{\mathsf T}=I,
\]
{Let $\eta_{ab}=\eta^{B,\cX}(\partial_{x_a},\partial_{x_b})$
and let $(\eta^{ab})$ be its inverse. Define the transition matrix by
\[
  (\Psi^B)^a{}_{\alpha}
  =\sum_b\eta^{ab}\,
    \eta^{B,\cX}(\partial_{x_b},\sqrt{\Delta^\cX_\alpha}e_\alpha).
\]
Its $\alpha$-th column gives the components of
$\widetilde e_\alpha=\sqrt{\Delta^\cX_\alpha}e_\alpha$ in the flat basis
$\{\partial_{x_a}\}$. With multiplication matrices acting on column vectors,}
\[
  S^B:=\Psi^BQe^{U/z}
\]
satisfies the B-model Dubrovin equation
\[
  z\,dS^B=C^B S^B.
\]
Consequently $\hat R^B$ is the normalized canonical calibration of $\cM^B|_{\Omega_B}$.
\end{proposition}

\begin{proof}
{The proof follows Shramchenko's construction of the fundamental solution on Hurwitz spaces
\cite[Sec.~4.2, Thm.~1]{shramchenko2008riemannhilbertirregular}, in the $R$-matrix formulation
of \cite[Thm.~7 and the proof of Thm.~9]{duninbarkowski2016primaryinvariantshurwitz}, and extends
it to the logarithmic superpotential and the Prym kernel. Lemma~\ref{lem:logarithmic-rauch}
supplies the variational identity. We first use it to derive the differential equation for $Q$.
We then follow the standard variation argument for primary differentials in
\cite[Sec.~4.3]{duninbarkowski2016primaryinvariantshurwitz}. For the present family, we express
the B-model primary differentials as periods of the Prym kernel, verify that the moving endpoints
contribute no additional terms, and check that their critical-point evaluations give the B-model
transition matrix with the prescribed normalization. Together these identities give the Dubrovin equation. Normalization and symplecticity follow from stationary phase
and Proposition~\ref{prop:prym-double-laplace}.}

The relation with Definition~\ref{def:b-rmatrix} is local. Write, near $q'=p_\beta$,
\[
  B^-(q,q')=\bigl(b_0(q)+O(s_\beta(q'))\bigr)\,ds_\beta(q').
\]
By the residue definition of the bare primary second-kind form,
\[
  \vartheta^0_\beta(q)
  =
  -\Res_{q'=p_\beta} B^-(q,q')\,s_\beta(q')^{-1}
  =-b_0(q),
\]
whereas $d\zeta_\beta=\sqrt2\,ds_\beta$, so
\[
  \mathcal B_\beta={b_0\over\sqrt2},
  \qquad
  \vartheta^0_\beta=-\sqrt2\,\mathcal B_\beta.
\]
Substituting this into Definition~\ref{def:b-rmatrix} identifies its regularized stationary-phase
expansion with the displayed formal matrix $Q$. We differentiate this expansion coefficientwise.

We now apply Lemma~\ref{lem:logarithmic-rauch} to prove the differential equation.
The matrix $\widetilde\Gamma$ is symmetric, since the invariance of $B^{\rm std}$ under the
simultaneous action of $\iota$ in both variables implies that $B^-$ is symmetric.
{
On {the domain $\Omega_B$}, the thimble class $[\Gamma_\alpha]$ is flat for the Gauss--Manin connection.
We differentiate at fixed $\hat x$. This calculation can be performed in the local Morse coordinate:
each stationary-phase coefficient is determined by finitely many Laurent coefficients of the
integrand, and the change of coordinate under differentiation is handled by formal integration
by parts. For the second-kind pole, we use the regularization
\eqref{eq:bgraphsum:regularized-primary-transform}. Thus the differentiated transform is the
formal transform of the horizontal derivative of the integrand.
Using \eqref{eq:anchor:kernel-variation} gives
}
\begin{equation}
\label{eq:anchor:q-directional-equation}
  \partial_{\widetilde u_\gamma}Q_{\beta\alpha}
  =
  {\delta_{\gamma\alpha}\over z}Q_{\beta\alpha}
  +(\widetilde\Gamma)_{\gamma\beta}Q_{\gamma\alpha},
  \qquad \gamma\ne\beta.
\end{equation}
{
The remaining equation follows from the scale, or identity, direction
$\sum_\gamma\partial_{\widetilde u_\gamma}$. Its lift changes $\hat x$ and every
$\widetilde u_\alpha$ by the same locally constant amount. It therefore leaves
$\hat x-\widetilde u_\alpha$, the Morse coordinates, $B^-$, $\mathcal B_\beta$, and the horizontal
thimble classes unchanged. Hence
$\sum_\gamma\partial_{\widetilde u_\gamma}Q=0$.
}
{Together with}
\eqref{eq:anchor:q-directional-equation} this gives
\begin{equation}
\label{eq:anchor:q-differential-equation}
  dQ={[Q,d\widetilde U]\over z}+[\widetilde\Gamma,d\widetilde U]Q.
\end{equation}

We translate this to the convention $W=-\hat x$ of Section~\ref{sec:frobenius} by identifying the
rotation coefficients of the B-model primary differentials. For a flat coordinate $x_a$, set
\[
  \omega_a:=c_a^B\phi_D,
  \qquad
  c_a^B=2\nu\partial_{x_a}\log\lambda.
\]
We claim that each $\omega_a$ is a period of the Prym kernel over an
anti-invariant relative cycle that is flat for the Gauss--Manin connection. We first check this in the logarithmic coordinate basis
$\btau_i=\log_{\rm orb}\kappa_i$ of Definition~\ref{def:mirror-map}. Recall that
\[
  \lambda(w)=K{\prod_i(w-c_i)\over w^2-4},
  \qquad
  c_i=\kappa_i+\kappa_i^{-1}.
\]
Then
\[
  \partial_{\btau_i}\log\lambda
  ={ \kappa_i^{-1}-\kappa_i \over w-c_i},
  \qquad
  w-c_i={(\mu-\kappa_i)(\mu-\kappa_i^{-1})\over \mu},
\]
and therefore
\[
\begin{aligned}
  \omega_{\btau_i}
  &=
  2\nu\sqrt\nu\,{\kappa_i^{-1}-\kappa_i \over w-c_i}{d\mu\over \mu}  \\
  &=-2\nu\sqrt\nu
  \left(
    {d\mu\over \mu-\kappa_i}
    -
    {d\mu\over \mu-\kappa_i^{-1}}
  \right).
\end{aligned}
\]
Start with any oriented relative path $\mathcal C_i^0$ from $\kappa_i^{-1}$ to $\kappa_i$ and define
{the relative chain with coefficients in the sign local system}
\[
  \mathcal C_i:={1\over2}\bigl(\mathcal C_i^0-\iota_*\mathcal C_i^0\bigr).
\]
Then $\partial\mathcal C_i=[\kappa_i]-[\kappa_i^{-1}]$ and
$\iota_*\mathcal C_i=-\mathcal C_i$. Since $B^-$ is anti-invariant in its integration variable,
integrating over $\mathcal C_i$ gives the same result as integrating over $\mathcal C_i^0$. Direct
integration
of the Prym kernel gives
\[
  \int_{\mathcal C_i}B^-
  =
  2\left(
    {d\mu\over \mu-\kappa_i}
    -
    {d\mu\over \mu-\kappa_i^{-1}}
  \right),
\]
so
\[
  \omega_{\btau_i}=-\nu\sqrt\nu\int_{\mathcal C_i}B^-.
\]
For the scale direction $e=\partial_{x_{l+1}}$, one has $c_e^B=1$ and
$\omega_e=\phi_D=\sqrt\nu\,d\mu/\mu$. If $\mathcal C_0$ is an oriented path from $0$ to $\infty$, then
\[
  \int_{\mathcal C_0}B^-=-2{d\mu\over \mu},
  \qquad
  \omega_e=-{\sqrt\nu\over2}\int_{\mathcal C_0}B^-.
\]
The flat $x_a$-basis is obtained from the $\btau_i$-basis and the scale direction by the constant
mirror-map matrix of Section~\ref{sec:frobenius}. {We use the following observation about relative cycles to
differentiate these kernel periods.} Let $Z=Z_0\cup Z_\infty$ be the divisor of the moving zeros
and poles of $\lambda$ on the compactified $\mu$-cover.
We work with the anti-invariant part of the relative homology group
\[
  H_1(\bar C_\kappa,Z;\bC)^- .
\]
The classes $\mathcal C_i$ and $\mathcal C_0$ above are chosen as relative classes with
\[
  \partial\mathcal C_i=\kappa_i-\kappa_i^{-1},
  \qquad
  \partial\mathcal C_0=\infty-0,
\]
and with
\[
  \iota_*\mathcal C_i=-\mathcal C_i,
  \qquad
  \iota_*\mathcal C_0=-\mathcal C_0.
\]
We check the behavior at the zeros and poles of $\lambda$. At a point $p_0$ of this divisor,
choose a local coordinate $\tau_{p_0}$ with $\tau_{p_0}(p_0)=0$. Then
\[
  \lambda=\lambda_{p_0}\tau_{p_0}^{m_{p_0}}(1+O(\tau_{p_0})),
  \qquad
  \lambda_{p_0}\ne0,
  \qquad
  m_{p_0}\in\mathbb Z\setminus\{0\},
\]
and
\[
  d\hat x=-2\nu {d\lambda\over\lambda}
  =-2\nu\left(m_{p_0}{d\tau_{p_0}\over\tau_{p_0}}+O(d\tau_{p_0})\right),
\]
and therefore $1/d\hat x=O(\tau_{p_0}/d\tau_{p_0})$. For $q_1,q_2$ away from the puncture,
$B^-(q_i,q)=O(d\tau_{p_0})$, so
\[
  {B^-(q_1,q)B^-(q_2,q)\over d\hat x(q)}=O(\tau_{p_0})\,d\tau_{p_0}
\]
has zero residue.

{These relative classes are flat for the fixed-$\hat x$ Gauss--Manin connection. Indeed, if
$\tau_{p_0}$ is the puncture coordinate used above and $\widetilde X$ is the horizontal lift of a base
vector field, then
\[
  X(\hat x)+\iota_{\widetilde X}d\hat x=0,
  \qquad
  \widetilde X(\tau_{p_0})=O(\tau_{p_0}).
\]
After truncating a relative path at $|\tau_{p_0}|=\epsilon_0$, its endpoint term is therefore
$O(\epsilon_0)$ and vanishes in the limit. The puncture terms in the differentiated kernel period
vanish by the local puncture estimate above:
$B^-=O(d\tau_{p_0})$ and $1/d\hat x=O(\tau_{p_0}/d\tau_{p_0})$. If one kernel factor is integrated over
a relative cycle, the explicit primary differentials above have at worst a logarithmic pole, so the
resulting quotient is $O(d\tau_{p_0})$, so it is holomorphic at the puncture and has zero residue.}

Hence every $\omega_a$ has the form
\[
  \omega_a(q)=\int_{\mathcal C_a}B^-(q,p)
\]
for an anti-invariant relative cycle $\mathcal C_a$ that is flat for the Gauss--Manin connection.

We may now integrate the Rauch formula over $\mathcal C_a$. Since the cycle is flat,
\[
\begin{aligned}
  \nabla_{\partial_{\widetilde u_\gamma}}\omega_a(q)
  &=
  \int_{\mathcal C_a}
  \nabla_{\partial_{\widetilde u_\gamma}}B^-(q,p) \\
  &=
  \Res_{p'=p_\gamma}
  {B^-(q,p')\int_{\mathcal C_a}B^-(p,p')\over d\hat x(p')} \\
  &=
  \Res_{p'=p_\gamma}
  {B^-(q,p')\omega_a(p')\over d\hat x(p')}.
\end{aligned}
\]
The puncture terms vanish by the calculation above.

Near $p_\gamma$ write
\[
  B^-(q,p')=\bigl(\mathcal B_\gamma(q)+O(\zeta_\gamma(p'))\bigr)d\zeta_\gamma(p'),
  \qquad
  \omega_a(p')=\bigl(\omega_{a,\gamma}+O(\zeta_\gamma(p'))\bigr)d\zeta_\gamma(p'),
\]
where
\[
  \omega_{a,\gamma}:=
  \left.{\omega_a\over d\zeta_\gamma}\right|_{p_\gamma}.
\]
Since $d\hat x=\zeta_\gamma d\zeta_\gamma$, the residue gives
\[
  \nabla_{\partial_{\widetilde u_\gamma}}\omega_a(q)
  =
  \mathcal B_\gamma(q)\omega_{a,\gamma}.
\]
Evaluating at $p_\beta$, for $\beta\ne\gamma$, gives
\[
  \partial_{\widetilde u_\gamma}\omega_{a,\beta}
  =
  (\widetilde\Gamma)_{\beta\gamma}\omega_{a,\gamma}.
\]
As before, simultaneously translating every critical
value leaves $d\hat x$, the Prym kernel, and the flat relative cycles unchanged, so
\[
  \sum_\gamma\partial_{\widetilde u_\gamma}\omega_{a,\beta}=0.
\]
{Together, the off-diagonal equations and the identity above determine all derivatives of the
critical-point values $\omega_{a,\beta}$. We now express these equations in terms of the normalized
canonical transition matrix.}
Fix a square root $\mathfrak g_\nu$ of $8\nu^3$ and put
\[
  \Bbbk_\nu:=\bC(\nu,\mathfrak g_\nu).
\]
All normalized canonical frames in this section are taken after scalar extension to $\Bbbk_\nu$;
$\mathfrak g_\nu$ normalizes the primary differentials below. If
\[
  a_\beta=
  \left.{d\log\mu\over ds_\beta}\right|_{p_\beta},
\]
then $d\zeta_\beta=\sqrt2\,ds_\beta$ and
\[
  \left.{\phi_D\over d\zeta_\beta}\right|_{p_\beta}
  =
  {\sqrt\nu\over\sqrt2}a_\beta.
\]
The residue normalization of Section~\ref{sec:frobenius} is
\[
  {1\over \Delta_\beta^\cX}={a_\beta^2\over8\nu^2}.
\]
With the chosen square roots this gives
\begin{equation}
\label{eq:anchor:omega-normalization}
  {\sqrt2\over\mathfrak g_\nu}
  \left.{\phi_D\over d\zeta_\beta}\right|_{p_\beta}
  =
  {1\over\sqrt{\Delta_\beta^\cX}}.
\end{equation}
Therefore
\[
  {\sqrt2\over \mathfrak g_\nu}\omega_{a,\beta}
  =
  {c_a^B(p_\beta)\over\sqrt{\Delta_\beta^\cX}}.
\]
{The primary evaluations give the lowered transition matrix
$P^B=\eta\Psi^B$, whose entries are
\[
  P^B_{a\beta}
  =\eta^{B,\cX}(\partial_{x_a},\widetilde e_\beta)
  =\frac{c_a^B(p_\beta)}{\sqrt{\Delta^\cX_\beta}}
  =\frac{\sqrt2}{\mathfrak g_\nu}\omega_{a,\beta}.
\]
The variation equations above therefore give
\[
  d(P^B)^{\mathsf T}
  =[\widetilde\Gamma,d\widetilde U](P^B)^{\mathsf T}.
\]
Since the flat metric $\eta$ is constant and invertible, substituting
$(P^B)^{\mathsf T}=(\Psi^B)^{\mathsf T}\eta$ and canceling $\eta$ yields
\begin{equation}
\label{eq:anchor:flat-primary-variation}
  d(\Psi^B)^{\mathsf T}
  =[\widetilde\Gamma,d\widetilde U](\Psi^B)^{\mathsf T}.
\end{equation}
}
Set
\[
  \Gamma^B=-\widetilde\Gamma.
\]
Since $\widetilde U=-U$, transposing \eqref{eq:anchor:flat-primary-variation} gives
\begin{equation}
\label{eq:anchor:psi-variation}
  d\Psi^B=-\Psi^B[\Gamma^B,dU],
\end{equation}
and \eqref{eq:anchor:q-differential-equation} becomes
\begin{equation}
\label{eq:anchor:q-paper-equation}
  dQ={[dU,Q]\over z}+[\Gamma^B,dU]Q.
\end{equation}
Using \eqref{eq:anchor:psi-variation} and \eqref{eq:anchor:q-paper-equation}, we obtain
\[
\begin{aligned}
  z\,d(\Psi^BQ)
  &=
  z\,d\Psi^B\,Q+z\,\Psi^B\,dQ  \\
  &=
  -z\Psi^B[\Gamma^B,dU]Q
  +\Psi^B[dU,Q]
  +z\Psi^B[\Gamma^B,dU]Q   \\
  &=\Psi^B[dU,Q].
\end{aligned}
\]
Therefore
\[
\begin{aligned}
  z\,d(\Psi^BQe^{U/z})
  &=
  \Psi^B[dU,Q]e^{U/z}+\Psi^BQ\,dU\,e^{U/z}  \\
  &=
  \Psi^B dU\,Qe^{U/z}.
\end{aligned}
\]
{The columns of $\Psi^B$ are the normalized canonical vectors
in the flat basis, so $C^B\Psi^B=\Psi^B dU$. Hence}
\[
  z\,dS^B=C^B S^B.
\]

Finally, the local pole
\[
  \vartheta^0_\alpha=-{ds_\alpha\over s_\alpha^2}+O(1)\,ds_\alpha
\]
gives $Q(0)=I$. Proposition~\ref{prop:prym-double-laplace} gives the unitarity
$Q(z)Q(-z)^{\mathsf T}=I$. Thus $Q$ is the normalized canonical calibration, and by the definition of
$\hat R^B$ it is $\hat R^B$.
\end{proof}

\begin{proposition}[{B-model unit contraction}]\label{prop:prym-unit-translation}
Put
\[
  v_\beta={1\over\sqrt{\Delta^\cX_\beta}}
  ={\check h^\beta_1\over\sqrt{-2}},
  \qquad
  L^B_\alpha(z):=\sum_\beta v_\beta\,\hat R^B_{\beta\alpha}(z).
\]
Then
\[
  L^B_\alpha(z)
  ={1\over\sqrt{-2}}\sum_{k\ge1}\check h^\alpha_k z^{k-1}.
\]
Equivalently,
\[
  \check h^\alpha_k
  =
  [z^{k-1}]
  \sum_\beta \check h^\beta_1\hat R^B_{\beta\alpha}(z).
\]
\end{proposition}

\begin{proof}
The normalized unit primary differential satisfies
\[
  {\phi_D\over\mathfrak g_\nu}
  ={d\widetilde y\over\sqrt{-2}}.
\]
Put $\omega_e=\phi_D$.
We first prove the Prym Cauchy identity
\begin{equation}
\label{eq:anchor:prym-cauchy-unit}
  \sum_{\beta}
  \omega_{e,\beta}\mathcal B_\beta(q)
  =
  -d_q\left({\omega_e\over d\hat x}\right)(q).
\end{equation}
Here $\omega_{e,\beta}$ is defined by the expansion
\[
  \omega_e(p)=\bigl(\omega_{e,\beta}+O(\zeta_\beta)\bigr)d\zeta_\beta
\]
at a lift of the critical point $p_\beta$. Let $q$ be away from the critical points and logarithmic
punctures. Apply the global residue theorem in the $p$-variable on the compactified $\mu$-cover to
\[
  \Xi_q(p):=B^-(q,p){\omega_e(p)\over d\hat x(p)}
  ={B^-(q,p)\omega_e(p)\over d\hat x(p)}.
\]
Near a lifted critical point $p_\beta$ we have
\[
  d\hat x=\zeta_\beta\,d\zeta_\beta,
  \qquad
  B^-(q,p)=\bigl(\mathcal B_\beta(q)+O(\zeta_\beta)\bigr)d\zeta_\beta,
\]
and therefore
\[
  \Res_{p=p_\beta}\Xi_q(p)
  =
  \omega_{e,\beta}\mathcal B_\beta(q).
\]
The product $\Xi_q$ is invariant in the $p$-variable under $\iota$: both $B^-(q,p)$ and
$\omega_e(p)/d\hat x(p)$ are anti-invariant. Thus the two lifts of each quotient critical point have
equal residues, and the sum of the critical-point residues on the cover is
\[
  2\sum_\beta \omega_{e,\beta}\mathcal B_\beta(q).
\]
The only remaining finite poles are $p=q$ and $p=\iota(q)$. Since
$\omega_e/d\hat x$ is anti-invariant, the residues at these two Cauchy poles both equal
$d_q(\omega_e/d\hat x)(q)$. {The residues at the logarithmic punctures vanish} by the same estimate as in the
proof of Proposition~\ref{prop:smooth-chamber-calibration}. Hence the global residue theorem gives
\[
  0=
  2\sum_\beta \omega_{e,\beta}\mathcal B_\beta(q)
  +2d_q\left({\omega_e\over d\hat x}\right)(q).
\]

Applying \eqref{eq:anchor:omega-normalization} to the unit direction $e$, where $c_e^B=1$, gives
\[
  v_\beta={\sqrt2\over\mathfrak g_\nu}\omega_{e,\beta}.
\]
Using Proposition~\ref{prop:smooth-chamber-calibration} and \eqref{eq:anchor:prym-cauchy-unit}, we
obtain
\[
\begin{aligned}
  L^B_\alpha(z)
  &=
  -{\sqrt z\over\mathfrak g_\nu\sqrt\pi}\,
  e^{-u_\alpha/z}
  \int_{\Gamma_\alpha}^{\rm reg}
  e^{-\hat x/z}
  \sum_\beta \omega_{e,\beta}\mathcal B_\beta   \\
  &=
  {\sqrt z\over\mathfrak g_\nu\sqrt\pi}\,
  e^{-u_\alpha/z}
  \int_{\Gamma_\alpha}^{\rm reg}
  e^{-\hat x/z}\,d\left({\omega_e\over d\hat x}\right).
\end{aligned}
\]
The primitive in the regularization is $\omega_e/d\hat x$, so integration by parts in the sense of
Definition~\ref{def:b-rmatrix} gives
\[
  \int_{\Gamma_\alpha}^{\rm reg}
  e^{-\hat x/z}\,d\left({\omega_e\over d\hat x}\right)
  =
  {1\over z}
  \int_{\Gamma_\alpha}e^{-\hat x/z}\omega_e .
\]
Since $\omega_e/\mathfrak g_\nu=d\widetilde y/\sqrt{-2}$, it follows that
\begin{equation}
\label{eq:anchor:unit-row-integral}
  L^B_\alpha(z)
  =
  {1\over\sqrt{-2}}\,
  {e^{-u_\alpha/z}\over\sqrt{\pi z}}
  \int_{\Gamma_\alpha} e^{-\hat x/z}\,d\widetilde y.
\end{equation}
Write $\hat x=-u_\alpha+s_\alpha^2$. Formal Gaussian integration by parts gives
$\int e^{-s^2/z}(f'(s)-2sf(s)/z)\,ds=0$ coefficientwise for a Taylor series $f$.
Applying this identity to $f=\widetilde y$ gives
\[
  0=\int_{\Gamma_\alpha} d(e^{-\hat x/z}\widetilde y)
  =
  \int_{\Gamma_\alpha}e^{-\hat x/z}d\widetilde y
  -{1\over z}\int_{\Gamma_\alpha}e^{-\hat x/z}\widetilde y\,d\hat x.
\]
Only the odd part of $\widetilde y$ in the Morse coordinate contributes, and by
\eqref{eq:bgraphsum:dilaton-coefficients},
\[
  \widetilde y(s)-\widetilde y(-s)
  =
  2\sum_{k\ge1}{2^{k-1}\over(2k-1)!!}\check h^\alpha_k s^{2k-1}.
\]
Since $d\hat x=2s\,ds$,
\[
  {e^{-u_\alpha/z}\over\sqrt{\pi z}}\,{1\over z}
  \int_{\Gamma_\alpha}
  e^{-\hat x/z}
  {2^{k-1}\over(2k-1)!!}s^{2k-1}\,d\hat x
  =
  z^{k-1}.
\]
Substituting this moment calculation into \eqref{eq:anchor:unit-row-integral} gives the displayed
series for $L^B_\alpha$.
\end{proof}

\begin{remark}[{{Contraction with the flat unit}}]
{The term ``unit contraction'' comes from the following interpretation. Since the normalized
idempotents are orthonormal and $\mathbf 1=\sum_\beta e_\beta$, we have
\[
  \eta^{B,\cX}(\widetilde e_\beta,\mathbf 1)
  ={1\over\sqrt{\Delta^\cX_\beta}}
  ={\check h^\beta_1\over\sqrt{-2}}=v_\beta.
\]
Thus $L^B_\alpha=\eta^{B,\cX}(\hat R^B\widetilde e_\alpha,\mathbf 1)$: it contracts the row index
and leaves the column index $\alpha$ free.}
\end{remark}

\subsection{The A-side formal calibration}

{We construct the calibration of the shifted A-side cohomological field theory (CohFT) as a
formal series in the primary parameters, with its prescribed value at the orbifold point.
We then place it and $\hat R^B$ over a common ring. The later comparison will prove convergence
in the base parameters, coefficientwise in $z$.}

Let
\[
  A=\widehat{\mathcal O}_{\mathcal U_A,0}
  \cong \Bbbk_\nu[[x_1,\ldots,x_{l+1}]]
\]
be the completed local ring of the A-side primary parameter space at the orbifold point, and let
$A^{\rm can}$ be the finite \'etale canonical square-root extension of $A$ obtained by adjoining the
chosen square roots $\sqrt{\Delta_\alpha(x)}$, normalized to take their prescribed values in
$\Bbbk_\nu$ at the chosen point over $x=0$. Let $\mathcal O^{\rm can}_{B,0}$
be the convergent B-side local ring for the corresponding B-side branch choices. By
Definition~\ref{def:mirror-map}, the centered mirror map is a holomorphic local coordinate
isomorphism, so it identifies the completion of $\mathcal O^{\rm can}_{B,0}$ with $A^{\rm can}$.

{The following constructions serve three purposes: finite extensions make the chosen branches
single-valued, completion allows formal series in the base parameters, and localization permits
the denominators occurring in stationary phase.}

\begin{definition}[The localized completed ring $\mathscr K$]
\label{def:ring-K}
Let $\mathcal O'$ be a finite normal analytic extension of $\mathcal O^{\rm can}_{B,0}$ on which all chosen
{critical points, canonical square roots, and chosen Morse coordinates are single-valued.} We replace
$\mathcal O'$ by the local normal analytic component containing the chosen lift of
$\Omega_B$. Thus $\mathcal O'$ is a local normal domain. Let $\widehat{\mathcal O}'$ be its completion
at the orbifold point; since analytic normal local rings are excellent, this completion is
again a domain. Let $\mathfrak s$ be the product of the local equations for the branches of the
pole-cancellation locus and {the factors of the discriminant of the critical-point polynomial {inverted on the domain $\Omega_B$}}.
It is a non-zero-divisor. Define
\[
  \mathscr K=\widehat{\mathcal O}'[\mathfrak s^{-1}].
\]
\end{definition}

{Thus $\mathcal O'$ consists of convergent germs on the chosen cover,
$\widehat{\mathcal O}'$ allows formal germs, and $\mathscr K$ also allows powers of
$\mathfrak s$ in denominators. Since $\widehat{\mathcal O}'$ is a domain and $\mathfrak s\ne0$,
$\mathscr K$ embeds in its fraction field.}

\begin{proposition}[A-side formal calibration and $\hat R^B$ over $\mathscr K$]
\label{prop:aside-chamber-calibration}
Over $A^{\rm can}$, the shifted
$\bT$-equivariant CohFT of $\cX$ has a normalized canonical calibration
\[
  \hat R_A^\cX(z)
  \in
  I+z\operatorname{End}(H_{\CR,\bT}^*(\cX))\otimes A^{\rm can}[[z]].
\]
Its specialization at the closed point is
\[
  \operatorname{sp}_0(\hat R_A^\cX(z))
  =
  \hat R^\cX_{\rm orb}(z)
  =
  E(z)\hat R^{A,\mathrm{surf}}(z).
\]
Each coefficient of the B-model $R$-matrix $\hat R^B$ lies in $\mathcal O'[\mathfrak s^{-1}]$
and hence $\hat R^B$ defines an element of
\[
  I+z\operatorname{End}(H_{\CR,\bT}^*(\cX))\otimes\mathscr K[[z]].
\]
The pullback of $\hat R_A^\cX$ to $\mathscr K$ is the A-side formal normalized canonical solution
for the same Frobenius structure.
\end{proposition}

\begin{proof}
Let
\[
  \mathbf x=\sum_a x_aT_a\in A\otimes H_{\CR,\bT}^*(\cX).
\]
We define the $A$-valued shifted CohFT by
\begin{equation}
\label{eq:anchor:shifted-cohft}
  \Omega^A_{g,n}(\gamma_1,\ldots,\gamma_n)
  =
  \sum_{m\geq0}{1\over m!}
  (\pi_m)_*\Omega_{g,n+m}(\gamma_1,\ldots,\gamma_n,\mathbf x,\ldots,\mathbf x).
\end{equation}
Coefficientwise in $A$, only finitely many terms in the sum contribute. Its genus-zero product
is the formal quantum product at the universal point $\mathbf x$.

At $x=0$ the Frobenius algebra is the scalar extension of the semisimple class algebra,
\[
  Z(\Bbbk_\nu[\Gamma])=\Bbbk_\nu\otimes_{\bC(\nu)}Z(\bC(\nu)[\Gamma]),
\]
by
Lemma~\ref{lem:frob}. Its discriminant is a unit in $A$, since its value at $x=0$ is nonzero. Therefore the Frobenius algebra
\[
  H_A=A\otimes H_{\CR,\bT}^*(\cX)
\]
is finite \'etale over $A$, and the primitive idempotents lift uniquely from the orbifold point:
\[
  e_\alpha(x)\in H_A,
  \qquad
  e_\alpha(0)=e_\alpha^{\rm orb}.
\]
Their norms are units. The chosen square roots in $A^{\rm can}$ give the normalized canonical basis
\[
  \widetilde e_\alpha(x)=\sqrt{\Delta_\alpha(x)}e_\alpha(x).
\]

{
Teleman's reconstruction theorem is stated over a characteristic-zero field
\cite{teleman2012structure2dsemisimplefield}; here we only need its coefficientwise extension to the
complete local ring $A^{\rm can}$. Fix $M,N$ and work modulo $\mathfrak m^{M+1}$ and $z^{N+1}$. At this finite
order, reconstruction of the shifted CohFT uses finitely many CohFT tensors and algebraic operations,
with denominators only among the idempotent norms and the discriminant, which are units over
$A^{\rm can}$. It therefore gives the same unique calibration over
$A^{\rm can}/\mathfrak m^{M+1}$. Compatibility in $M$ and $N$ yields
$\hat R_A^\cX(z)$ by inverse limit, compatibly with continuous local base change.}
Let $\operatorname{sp}_0:A^{\rm can}\to \Bbbk_\nu$ be evaluation at the chosen point over $x=0$.
Then
$\operatorname{sp}_0(\hat R_A^\cX)$ is the reconstruction matrix of the unshifted orbifold theory.
Section~\ref{sec:amodel}, namely Proposition~\ref{prop:rmatrix}
together with Corollary~\ref{cor:thirdleg},
computes this normalized canonical matrix as
\[
  \operatorname{sp}_0(\hat R_A^\cX(z))=E(z)\hat R^{A,\mathrm{surf}}(z).
\]

It remains to verify that $\hat R^B$ has coefficients in $\mathscr K$. On
$\operatorname{Spec}\mathcal O'[\mathfrak s^{-1}]$ the critical points, Hessians, canonical norms, and
Morse jets are single-valued analytic functions, with inverse Hessians regular after inverting
$\mathfrak s$.

{Each coefficient of the stationary-phase expansion of the regularized Laplace integral defining
$\hat R^B$ is a universal polynomial in finitely many jets of the phase and of the evaluated
Prym-kernel amplitude $\mathcal B_\beta$, with powers of the inverse Hessians as denominators. On the
chosen finite analytic cover these amplitude jets are analytic along the indexed critical sections;
after inverting $\mathfrak s$, the phase jets and inverse Hessians are regular as well.}
Therefore these coefficients lie in $\mathcal O'[\mathfrak s^{-1}]$, which maps into
\[
  \mathscr K=\widehat{\mathcal O}'[\mathfrak s^{-1}].
\]
By Theorem~\ref{thm:frob-iso}, the pullback of $\hat R_A^\cX$ is a formal normalized canonical
solution for the same Frobenius manifold.
\end{proof}

\subsection{Boundary series at the orbifold point}

{The gauge calculation uses the boundary series obtained by
contracting each canonical column with the flat unit. We compute these
series for free main-component thimbles and for both columns over each fixed point.}

Until the gauge is removed, all fixed-node boundary limits are taken along one common approach from
{the domain $\Omega_B$}. We fix an analytic arc
\[
  \gamma(\epsilon)\in\Omega_B,
  \qquad
  \gamma(\epsilon)\longrightarrow\kappa^{\rm orb},
\]
on which all pole-cancellation parameters and the remaining base parameters tend to zero, and which
avoids the excluded discriminant, Maxwell, and Stokes loci for $\epsilon\ne0$. At a fixed node $r$ we
write $a_r(\epsilon)$ for the local base smoothing parameter transverse to the corresponding branch
of the pole-cancellation locus; it is
the Section~\ref{sec:frobenius} parameter $\varepsilon_r$ up to a unit, and it is not the
{local derivative} $a_\beta=(d\log\mu/ds_\beta)|_{p_\beta}$. We write $\mathbf b(\epsilon)$ for the remaining local
base parameters. In the local estimates we first regard $\mathbf b$ as a passive parameter and prove
estimates uniform for $\mathbf b$ near zero; only afterward do we substitute
$\mathbf b=\mathbf b(\epsilon)\to0$ along the arc. Since the residual diagonal gauge will be
{independent of the base parameters,} {this arc determines the gauge; removability then gives the path-independent
extension.}

{The gauge calculation uses only {the normalized series obtained by contraction with the flat unit}
$S_{\rm unit}=\Theta_1E$. The auxiliary bubble series is defined separately in
Lemma~\ref{lem:bubble-beta}.}

\begin{lemma}[{A-model unit contraction}]\label{lem:binary-dihedral-unit-contraction}
{{For every index $\alpha$ corresponding to an irreducible representation,}}
\[
  {1\over d_\alpha}\sum_\beta d_\beta
  \bigl(\hat R^\cX_{\rm orb}\bigr)_{\beta\alpha}(z)
  =
  S_{\rm unit}(z),
\]
where
\[
  S_{\rm unit}(z):=\Theta_1(z)E(z)\in1+z\Bbbk_\nu[[z]].
\]
It is the formal asymptotic expansion, as $z\to0$ in the chosen sector, of
\[
  S_{\rm unit}^{\rm an}(z)
  :=
  \sqrt{\nu\over z}\,
  {\Gamma(\nu/z+1/2)\over\Gamma(\nu/z+1)}.
\]
\end{lemma}

\begin{proof}
Using $\hat R^\cX_{\rm orb}=E\hat R^{A,\mathrm{surf}}$ and the normalized-canonical surface matrix,
\[
  \hat R^{A,\mathrm{surf}}_{\beta\alpha}(z)
  =
  {1\over|\Gamma|}\sum_{h\in\Gamma}
  \chi_\beta(h^{-1})\chi_\alpha(h)\Theta_h(z),
\]
we obtain
\[
  \sum_\beta d_\beta\hat R^{A,\mathrm{surf}}_{\beta\alpha}
  =
  {1\over|\Gamma|}
  \sum_{h\in\Gamma}
  \left(\sum_\beta d_\beta\chi_\beta(h^{-1})\right)
  \chi_\alpha(h)\Theta_h.
\]
The regular-character identity
\[
  \sum_\beta d_\beta\chi_\beta(h^{-1})=|\Gamma|\delta_{h,1}
\]
leaves $d_\alpha\Theta_1$. Multiplication by the third-leg factor $E$ gives the displayed formula.
The unit Bernoulli identity of Proposition~\ref{prop:rmatrix}, together with
Corollary~\ref{cor:thirdleg}, identifies $\Theta_1E$ with the formal asymptotic expansion of the
displayed Gamma ratio.
\end{proof}

At a fixed node the central main chain is a quotient chain. We record the normalization because this
is where the primitive quotient chain differs from the full cover arc. We write
$\Gamma_r^{\rm main}$ for the primitive quotient chain on the main component at the fixed node
$r=\pm1$, and $\Gamma_r^{\rm cover}$ for its full lifted cover arc.

{
\begin{lemma}[{Normalization of the quotient chain at a fixed node}]\label{lem:half-cover-chain}
Near a fixed point choose a cover coordinate $s$ with $\iota(s)=-s$ and quotient coordinate $t=s^2$.
If $\widetilde\gamma_{\rm q}=[0,\delta]$ lifts the primitive quotient chain
$0\le t\le\delta^2$, then {its lift with coefficients in the sign local system} is
\[
  {1\over2}\bigl(\widetilde\gamma_{\rm q}-\iota_*\widetilde\gamma_{\rm q}\bigr)
  ={1\over2}[-\delta,\delta].
\]
Consequently $2\Gamma_r^{\rm main}=\Gamma_r^{\rm cover}$.
\end{lemma}

\begin{proof}
Since $\iota_*[0,\delta]=[0,-\delta]$, their difference is $[-\delta,\delta]$, the full lifted
cover arc.
\end{proof}
}

{For the germ on the main component of the quotient at a fixed node} we use the normalization determined by
Lemma~\ref{lem:half-cover-chain}: it is one half of the full-cover odd germ.

\begin{lemma}[Main thimbles and the unit series]\label{lem:main-thimbles}
{On the main component at the orbifold boundary, write $\mu=re^q$ near a fixed node
$r\in\{\pm1\}$. For the Morse branch $\hat x-\hat x(r)=s^2$, define the centered
full-cover germ by
\[
  \widetilde y_r^{\rm cov}(s)
  :={\sqrt{-1}\over2\nu}\bigl(\log\mu(s)-\log r\bigr)
  ={\sqrt{-1}\over2\nu}q(s).
\]
The logarithms use the same local branch.} Let the full-cover coefficients be defined by
\[
  \widetilde y_r^{\rm cov}(s)-\widetilde y_r^{\rm cov}(-s)
  =
  2\sum_{k\geq1}{2^{k-1}\over(2k-1)!!}\check h_k^{r,\rm cov}s^{2k-1}.
\]
We define the quotient main coefficients by
\[
  \check h_k^{r,\rm main}:={1\over2}\check h_k^{r,\rm cov}.
\]
Then
\[
  \sum_{k\geq1}\check h_k^{r,\rm cov}z^{k-1}
  =
  \check h_1^{r,\rm cov}S_{\rm unit}(z),
  \qquad
  \sum_{k\geq1}\check h_k^{r,\rm main}z^{k-1}
  =
  \check h_1^{r,\rm main}S_{\rm unit}(z),
\]
where $S_{\rm unit}(z)$ is the unit series of
Lemma~\ref{lem:binary-dihedral-unit-contraction}.
Moreover, with the branch fixed in Section~\ref{sec:frobenius},
\[
  v_r:={\check h_1^{r,\rm main}\over\sqrt{-2}}
  ={\check h_1^{r,\rm cov}\over2\sqrt{-2}}
\]
satisfies
\[
  v_r^2=-{1\over32m^2\nu^3},
\]
which is the fixed-node norm from Section~\ref{sec:frobenius}.
For any free main thimble whose central cover orbit contains a point $\mu_0$ with $\mu_0^{2m}=1$ and
$\mu_0\ne\pm1$, the same normalized unit series holds:
\[
  \sum_{k\geq1}\check h_k^{\mu_0}z^{k-1}
  =
  \check h_1^{\mu_0}S_{\rm unit}(z).
\]
\end{lemma}

\begin{proof}
We first record the leading fixed-node normalization. Write $\mu=re^q$ near $r=\pm1$. On the main component
\[
  \lambda^{\rm orb}=K(\mu^m+\mu^{-m})=2Kr^m\cosh(mq),
\]
and hence
\[
  \hat x(q)-\hat x(0)=-2\nu\log\cosh(mq)=-\nu m^2q^2+O(q^4).
\]
For the Morse branch $\hat x-\hat x(r)=s^2$ this gives
\[
  q(s)={s\over m\sqrt{-\nu}}+O(s^3).
\]
Since the full-cover germ satisfies $\widetilde y_r^{\rm cov}=\sqrt{-1}q/(2\nu)$,
\[
  \widetilde y_r^{\rm cov}(s)-\widetilde y_r^{\rm cov}(-s)
  ={\sqrt{-1}\over m\nu\sqrt{-\nu}}s+O(s^3),
\]
so
\[
  \check h_1^{r,\rm cov}={\sqrt{-1}\over2m\nu\sqrt{-\nu}},
  \qquad
  \check h_1^{r,\rm main}={\sqrt{-1}\over4m\nu\sqrt{-\nu}}.
\]
This proves the displayed norm identity for $v_r$.

For the full cover arc through $\mu=1$, namely $\Gamma_{+1}^{\rm cover}$, use
\[
  \mu=e^{\sqrt{-1}\theta/m},
  \qquad
  -{\pi\over2}\leq\theta\leq{\pi\over2}.
\]
Then $\lambda=2K\cos\theta$ and $d\mu/\mu=\sqrt{-1}\,d\theta/m$, so
\[
  \int_{\Gamma_{+1}^{\rm cover}} e^{-\hat x/z}\phi_D
  =
  {\sqrt{-1}\sqrt\nu\over m}K^{2\nu/z}
  \int_{-\pi/2}^{\pi/2}(2\cos\theta)^{2\nu/z}\,d\theta.
\]
Euler's Beta integral gives
\[
  \int_{-\pi/2}^{\pi/2}(2\cos\theta)^{2\nu/z}\,d\theta
  =
  2^{2\nu/z}\sqrt\pi\,
  {\Gamma(\nu/z+1/2)\over\Gamma(\nu/z+1)}
\]
initially for $\operatorname{Re}(\nu/z)>-1/2$, with $\log(2\cos\theta)$ real on the
interior of the interval, and then by analytic continuation in the chosen sector.
Equivalently, as an identity of sectorial analytic functions,
\[
  e^{\hat x(0)/z}
  \int_{\Gamma_{+1}^{\rm cover}} e^{-\hat x/z}\phi_D
  =
  {\sqrt{-1}\sqrt{\pi z}\over m}\,S_{\rm unit}^{\rm an}(z).
\]
Its stationary-phase expansion has leading term
\[
  {\sqrt{-1}\sqrt{\pi z}\over m},
\]
or equivalently the leading Gaussian coefficient is $\sqrt{-1}/m$. Dividing by this leading term
gives $S_{\rm unit}^{\rm an}(z)$, whose formal asymptotic expansion is $S_{\rm unit}(z)$. By
Lemma~\ref{lem:half-cover-chain}, $\Gamma_{+1}^{\rm main}$ and its leading term are both one half of
the cover values, so the normalized analytic ratio and its formal expansion are unchanged. The
coefficient identity for the full-cover germ follows from the integration-by-parts calculation in
Proposition~\ref{prop:prym-unit-translation}; {by the normalization on the main component of the quotient given above}, the same
identity holds for $\check h_k^{r,\rm main}$.

The free main thimbles are obtained by rotating the same calculation. Let $\mu_0^{2m}=1$ and write
$\mu=\mu_0e^q$. Then
\[
  \lambda^{\rm orb}
  =
  K(\mu^m+\mu^{-m})
  =
  2K\mu_0^m\cosh(mq),
\]
because $\mu_0^{-m}=\mu_0^m$. The main contour
\[
  \mu=\mu_0e^{\sqrt{-1}\theta/m},
  \qquad
  -{\pi\over2}\leq\theta\leq{\pi\over2},
\]
gives
\[
  \lambda=2K\mu_0^m\cos\theta.
\]
Use the chosen branch of $\log\lambda$ to interpret the factor $(\mu_0^m)^{2\nu/z}$.
It cancels upon removal of the critical exponential. The leading Gaussian coefficient is determined
by the chosen Morse coordinate and the orientation of the corresponding free main thimble.
After division by the corresponding leading term, the same Beta integral gives
$S_{\rm unit}^{\rm an}(z)$, with formal asymptotic expansion $S_{\rm unit}(z)$. Taking $\mu_0=-1$
gives $\Gamma_{-1}^{\rm cover}$; passing to $\Gamma_{-1}^{\rm main}$ {uses the same normalization by one half as above.}
\end{proof}

\begin{lemma}[{Coefficientwise fixed-point limit}]
\label{lem:coefficientwise-parity-limit}
{For each fixed point $r$, write $\check h^{r,\sigma}_k(\epsilon)$ for the coefficient at the
smooth critical point indexed by $\sigma=\pm$.} Then, for every $k\geq1$ and each $\sigma=\pm$,
\[
  \lim_{\epsilon\to0}\check h^{r,\sigma}_k(\epsilon)
  =
  \check h_k^{r,\rm main}.
\]
Consequently,
\[
  \lim_{\epsilon\to0}
  {\check h^{r,+}_k(\epsilon)+\check h^{r,-}_k(\epsilon)\over2}
  =
  \check h_k^{r,\rm main},
\]
where $\check h_k^{r,\rm main}$ is the coefficient of the quotient main local germ at the node
$\mu=r$.
\end{lemma}

\begin{proof}
{Fix $r=\pm1$ and let $\sigma=\pm$ label the two nearby critical points. On the transverse
slice, write $\mu=re^\ell$ and let {the moving pair of zeros} be $re^{\pm a}$. Then
\[
  \lambda(\ell,a)
  =2Kr^m\cosh(m\ell){\cosh\ell-\cosh a\over\cosh\ell-1}.
\]
The smooth critical roots satisfy $\ell_\sigma=O(a^{1/2})$ and
\[
  \partial_\ell^2\log\lambda(\ell_\sigma,a)=4m^2+O(a),
  \qquad
  \partial_\ell^2\log\lambda(0,0)=m^2.
\]
Thus the leading smooth coefficient tends to one half of the full-cover coefficient. We show that
this holds coefficientwise.

Put $t=w-2r$. Uniformly in the remaining base parameters, the fixed-node germ has the form
\[
  \lambda(t,a)=H_r(t,a){t-\delta(a)\over t},
  \qquad H_r(0,0)\ne0,
  \qquad \delta(a)=a^2\upsilon_\delta(a),
\]
with $\upsilon_\delta(0)\ne0$. On the chosen pole-cancellation chart,
$t_\sigma(a)=a\upsilon_{t,\sigma}(a)$ with $\upsilon_{t,\sigma}(0)\ne0$; hence the critical value
$\Lambda_\sigma(a)=\lambda(t_\sigma(a),a)$ is holomorphic. The Morse equation is equivalent to
\[
  \mathcal F_\sigma(t,s,a):=
  H_r(t,a)(t-\delta(a))-\Lambda_\sigma(a)e^{-s^2/(2\nu)}t=0.
\]
Since
\[
  \mathcal F_\sigma(t,0,0)
  =t\bigl(H_r(t,0)-H_r(0,0)\bigr)
  =Km^2r^{m+1}t^2+O(t^3),
\]
Weierstrass preparation gives a degree-two factor whose roots
$t_{\sigma,1}(s,a),t_{\sigma,2}(s,a)$ are exchanged by $s\mapsto-s$. Evaluating at $t=0$ shows
\[
  t_{\sigma,1}t_{\sigma,2}=a^2\upsilon_{B,\sigma}(s^2,a),
  \qquad \upsilon_{B,\sigma}(0,0)\ne0.
\]

Lift these roots by $\ell^2=t\upsilon_{{\rm cov},r}(t)$ on the chosen canonical cover and set
$\mathfrak d_\sigma=\ell_{\sigma,1}-\ell_{\sigma,2}$.
The coefficients of the prepared quadratic are holomorphic in $(s^2,a)$.
Consequently, symmetric holomorphic functions of its two roots are
holomorphic in these parameters. In particular,
\[
  \ell_{\sigma,1}^2+\ell_{\sigma,2}^2
  =\sum_{j=1}^2t_{\sigma,j}\upsilon_{{\rm cov},r}(t_{\sigma,j})
\]
is holomorphic. The product satisfies
\[
  (\ell_{\sigma,1}\ell_{\sigma,2})^2
  =a^2\upsilon_{B,\sigma}(s^2,a)
    \prod_{j=1}^2\upsilon_{{\rm cov},r}(t_{\sigma,j}).
\]
The factor multiplying $a^2$ is a holomorphic unit. Choose its square root
so that the two lifts agree at $s=0$; then
$\ell_{\sigma,1}\ell_{\sigma,2}$ is $a$ times a holomorphic unit.
Thus $\mathfrak d_\sigma^2$ is holomorphic in $(s^2,a)$.
It vanishes at $s=0$ for $a\ne0$, and hence also at $a=0$ by holomorphicity,
so it is divisible by $s^2$. Moreover the Hessian calculation above gives
\[
  \lim_{a\to0}(\partial_s\mathfrak d_\sigma(0,a))^2
  =\lim_{a\to0}
    \frac{-4}{\nu\,\partial_\ell^2\log\lambda(\ell_\sigma,a)}
  =-\frac1{\nu m^2}\ne0.
\]
The quotient by $s^2$ is therefore a unit at the origin and remains a unit
for the passive parameters near zero. We obtain
\[
  \mathfrak d_\sigma(s,a)^2
  =s^2\upsilon_{{\rm sep},\sigma}(s^2,a),
  \qquad \upsilon_{{\rm sep},\sigma}(0,0)\ne0.
\]
The canonical square-root normalization chooses
\[
  \mathfrak d_\sigma=s\sqrt{\upsilon_{{\rm sep},\sigma}}
\]
so that both branches have limiting leading coefficient $v_r$.

At $a=0$, the prepared equation becomes
\[
  t\bigl(H_r(t,0)-H_r(0,0)e^{-s^2/(2\nu)}\bigr)=0.
\]
Let $\ell_0(s)$ lift its nonzero root. Then
$\mathfrak d_\sigma(s,0)^2=\ell_0(s)^2$. Both $\mathfrak d_\sigma(s,0)$ and $\ell_0(s)$ have a simple
zero at $s=0$, and their derivatives agree by the canonical normalization; therefore
$\mathfrak d_\sigma(s,0)=\ell_0(s)$. The smooth odd primitive consequently satisfies
\[
  \widetilde y_\sigma(s)-\widetilde y_\sigma(-s)
  ={\sqrt{-1}\over2\nu}\mathfrak d_\sigma(s,a)
  \longrightarrow {\sqrt{-1}\over2\nu}\ell_0(s).
\]
This is half the full-cover odd primitive and {exactly the germ on the main component of the quotient.} Comparing the
coefficients of $s^{2k-1}$ proves
$\check h_k^{r,\sigma}\to\check h_k^{r,{\rm main}}$ for each $\sigma$; averaging gives the last
assertion.}
\end{proof}

{For the boundary comparison, set
\[
  L^B_\alpha(\kappa;z):=\sum_\beta v_\beta(\kappa)\hat R^B_{\beta\alpha}(\kappa;z),
  \qquad
  L^A_\alpha(x;z):=\sum_\beta v_\beta(x)(\hat R_A^\cX)_{\beta\alpha}(x;z).
\]
At $x=0$, the second contraction uses $\hat R^\cX_{\rm orb}$.}

{
\begin{lemma}[{Coefficientwise limit for free critical orbits on the main component}]\label{lem:free-main-coefficientwise-limit}
Let $\alpha$ index a free main-component thimble whose central cover orbit contains
$\mu_0$ with $\mu_0^{2m}=1$ and $\mu_0\ne\pm1$. Along the chosen arc, the critical section,
normalized Morse coordinate, centered primitive, and local stationary-phase germ extend
holomorphically to the central main component. In particular, for every $k\ge1$,
\[
  \lim_{\epsilon\to0}\check h_k^\alpha(\gamma(\epsilon))
  =\check h_k^{\mu_0}.
\]
The thimble class is flat for the Gauss--Manin connection on {the domain $\Omega_B$}, and its local
stationary-phase expansion converges coefficientwise to that of the central main thimble.
\end{lemma}

\begin{proof}
The critical point at $\mu_0$ is simple and stays away from the nodes and logarithmic punctures.
The holomorphic implicit-function and parameter-dependent Morse lemmas therefore extend the critical
point and normalized Morse coordinate to the boundary. The centered primitive is holomorphic there,
so each of its odd Taylor coefficients specializes to $\check h_k^{\mu_0}$. The chosen arc crosses no
Stokes wall; its thimble class is flat for the Gauss--Manin connection, and every stationary-phase coefficient depends
on only finitely many of these Taylor coefficients. This proves both assertions.
\end{proof}
}

\begin{proposition}[{Boundary limits after contraction with the flat unit}]\label{prop:parity-even-anchor}
{For every canonical column $\alpha$, the limit along the chosen arc is
coefficientwise in $z$ and satisfies
\[
  \lim_{\epsilon\to0}L^B_\alpha(\gamma(\epsilon);z)
  =L^A_\alpha(0;z)=v_\alpha(0)S_{\rm unit}(z).
\]
For the two columns over a fixed point $r=\pm1$, the chosen square roots give
$v_{r,+}(0)=v_{r,-}(0)=v_r$, where
$v_r=\check h_1^{r,\rm main}/\sqrt{-2}$.}
\end{proposition}

\begin{proof}
{The B-model unit contraction expresses each column as
\[
  L^B_\alpha=\frac1{\sqrt{-2}}
    \sum_{k\ge1}\check h_k^\alpha z^{k-1}.
\]
For a free main orbit, Lemma~\ref{lem:free-main-coefficientwise-limit}
gives the limits of these coefficients. For each fixed-point column,
Lemma~\ref{lem:coefficientwise-parity-limit} gives
$\check h_k^{r,\sigma}\to\check h_k^{r,\rm main}$ separately for
$\sigma=\pm$. In both cases, Lemma~\ref{lem:main-thimbles} identifies
the resulting series with $v_\alpha(0)S_{\rm unit}$.

On the A-side, Proposition~\ref{prop:aside-chamber-calibration} gives
specialization to the orbifold calibration. The unit covector at the
orbifold point is a common nonzero scalar multiple of the dimension vector:
$v_\beta(0)=c\,d_\beta$. Lemma~\ref{lem:binary-dihedral-unit-contraction}
therefore gives $L^A_\alpha(0;z)=v_\alpha(0)S_{\rm unit}(z)$ for every
column. Metric compatibility and the chosen square roots identify these
leading coefficients with the B-side values, including
$v_{r,+}(0)=v_{r,-}(0)=v_r$.}
\end{proof}

{
\begin{lemma}[Limiting bubble period]\label{lem:bubble-beta}
At a fixed node $r=\pm1$, let $a_r$ be a smoothing parameter and use the local normal form
\[
  \lambda=H_r(t,a){t-ra^2\over t},
  \qquad
  \mu=r+{a\over y_r},
  \qquad
  t={ra^2\over y_r^2}+O(a^3),
\]
where $a=a_r$ and $H_r(0,0)\ne0$. The primary differential
$\omega_{a_r}=2\nu\,\partial_{a_r}\log\lambda\,\phi_D$ has the bubble limit
\[
  \omega_{a_r}\longrightarrow
  {4\nu\sqrt\nu\over r}{dy_r\over1-y_r^2}.
\]
With $X=2\nu/z$, {after removing the critical exponential and dividing by the
leading stationary-phase term $(4\nu\sqrt\nu/r)\sqrt{\pi/X}$, its sectorial period is}
\[
  S_{\rm bub}^{\rm an}(z)
  :={\sqrt X\over\sqrt\pi}\int_{-1}^{1}(1-y^2)^{X-1}\,dy
  ={\Gamma(X+1)\over\sqrt X\,\Gamma(X+1/2)}.
\]
Its formal stationary-phase expansion is
\[
  S_{\rm bub}(z)=\Theta_{1/4}(z)E(z)\in1+z\Bbbk_\nu[[z]],
\]
where $\Theta_{1/4}$ denotes the series $\Theta_h$ of Proposition~\ref{prop:rmatrix} for the
unordered eigenangle pair $(1/4,3/4)$.
\end{lemma}

\begin{proof}
The local coordinates give
\[
  \phi_D=-{a\sqrt\nu\over r}{dy_r\over y_r^2}+O(a^2),
  \qquad
  \partial_a\log\lambda=-{2ra\over t-ra^2}+\partial_a\log H_r.
\]
The $H_r$ term vanishes after multiplication by $\phi_D$, while
$t-ra^2=ra^2(1-y_r^2)/y_r^2+O(a^3)$. Hence
\[
  2\nu\,\partial_a\log\lambda\,\phi_D
  \longrightarrow {4\nu\sqrt\nu\over r}{dy_r\over1-y_r^2}.
\]
A change $a'=ca+O(a^2)$ rescales both this differential and {the leading stationary-phase
term of its period}, so the normalized period is unchanged. Euler's Beta integral gives the displayed Gamma ratio. Its
asymptotic expansion is $\Theta_{1/4}E$ by Proposition~\ref{prop:rmatrix} and
Corollary~\ref{cor:thirdleg}.
\end{proof}
}

\subsection{{Orbifold-point gauge fixing}}

{We compare the two calibrations over $\mathscr K$, where both are defined without assuming
convergence of the A-side series. Theorem~\ref{thm:frob-iso} identifies their products, metrics,
canonical coordinates, and normalized canonical frames. They therefore satisfy the same formal
Dubrovin equation.}

\begin{lemma}[Formal canonical uniqueness]\label{lem:formal-canonical-uniqueness}
Let $F=\operatorname{Frac}(\widehat{\mathcal O}')$. Suppose
\[
  R_i\in I+z\operatorname{Mat}_{l+1}(F[[z]]),\qquad i=1,2,
\]
satisfy
\[
  dR_i={[dU,R_i]\over z}+[\Gamma^B,dU]R_i
\]
in the normalized canonical frame, and suppose both are symplectic:
\[
  R_i(z)R_i(-z)^{\mathsf T}=I.
\]
Then
\[
  R_2=R_1D,
\]
where $D=\operatorname{diag}(D_\alpha)$ is diagonal, is annihilated by all base derivations, and
satisfies
\[
  D(z)D(-z)=I.
\]
Equivalently,
\[
  D_\alpha(z)=\exp\left(\sum_{j\ge0}\gamma_{\alpha,j}z^{2j+1}\right).
\]
If $R_1$ and $R_2$ have coefficients in $\mathscr K[[z]]$, then so does $D$.
\end{lemma}

\begin{proof}
{After scalar extension to the characteristic-zero field $F$, this is precisely the unitary
uniqueness statement for a normalized canonical fundamental solution
\cite[Prop.~(c)]{givental2001semisimplefrobeniusstructureshigher}. It gives
$R_2=R_1D$, with $D$ constant and diagonal, and identifies every diagonal entry as the exponential of
an odd series. Equivalently, $D(z)D(-z)=I$. Finally, if both matrices are defined over
$\mathscr K[[z]]$, then so is $D=R_1^{-1}R_2$.}
\end{proof}

\begin{proposition}[{Gauge fixing by the limit after contraction with the flat unit}]\label{prop:gauge-fixing}
In $\operatorname{End}(H_{\CR,\bT}^*(\cX))\otimes\mathscr K[[z]]$,
\[
  \hat R^B=\hat R_A^\cX.
\]
Consequently, after pullback to the chosen finite analytic cover {over the domain $\Omega_B$}, the
coefficients of $\hat R_A^\cX$ are convergent. Their
coefficientwise analytic continuation in the chosen normalized canonical frame satisfies
\[
  \hat R^B(\kappa;z)=\hat R^\cX_{\rm can}(\Phi(\kappa);z),
  \qquad
  \kappa\in\Omega_B.
\]
\end{proposition}

\begin{proof}
Lemma~\ref{lem:formal-canonical-uniqueness}, applied over
$F=\operatorname{Frac}(\widehat{\mathcal O}')$, gives
\[
  \hat R^B=\hat R_A^\cX D,
\]
{where $D=\operatorname{diag}(D_\alpha)$ is independent of $\kappa$ and
$D_\alpha(z)=\exp(\sum_{j\ge0}\gamma_{\alpha,j}z^{2j+1})$. Right multiplication by $D$ scales
canonical column $\alpha$, so the contractions defined above satisfy}
\begin{equation}
\label{eq:anchor:gauge-column-scaling}
  L^B_\alpha=D_\alpha L^A_\alpha.
\end{equation}

All boundary comparisons in the gauge calculation are made by one formal operation. The chosen arc
$\gamma(\epsilon)$ induces a continuous homomorphism
\[
  \widehat\gamma^*:\widehat{\mathcal O}'\longrightarrow \Bbbk_\nu[[\epsilon]].
\]
For $\epsilon\ne0$ the arc lies in $\Omega_B$, hence $\mathfrak s(\gamma(\epsilon))\ne0$, and the map
extends to
\[
  \widehat\gamma^*:\mathscr K=\widehat{\mathcal O}'[\mathfrak s^{-1}]
  \longrightarrow \Bbbk_\nu((\epsilon)).
\]
We apply this map coefficientwise in $z$ and then take the Laurent coefficient $[\epsilon^0]$. Since
$D_\alpha(z)$ is independent of the base parameter, for any $\varphi$,
\[
  [\epsilon^0]\widehat\gamma^*\bigl(D_\alpha\varphi\bigr)
  =
  D_\alpha(z)[\epsilon^0]\widehat\gamma^*\varphi .
\]
If $\varphi$ lies in $\widehat{\mathcal O}'$, this constant term is its closed-point specialization.

{For every canonical column $\alpha$, apply
$[\epsilon^0]\widehat\gamma^*$ to
\eqref{eq:anchor:gauge-column-scaling}. This gives
\[
  [\epsilon^0]\widehat\gamma^*L^B_\alpha
  =D_\alpha(z)[\epsilon^0]\widehat\gamma^*L^A_\alpha.
\]
Proposition~\ref{prop:parity-even-anchor} identifies the B-side limit
and the A-side closed-point value with $v_\alpha(0)S_{\rm unit}(z)$.
This applies separately to both columns over each fixed point.
Since $v_\alpha(0)\ne0$ and $S_{\rm unit}(0)=1$, cancellation gives
$D_\alpha(z)=1$ for every $\alpha$.}

Thus $\hat R^B=\hat R_A^\cX$ over $\mathscr K$. We now prove removability coefficientwise in $z$. Let
$b$ be any matrix coefficient of any fixed $z^N$ coefficient of $\hat R^B$. By
Proposition~\ref{prop:aside-chamber-calibration}, before completion
\[
  b\in\mathcal O'[\mathfrak s^{-1}].
\]
{Equality with the A-side formal coefficient puts the image $\widehat b$ of $b$ in
$\widehat{\mathcal O}'$. Choose $M$ such that $a=\mathfrak s^M b\in\mathcal O'$.
Since $\widehat{\mathcal O}'$ is a domain, equality after localization gives
$a=\mathfrak s^M\widehat b$. Faithful flatness of completion gives
\[
  a\in\mathfrak s^M\widehat{\mathcal O}'\cap\mathcal O'
  =\mathfrak s^M\mathcal O',
\]
hence $b\in\mathcal O'$. {Thus the apparent poles along the pole-cancellation locus and the discriminant locus of the
critical-point polynomial are removable on the chosen finite analytic cover.}

The Taylor series of $b$ is the corresponding pulled-back A-side coefficient. Thus each
coefficient of the A-side calibration converges in the base parameters on this cover.
The series in $z$ remains formal.}

We denote this pulled-back analytic A-side calibration near the boundary by
$\hat R^\cX_{\rm can}(\Phi(\kappa);z)$. {On the simply connected domain $\Omega_B$, both}
$\hat R^B(\kappa;z)$ and this analytic continuation solve the same Dubrovin equation in the same
normalized canonical frame. They agree near the boundary on the chosen cover, and hence agree
throughout $\Omega_B$ by uniqueness of analytic continuation for solutions of the Dubrovin equation
on the simply connected domain.
\end{proof}

\begin{theorem}[{$R$-matrix comparison on $\Omega_B$}]\label{thm:anchor}
{Under the mirror identification of normalized canonical frames, for every $\kappa\in\Omega_B$,
\[
  \hat R^B(\kappa;z)
  =
  \hat R^\cX_{\rm can}(\Phi(\kappa);z).
\]}
\end{theorem}

\begin{proof}
{Propositions~\ref{prop:smooth-chamber-calibration} and~\ref{prop:aside-chamber-calibration}
place both calibrations over $\mathscr K$. Proposition~\ref{prop:gauge-fixing} proves their equality,
removes the apparent poles, and gives the coefficientwise analytic continuation on $\Omega_B$.}
\end{proof}

{The calibration comparison and the graph comparison in Section~\ref{sec:graphsum-remodeling}
use the corresponding normalized canonical frames from Corollary~\ref{cor:frob-identification-extension}
and are invariant under simultaneous relabeling of their indices. The explicit boundary labels
and character formulas will be given in Appendix~\ref{app:boundary-characters}.}


\section{Descendant remodeling and {free energies}}\label{sec:graphsum-remodeling}

{
We compare the Givental--Teleman and DOSS graph sums using the Frobenius isomorphism and
$R$-matrix comparison established above. Following the procedure for equivariant $\bP^1$
\cite{fang2016eynardorantinrecursionequivariantmirror}, we express the DOSS ordinary leaves in
an undressed basis and identify them with ancestor inputs through a parity substitution.
The ancestor--descendant transformation then gives descendant remodeling, and the residue
specialization of a single ordinary leaf gives the higher-genus free energies.
}

\subsection{Ancestor and descendant variables}

{
We work on the simply connected domain $\Omega_B$, with $x=\Phi(\kappa)$ and the normalized
canonical frame chosen in Sections~\ref{sec:bgraphsum} and~\ref{sec:anchor}. We write
$\Omega_B^{\rm cov}\to\Omega_B$ for the restriction of the finite analytic cover associated
with $\mathcal O'$ in Definition~\ref{def:ring-K}, and use the chosen branches on this cover.
The superscript ${\rm an}$ will denote coefficientwise analytic continuation from the formal
orbifold neighborhood; Proposition~\ref{prop:aside-analytic-graph-theory} establishes these
continuations for the objects used below.

Choose a flat basis $\{T_a\}$ of $H_{\CR,\bT}^*(\cX)$ and polynomial inputs
\[
  \bu_j(z)=\sum_{a,k}u_j^{a,k}z^kT_a.
\]
Put $\mathbf x=\sum_a x_aT_a$. For $2g-2+n>0$, the formal Gromov--Witten descendant series is
\begin{equation}\label{eq:remodeling:formal-descendant-series}
  \mathcal F^{\cX}_{g,n}(\bu_1,\ldots,\bu_n;x)
  =\sum_{m\geq0}{1\over m!}
  \left\langle
  \bu_1(\bar\psi_1),\ldots,\bu_n(\bar\psi_n),
  \mathbf x,\ldots,\mathbf x
  \right\rangle^{\cX,\bT}_{g,n+m},
\end{equation}
where $\bar\psi_j$ is the descendant cotangent class of Section~\ref{sec:amodel}.
The shifted CohFT of Proposition~\ref{prop:aside-chamber-calibration} is
\begin{equation}\label{eq:remodeling:shifted-cohft}
  \Omega^A_{g,n}(\gamma_1,\ldots,\gamma_n)
  =\sum_{m\geq0}{1\over m!}(\pi_m)_*
  \Omega_{g,n+m}(\gamma_1,\ldots,\gamma_n,\mathbf x,\ldots,\mathbf x),
\end{equation}
with $\pi_m$ forgetting the last $m$ markings. Its ancestor series is
\begin{equation}\label{eq:remodeling:ancestor-series}
  \mathcal A^{\cX}_{g,n}(\mathbf v_1,\ldots,\mathbf v_n;x)
  =\int_{\Mbar_{g,n}}
  \Omega^A_{g,n}\bigl(\mathbf v_1(\psi_1),\ldots,\mathbf v_n(\psi_n)\bigr),
\end{equation}
where the expression is extended multilinearly in the polynomial inputs and the
$\psi_j$ are the cotangent classes on $\Mbar_{g,n}$.

The genus-zero operator relating these series is defined over $A^{\rm can}$ by
\begin{equation}\label{eq:remodeling:s-operator}
  (\varphi,S_x(z)\varphi')
  =(\varphi,\varphi')
  +\sum_{m\geq1}{1\over m!}
  \left\langle
  \varphi,{\varphi'\over z-\bar\psi},\mathbf x,\ldots,\mathbf x
  \right\rangle^{\cX,\bT}_{0,m+2}.
\end{equation}
The pairing accounts for the unstable two-point term, so $S_0(z)=I$.
Let $\hat\phi_\alpha(x)=\sqrt{\Delta_\alpha(x)}e_\alpha(x)$ be the normalized canonical vector.
We set
\begin{equation}\label{eq:remodeling:transformed-input}
  \widetilde u_j^\alpha(z)
  =\left[\bigl(\hat\phi_\alpha(x),S_x(z)\bu_j(z)\bigr)\right]_+,
  \qquad
  \widetilde\bu_j(z)=\sum_\alpha\widetilde u_j^\alpha(z)\hat\phi_\alpha(x),
\end{equation}
where $[\ ]_+$ retains the nonnegative powers of $z$. On the domain we use the same notation
with $S_x^{\rm an}(z)$ in place of $S_x(z)$.

\begin{lemma}[Stable ancestor--descendant transformation]
\label{lem:stable-ancestor-descendant}
For $2g-2+n>0$, the formal series satisfy
\begin{equation}\label{eq:remodeling:descendant-series}
  \mathcal F^{\cX}_{g,n}(\bu_1,\ldots,\bu_n;x)
  =\mathcal A^{\cX}_{g,n}(\widetilde\bu_1,\ldots,\widetilde\bu_n;x).
\end{equation}
\end{lemma}

\begin{proof}
The ancestor--descendant comparison expresses the difference between $\bar\psi_j$ and
$\psi_j$ by the rational tails contracted when the extra primary markings are forgotten
\cite{kontsevich1997relationscorrelatorstopological,
givental2001semisimplefrobeniusstructureshigher}. Summing the contributions of the tails
at the $j$-th marking gives $[S_x(z)\bu_j(z)]_+$. The remaining stable component carries
the shifted class \eqref{eq:remodeling:shifted-cohft}, yielding the displayed identity.
The quadratic term in the quantized $S$-operator belongs to the unstable genus-zero
two-point sector.
\end{proof}

\begin{proposition}[Analytic A-side graph-sum data on $\Omega_B$]
\label{prop:aside-analytic-graph-theory}
The shifted CohFT $\Omega^A_{g,n}$, the coefficients of $S_x(z)$, and the ancestor and descendant
series admit coefficientwise analytic continuations to the chosen domain. The continued
series $\mathcal A^{\cX,{\rm an}}_{g,n}$ and $\mathcal F^{\cX,{\rm an}}_{g,n}$ have coefficients in
\[
  \mathcal O_{\rm an}(\Omega_B^{\rm cov})[[u_j^{a,k}]],
\]
with the product completion in the descendant variables, and satisfy
\eqref{eq:remodeling:descendant-series} with superscript ${\rm an}$.
\end{proposition}

\begin{proof}
Theorems~\ref{thm:frob-iso} and~\ref{thm:anchor} supply the analytic Frobenius data and the
coefficientwise analytic calibration. Givental--Teleman reconstruction therefore gives
analytic shifted CohFT classes, including $\Omega^{\rm an}_{g,0}$ for $g\geq2$.
For fixed $(g,n)$ only finitely many graph terms contribute: if a decorated graph has $m$
dilaton leaves and $|E|$ edges, the dimension constraint and the minimum dilaton height $2$
give $m+|E|\leq3g-3+n$. Integrating against the ancestor classes gives
$\mathcal A^{\cX,{\rm an}}_{g,n}$.

To construct the continuation of $S_x(z)$, write
\[
  S_x(z)=\sum_{r\geq0}S_r(x)z^{-r},\qquad S_0(x)=I,
  \qquad C=\sum_a(T_a\star_x)\,dx_a.
\]
Genus-zero topological recursion gives $z\,dS_x=C S_x$, hence
\[
  dS_{r+1}=C S_r,\qquad S_{r+1}(0)=0\quad(r\geq0).
\]
The one-form $C$ is analytic near $x=0$ by the genus-zero mirror theorem.
Flatness gives $dC=0$ and $C\wedge C=0$, so, inductively, $C S_r$ is closed.
The holomorphic Poincar\'e lemma constructs each $S_{r+1}$ near $0$ with the prescribed
initial value. Its Taylor expansion equals the formal coefficient by uniqueness.
Continuing these primitives gives single-valued coefficients on the simply connected
domain.

Lemma~\ref{lem:stable-ancestor-descendant} now gives the analytic continuation of the formal
GW descendant series. For polynomial inputs, the projection in
\eqref{eq:remodeling:transformed-input} uses finitely many $S_r$. More generally, a fixed
monomial in the variables $u_j^{a,k}$ involves only finitely many descendant degrees and
therefore finitely many coefficients of this transformation. {This gives the statement in the product completion.} The Taylor expansions of the constructed objects agree with the original formal
series.
\end{proof}
}

\subsection{The undressed B-model leaves}

{
The undressed tower $\mathsf W^k_\alpha$ is obtained by differentiating the primary
second-kind forms. Its lower polar terms are canceled by the dressing identity of
Proposition~\ref{prop:leaf-dressing}.
}
For $k\geq0$ define the anti-invariant forms of the undressed tower
\begin{equation}
\label{eq:remodeling:undressed-tower}
  \mathsf W_\alpha^0=\vartheta_\alpha^0,\qquad
  \mathsf W_\alpha^k=
  d\left((-1)^k
  \left({d\over d\hat x}\right)^{k-1}
  {\vartheta_\alpha^0\over d\hat x}\right)\quad(k\geq1).
\end{equation}
Here $\vartheta_\alpha^0/d\hat x$ is the meromorphic function obtained by
dividing one local one-form by the other.  Since $\hat x$ is invariant under
$\iota$ and $\vartheta_\alpha^0$ is anti-invariant, each $\mathsf W_\alpha^k$
{is anti-invariant and hence descends to the sign local system on the quotient curve.}

\begin{lemma}[Principal parts of the undressed tower]\label{lem:undressed-leaf-principal-part}
On the chosen local disk at $p_\beta$,
\begin{equation}
\label{eq:remodeling:principal-parts}
  \mathsf W_\alpha^k
  =
  -\delta_{\alpha\beta}{(2k+1)!!\over 2^k}
  {ds_\beta\over s_\beta^{2k+2}}
  +O(s_\beta^{-2k})\,ds_\beta .
\end{equation}
Only even polar orders occur, and in particular the residue is zero at every
critical point.  Moreover the forms $\mathsf W_\alpha^k$ have no additional
poles at the zero or pole punctures of $\lambda$.
\end{lemma}

\begin{proof}
Near $p_\beta$ we have $\hat x=-u_\beta+s_\beta^2$.  From the definition of
$\vartheta_\alpha^0$ and the local expansion of the Prym kernel,
\begin{equation}
\label{eq:remodeling:bare-expansion}
  \vartheta_\alpha^0
  =
  -\delta_{\alpha\beta}{ds_\beta\over s_\beta^2}
  -\sum_{m\geq0}B^{\beta\alpha}_{m,0}s_\beta^m\,ds_\beta .
\end{equation}
Dividing \eqref{eq:remodeling:bare-expansion} by $d\hat x=2s_\beta ds_\beta$ gives a function with leading term
$-\delta_{\alpha\beta}/(2s_\beta^3)$.  Repeated application of
\[
  {d\over d\hat x}={1\over2s_\beta}{d\over ds_\beta}
\]
and the factor $(-1)^k$ in \eqref{eq:remodeling:undressed-tower} gives the leading term in
\eqref{eq:remodeling:principal-parts}.  The term
with $m$ odd becomes a constant at the step where its exponent reaches zero,
before the final differential can produce a pole.  Hence only even polar
orders occur.

For example, if near $p_\alpha$
\[
  \vartheta_\alpha^0
  =
  -{ds_\alpha\over s_\alpha^2}
  -B^{\alpha\alpha}_{0,0}\,ds_\alpha
  +O(s_\alpha)\,ds_\alpha,
\]
then
\[
  \mathsf W_\alpha^1
  =
  -{3\over2}{ds_\alpha\over s_\alpha^4}
  -{1\over2}B^{\alpha\alpha}_{0,0}
  {ds_\alpha\over s_\alpha^2}
  +O(1)\,ds_\alpha .
\]
The lower double pole is the term canceled by the dressing identity of
Proposition~\ref{prop:leaf-dressing}.

At a zero or pole $p_0$ of $\lambda$, write a local coordinate $\tau_{p_0}$ so that
\[
  d\hat x=-2\nu\,{d\lambda\over\lambda}
  =
  c\,{d\tau_{p_0}\over\tau_{p_0}}+O(1)d\tau_{p_0},
  \qquad c\neq0 .
\]
The form $\vartheta_\alpha^0$ is holomorphic there, and division by $d\hat x$
therefore multiplies a holomorphic differential by $O(\tau_{p_0})$.  The operator
$d/d\hat x$ and the final differential $d$ do not create a puncture pole.
\end{proof}

{
The \emph{$R$-undressed ordinary leaves} are
\begin{equation}\label{eq:remodeling:undressed-leaves}
  \widehat\theta_\alpha^k={\mathsf W_\alpha^k\over\sqrt{-2}},\qquad k\geq0.
\end{equation}
In particular, $\widehat\theta_\alpha^0=\theta_\alpha^0$.
The one-point specialization of Proposition~\ref{prop:prym-double-laplace} is
\begin{equation}\label{eq:remodeling:one-leg-edge}
  \sum_{\ell\geq0}\check B^{\alpha\beta}_{\ell,0}z^\ell
  ={\delta_{\alpha\beta}-\hat R^B_{\beta\alpha}(z)\over z}.
\end{equation}

\begin{lemma}[Reduced leaf recurrence]\label{lem:projected-leaf-recurrence}
For every $k\geq0$,
\[
  \vartheta_\alpha^k
  =\mathsf W_\alpha^k-
  \sum_\beta\sum_{i=0}^{k-1}
  \check B^{\alpha\beta}_{k-1-i,0}\,\mathsf W_\beta^i.
\]
\end{lemma}

\begin{proof}
For a meromorphic one-form $\eta$, put $\mathcal D\eta=-d(\eta/d\hat x)$.
Then $\mathsf W_\alpha^k=\mathcal D^k\vartheta_\alpha^0$.
At $p_\beta$, the definition of the bare forms gives
\[
  \vartheta_\alpha^k
  =-\delta_{\alpha\beta}{(2k+1)!!\over2^k}
    {ds_\beta\over s_\beta^{2k+2}}
   -{(2k-1)!!\over2^k}
    \sum_{j\geq0}B^{\beta\alpha}_{j,2k}s_\beta^j\,ds_\beta.
\]
Since $d\hat x=2s_\beta ds_\beta$, applying $\mathcal D$ to the leading pole gives the
principal part of $\vartheta_\alpha^{k+1}$. Among the regular terms only the constant term
produces a pole, whose coefficient is
\[
  -{(2k-1)!!\over2^{k+1}}B^{\beta\alpha}_{0,2k}
     {ds_\beta\over s_\beta^2}
  =-\check B^{\alpha\beta}_{k,0}{ds_\beta\over s_\beta^2}.
\]
Here we used symmetry of $B^-$ and \eqref{eq:bgraphsum:edge-coefficients}.
Thus the difference in
\[
  \mathcal D\vartheta_\alpha^k
  =\vartheta_\alpha^{k+1}
   +\sum_\beta\check B^{\alpha\beta}_{k,0}\vartheta_\beta^0
\]
is holomorphic at all representative critical points, and anti-invariance gives the same
conclusion at their other lifts. The puncture estimate in
Lemma~\ref{lem:undressed-leaf-principal-part} applies equally to $\vartheta_\alpha^k$:
division by $d\hat x$ gives a holomorphic function vanishing at each puncture.
There are no other possible poles. The difference is therefore a holomorphic one-form on
the rational $\mu$-cover and is zero. Iterating this identity proves the recurrence.
\end{proof}

Dividing the recurrence by $\sqrt{-2}$ and using \eqref{eq:remodeling:one-leg-edge} gives
\begin{equation}\label{eq:remodeling:dressing-relation}
  \theta_\alpha(z)=\sum_\beta\hat R^B_{\beta\alpha}(z)\widehat\theta_\beta(z),
  \qquad
  \theta_\alpha(z)=\sum_{k\geq0}\theta_\alpha^kz^k,\quad
  \widehat\theta_\alpha(z)=\sum_{k\geq0}\widehat\theta_\alpha^kz^k.
\end{equation}
This is the dressing relation of
\cite[Prop.~6.6]{fang2019remodelingconjecturetoriccalabiyau} in our normalization.
}

\begin{proposition}[Dressed and undressed leaves]\label{prop:leaf-dressing}
For every $k\geq0$,
\[
  \theta^k_\alpha
  =
  \sum_{\beta}\sum_{i=0}^{k}
  [z^{k-i}]\hat R^B_{\beta\alpha}(z)\,
  \widehat\theta_\beta^i .
\]
\end{proposition}

{
\begin{proof}
Extract the coefficient of $z^k$ in \eqref{eq:remodeling:dressing-relation}.
For example,
\[
  \theta^1_\alpha=\widehat\theta^1_\alpha
    +\sum_\beta[z]\hat R^B_{\beta\alpha}(z)\widehat\theta^0_\beta.
\]
Since $[z]\hat R^B_{\beta\alpha}=-\check B^{\alpha\beta}_{0,0}$, the second term cancels
the lower double poles of $\mathsf W^1_\alpha/\sqrt{-2}$.
\end{proof}
}

\subsection{The graph comparison}

We now define the B-side formal series by first rewriting each DOSS ordinary
leaf with Proposition~\ref{prop:leaf-dressing} and then replacing only the
undressed leaves by the transformed A-side inputs.

\begin{definition}[Parity-twisted B-side leaf-replacement series]
\label{def:b-formal-descendant-series}
For $n>0$ and $2g-2+n>0$, let
\[
  \cW^{B,\Pi}_{g,n}(\widetilde\bu_1,\ldots,\widetilde\bu_n;\kappa)
\]
be the result of applying the DOSS graph sum of Theorem~\ref{thm:doss}, using
Proposition~\ref{prop:leaf-dressing} on every ordinary leaf, and then making
the substitution
\begin{equation}
\label{eq:remodeling:parity-substitution}
  \widehat\theta_\alpha^k(q_j)
  \longmapsto
  (-1)^{k+1}\widetilde u_j^{k,\alpha}.
\end{equation}
Here
\[
  \widetilde u_j^{k,\alpha}:=[z^k]\widetilde u_j^\alpha(z).
\]
If in \eqref{eq:remodeling:parity-substitution} we substitute back
$\widetilde u_j^{k,\alpha}=(-1)^{k+1}\widehat\theta_\alpha^k(q_j)$, then
Proposition~\ref{prop:leaf-dressing} recovers the original DOSS expansion of
$\omega^\iota_{g,n}(q_1,\ldots,q_n)$.
\end{definition}

From now on we write
\begin{equation}
\label{eq:remodeling:shared-rmatrix}
  R(z):=\hat R^B(\kappa;z)=\hat R^\cX_{\rm can}(\Phi(\kappa);z),
\end{equation}
{
by Theorem~\ref{thm:anchor}. We regard input components as column vectors.
Symplecticity gives
\[
  R(z)^{-1}=R(-z)^{\mathsf T}.
\]
Thus the A-side action on an ordinary ancestor input is $R(-z)^{\mathsf T}$, while the
B-side dressing relation uses $R(z)^{\mathsf T}$.
}

\begin{proposition}[A-side graph weights in the canonical frame]\label{prop:aside-graph-weights}
The analytic ancestor series
$\mathcal A^{\cX,{\rm an}}_{g,n}$ is computed by the stable Givental--Teleman
graph sum with the following local weights.  Put
\[
  v_\alpha={1\over\sqrt{\Delta^\cX_\alpha}}.
\]
{At a vertex of genus $g_v$ and valence $m$ colored by $\alpha$,} the normalized
canonical topological field theory factor is
\[
  v_\alpha^{2-2g_v-m}
  \left\langle\prod_{j=1}^m\tau_{k_j}\right\rangle_{g_v} ,
\]
where the bracket is the tautological integral on $\Mbar_{g_v,m}$. {The edge joining
vertices of colors $\alpha$ and $\beta$ has coefficients}
\[
  [z^kw^\ell]\,
  {\delta_{\alpha\beta}
  -\sum_\gamma
  \hat R^\cX_{{\rm can},\gamma\alpha}(x;-z)
  \hat R^\cX_{{\rm can},\gamma\beta}(x;-w)
  \over z+w}.
\]
{The ordinary leaf attached to the $j$-th input and to a vertex of color $\alpha$ is}
\[
  [z^k]\sum_\beta
  \hat R^\cX_{{\rm can},\beta\alpha}(x;-z)\,
  \widetilde u_j^\beta(z),
\]
and {the dilaton translation determined by the flat unit} is
\[
  T^A_\alpha(z)
  =
  z\,v_\alpha
  -
  z\sum_\beta v_\beta
  \hat R^\cX_{{\rm can},\beta\alpha}(x;-z).
\]
Its stable coefficients are
\[
  [z^k]T^A_\alpha(z)
  =
  -[z^k]\sum_\beta z\,v_\beta\,
  \hat R^\cX_{{\rm can},\beta\alpha}(x;-z),
  \qquad k\geq2 .
\]
\end{proposition}

{
\begin{proof}
The normalized canonical vector is $\hat\phi_\alpha=e_\alpha/v_\alpha$, and the Frobenius
Euler element is $\mathcal E=\sum_\beta\Delta^\cX_\beta e_\beta$.
Since $\eta(e_\alpha,\mathbf1)=v_\alpha^2$, the topological field theory gives
\[
  \eta\bigl(\hat\phi_\alpha^{\,m}\mathcal E^{g_v},\mathbf1\bigr)
  =v_\alpha^{-m}(\Delta^\cX_\alpha)^{g_v}v_\alpha^2
  =v_\alpha^{2-2g_v-m}.
\]
The powers here are Frobenius products. Multiplying by the cotangent integral gives the
vertex weight.

The Givental--Teleman action attaches $R(z)^{-1}$ to an input and has propagator
\[
  {I-R(z)^{-1}(R(w)^{-1})^{\mathsf T}\over z+w}
  ={I-R(-z)^{\mathsf T}R(-w)\over z+w}
\]
\cite{givental2001semisimplefrobeniusstructureshigher,
teleman2012structure2dsemisimplefield}. These are the displayed ordinary-leaf and edge
coefficients. With $v=(v_\alpha)_\alpha$ the canonical component vector of the unit, the
{translation determined by the flat unit} is
\[
  T^A(z)=zv-zR(z)^{-1}v=zv-zR(-z)^{\mathsf T}v.
\]
It starts at order $z^2$, giving the stable dilaton coefficients.
\end{proof}
}

We use $E^A_{\alpha\beta;k,\ell}$, $L^A_{\alpha,k}$, and
$D^A_{\alpha,k}$ for the A-side edge, ordinary-leaf, and stable dilaton
coefficients, and $E^B_{\alpha\beta;k,\ell}$, $L^B_{\alpha,k}$, and
$D^B_{\alpha,k}$ for the corresponding B-side DOSS coefficients.

\begin{lemma}[Parity of the graph weights]\label{lem:graph-parity}
Let $\Pi U(z):=-U(-z)$.  With $R(z)$ as in \eqref{eq:remodeling:shared-rmatrix}, the B-side edge,
ordinary leaf, and stable dilaton factors are related to the A-side factors by
\[
  E^B_{\alpha\beta;k,\ell}
  =
  (-1)^{k+\ell+1}E^A_{\alpha\beta;k,\ell},
\]
\[
  L^B_{\alpha,k}(\Pi U)
  =
  (-1)^{k+1}L^A_{\alpha,k}(U),
\]
and
\[
  D^B_{\alpha,k}
  =
  (-1)^{k+1}D^A_{\alpha,k}.
\]
For every stable connected decorated graph with nonzero vertex contribution,
the product of these local parity signs is $(-1)^{g-1}$.
\end{lemma}

\begin{proof}
The edge identity follows from
\[
  {I-R(z)^{\mathsf T}R(w)\over z+w}
  =
  -
  \left.
  {I-R(-z')^{\mathsf T}R(-w')\over z'+w'}
  \right|_{z'=-z,\ w'=-w}.
\]
Taking the coefficient of $z^kw^\ell$ gives the edge sign. The ordinary leaf
identity is the same parity calculation in one variable:
{
\[
  R(z)^{\mathsf T}(-U(-z))
  =-\bigl(R(-z)^{\mathsf T}U(z)\bigr)\big|_{z\mapsto -z}.
\]
}
This gives the ordinary leaf sign. For the dilaton leaf, compare
\[
  -[z^k]\,z\sum_\beta v_\beta R_{\beta\alpha}(z)
  \quad\text{and}\quad
  -[z^k]\,z\sum_\beta v_\beta R_{\beta\alpha}(-z),
\]
which gives the dilaton sign.

It remains to count signs.  Let a connected stable graph have $|E(\Gamma)|$ edges, $n$
ordinary leaves, $m$ dilaton leaves, and total cotangent height
\[
  k_{\rm tot}=\sum_h k(h)
\]
{where the sum is over all half-edges and leaves.} A nonzero tautological
integral at a vertex $v$ satisfies
\[
  \sum_{h\in H(v)}k(h)=3g(v)-3+\val(v).
\]
Summing over vertices and using
$g=\sum_v g(v)+|E(\Gamma)|-|V(\Gamma)|+1$, we get
\[
  k_{\rm tot}=3g-3-|E(\Gamma)|+n+m.
\]
The exponent contributed by the three parity factors is
$k_{\rm tot}+|E(\Gamma)|+n+m$, hence it is
congruent to $g-1$ modulo $2$.  Thus the product of the local parity signs is
$(-1)^{g-1}$.
\end{proof}

{
\begin{remark}[A primary example]\label{rem:remodeling-consistency-checks}
On a primary $(0,3)$ graph, the three height-zero ordinary leaves contribute
$(-1)^3=-1$, which cancels the genus-zero DOSS prefactor.
\end{remark}
}

\begin{proposition}[Ancestor graph comparison]\label{prop:ancestor-graph-comparison}
For $n>0$, $2g-2+n>0$, and $\kappa\in\Omega_B$,
\[
  \cW^{B,\Pi}_{g,n}(\widetilde\bu_1,\ldots,\widetilde\bu_n;\kappa)
  =
  \mathcal A^{\cX,{\rm an}}_{g,n}
  (\widetilde\bu_1,\ldots,\widetilde\bu_n;\kappa).
\]
\end{proposition}

{
\begin{proof}
The Frobenius isomorphism identifies the vertex factors through
\[
  {\check h^\alpha_1\over\sqrt{-2}}
  ={1\over\sqrt{\Delta^\cX_\alpha}}=v_\alpha.
\]
Theorem~\ref{thm:anchor} identifies the $R$-matrices, and
Proposition~\ref{prop:leaf-dressing} expresses the ordinary B-side leaves in the undressed
variables. For the dilaton leaves, Proposition~\ref{prop:prym-unit-translation} gives
\[
  D^B_{\alpha,k}
  =-{\check h^\alpha_k\over\sqrt{-2}}
  =-[z^k]\,z\sum_\beta v_\beta R_{\beta\alpha}(z),\qquad k\geq2.
\]
After the substitution $\Pi U(z)=-U(-z)$, Lemma~\ref{lem:graph-parity} shows that the product
of local B-side weights is $(-1)^{g-1}$ times the A-side product. This cancels the global
DOSS factor $(-1)^{g-1}$. The graph decorations and automorphism factors agree, so summing
gives the identity.
\end{proof}
}

\begin{theorem}[Stable descendant remodeling]\label{thm:formal-graphsum-remodeling}
For $n>0$, $2g-2+n>0$, and $\kappa\in\Omega_B$,
\[
  \cW^{B,\Pi}_{g,n}(\widetilde\bu_1,\ldots,\widetilde\bu_n;\kappa)
  =
  \mathcal F^{\cX,{\rm an}}_{g,n}(\bu_1,\ldots,\bu_n;\kappa),
\]
{
The components $\widetilde u_j^\alpha$ are given by \eqref{eq:remodeling:transformed-input}
using $S_x^{\rm an}$. {The equality holds in the completed ring in the descendant variables.}
}
\end{theorem}

{
\begin{proof}
Apply Proposition~\ref{prop:ancestor-graph-comparison} to the transformed inputs and then
use Lemma~\ref{lem:stable-ancestor-descendant}, continued analytically by
Proposition~\ref{prop:aside-analytic-graph-theory}.
\end{proof}
}

\subsection{The free-energy specialization}

{The closed higher-genus free energies follow from the ancestor graph comparison with one ordinary
external leaf and {the dilaton specialization using the flat unit}.}

Let $\Phi_\kappa$ be a local primitive of the one-form with values in the sign local system
$\widetilde y\,d\hat x$ on the reduced local curve $\cS^\iota_\kappa$:
\[
  d\Phi_\kappa=\widetilde y\,d\hat x .
\]
Changing the local sign trivialization changes both $\Phi_\kappa$ and
$\omega^\iota_{g,1}$ by a sign, so their product is well-defined.  Adding a
constant to $\Phi_\kappa$ does not change the residues below, because stable
topological-recursion one-forms have zero residues at the ramification points.  For
$g\geq2$ we use the standard topological-recursion normalization
\cite{eynard2007invariantsalgebraiccurvestopological,
eynard2009geometricalinterpretationtopologicalrecursion}
\[
  F_g(\cS^\iota_\kappa)
  ={1\over 2-2g}\sum_{\alpha}\Res_{p\to p_\alpha}
  \Phi_\kappa(p)\,\omega^\iota_{g,1}(p),
\]
and write
\begin{equation}
\label{eq:remodeling:tr-dilaton}
  -\sum_{\alpha}\Res_{p\to p_\alpha}
  \Phi_\kappa(p)\,\omega^\iota_{g,1}(p)
  =(2g-2)F_g(\cS^\iota_\kappa).
\end{equation}
We also use the residue functional
\[
  \ell_{\rm dil}(\eta)
  :=
  -\sum_\beta\Res_{p\to p_\beta}\Phi_\kappa(p)\,\eta(p).
\]

\begin{lemma}[Residue pairing with dressed ordinary leaves]\label{lem:residue-unit-leaf}
For the dressed ordinary leaf $\theta^k_\alpha$,
\[
  -\sum_\beta\Res_{p\to p_\beta}\Phi_\kappa(p)\,\theta^k_\alpha(p)
  =
  \begin{cases}
  0,& k=0,\\[2mm]
  \dfrac{\check h^\alpha_k}{\sqrt{-2}},& k\geq1.
  \end{cases}
\]
\end{lemma}

\begin{proof}
Only the residue at $p_\alpha$ can contribute.  In the local coordinate
$s=s_\alpha$, the odd combination $\widetilde y(s)-\widetilde y(-s)$ is
\[
  \widetilde y(s)-\widetilde y(-s)
  =
  2\sum_{m\geq1}
  {2^{m-1}\over(2m-1)!!}\check h^\alpha_m s^{2m-1}.
\]
Since $d\hat x=2s\,ds$, an odd primitive has expansion
\[
  \Phi_{\rm odd}(s)
  =
  \sum_{m\geq1}
  {2^m\over(2m+1)!!}\check h^\alpha_m s^{2m+1}.
\]
The principal part of the bare second-kind form is
\[
  \vartheta^k_\alpha
  =
  -{(2k+1)!!\over 2^k}
  {ds\over s^{2k+2}}
  +\text{\rm holomorphic terms}.
\]
Thus $\Res_{p\to p_\alpha}\Phi_\kappa\,\vartheta^k_\alpha=0$ for $k=0$, and it is
$-\check h^\alpha_k$ for $k\geq1$.  Using the ordinary-leaf convention
$\theta^k_\alpha=\vartheta^k_\alpha/\sqrt{-2}$ from
\eqref{eq:bgraphsum:ordinary-leaf} and the minus sign in the
definition of the residue functional gives the two cases in the statement.
\end{proof}

{
The residue functional in Lemma~\ref{lem:residue-unit-leaf} realizes the ordinary ancestor
insertion $z\mathbf1$. Writing
}
\[
  \mathbf 1=\sum_\beta v_\beta\hat\phi_\beta,
  \qquad
  v_\beta={1\over\sqrt{\Delta^\cX_\beta}},
\]
{we obtain from Proposition~\ref{prop:prym-unit-translation}}
\begin{equation}
\label{eq:remodeling:unit-leaf-rmatrix}
  [z^k]\sum_\beta
  z\,v_\beta\,\hat R^B_{\beta\alpha}(z)
  =
  \begin{cases}
  0,& k=0,\\[2mm]
  \dfrac{\check h^\alpha_k}{\sqrt{-2}},& k\geq1.
  \end{cases}
\end{equation}

\begin{lemma}[Residue pairing with undressed ordinary leaves]\label{lem:undressed-unit-specialization}
For the residue functional $\ell_{\rm dil}$ defined above,
\[
  \ell_{\rm dil}(\widehat\theta_\beta^i)=
  \begin{cases}
  v_\beta,& i=1,\\
  0,& i\neq1.
  \end{cases}
\]
\end{lemma}

\begin{proof}
By Lemma~\ref{lem:residue-unit-leaf},
\[
  \ell_{\rm dil}(\theta_\alpha(z))
  =
  \sum_{k\geq1}{\check h^\alpha_k\over\sqrt{-2}}z^k
  =
  z\sum_\beta v_\beta\hat R^B_{\beta\alpha}(z),
\]
where the last equality is Proposition~\ref{prop:prym-unit-translation}.  On
the other hand, Proposition~\ref{prop:leaf-dressing} gives
\[
  \theta_\alpha(z)
  =
  \sum_\beta\hat R^B_{\beta\alpha}(z)\widehat\theta_\beta(z).
\]
Since $\hat R^B(0)=I$, the coefficient vector
$\bigl(\ell_{\rm dil}(\widehat\theta_\beta(z))\bigr)_\beta$ is determined
uniquely.  Comparing the two displays gives
\[
  \ell_{\rm dil}(\widehat\theta_\beta(z))=z\,v_\beta.
\]
\end{proof}

Let $\Omega^{\rm an}_{g,n}$ denote the coefficientwise analytic continuation
of $\Omega^A_{g,n}$ supplied by
Proposition~\ref{prop:aside-analytic-graph-theory}.  {For $g\geq2$, define} the genus-$g$
analytically continued shifted free energy by
\begin{equation}
\label{eq:remodeling:shifted-free-energy}
  F_g^{\mathrm{GW},{\rm an}}(\cX;\kappa)
  =
  \int_{\Mbar_g}\Omega^{\rm an}_{g,0}(\kappa).
\end{equation}
{
The Taylor expansion of $F_g^{\mathrm{GW},{\rm an}}$ is the formal shifted GW free energy.
The shifted CohFT has flat unit, so for the forgetful map
$\pi:\Mbar_{g,1}\to\Mbar_g$ its unit axiom gives
\[
  \Omega^{\rm an}_{g,1}(\mathbf1;\kappa)
  =\pi^*\Omega^{\rm an}_{g,0}(\kappa).
\]
This identity follows from the formal unit axiom by coefficientwise continuation.
The projection formula and $\pi_*\psi_1=2g-2$ therefore yield
\begin{equation}\label{eq:remodeling:shifted-dilaton}
  \int_{\Mbar_{g,1}}\psi_1\,\Omega^{\rm an}_{g,1}(\mathbf1;\kappa)
  =(2g-2)F_g^{\mathrm{GW},{\rm an}}(\cX;\kappa),\qquad g\geq2.
\end{equation}
}

\begin{theorem}[Higher-genus free energies]\label{thm:formal-free-energy-remodeling}
For every $g\geq2$ and every $\kappa\in\Omega_B$,
\[
  F_g(\cS^\iota_\kappa)
  =
  F_g^{\mathrm{GW},{\rm an}}(\cX;\kappa).
\]
\end{theorem}

\begin{proof}
{We apply the residue functional $\ell_{\rm dil}$ graph by graph to the one
ordinary external leaf of the DOSS expansion of $\omega^\iota_{g,1}$.
Lemma~\ref{lem:undressed-unit-specialization} identifies this specialization
of the parity-twisted undressed variables with the A-side ordinary input
$z\mathbf 1$,} because the height-one parity factor is $+1$.  Hence
Proposition~\ref{prop:ancestor-graph-comparison}, with this specific ordinary
input, and the identity \eqref{eq:remodeling:tr-dilaton} give
\[
  (2g-2)F_g(\cS^\iota_\kappa)
  =
  \int_{\Mbar_{g,1}}\psi_1\,
  \Omega^{\rm an}_{g,1}(\mathbf 1;\kappa).
\]
Using the shifted-CohFT dilaton equation \eqref{eq:remodeling:shifted-dilaton} on the A-side, we get
\[
  (2g-2)F_g(\cS^\iota_\kappa)
  =
  (2g-2)F_g^{\mathrm{GW},{\rm an}}(\cX;\kappa).
\]
{Since $2g-2\neq0$, division by $2g-2$ gives the equality.}
\end{proof}

{
\subsection{{The genus-one free energy}}\label{subsec:genus-one-free-energy}

We use the reduced kernel and the phase convention of this paper; for
related free-energy constructions, see
\cite{eynard2007invariantsalgebraiccurvestopological,borot2015abstractloopequations,
brini2020e8spectralcurves,fang2019remodelingconjecturetoriccalabiyau}.
In this subsection, primitives are taken on a simply connected
branch of $\Omega_B^{\rm cov}$.
The primary A-model potential is the formal series
\[
  F_1^{\mathrm{GW}}(\cX;x)
  =\sum_{m\geq1}{1\over m!}
  \langle\mathbf x,\ldots,\mathbf x\rangle^{\cX,\bT}_{1,m}.
\]
The unmarked degree-zero term is unstable and is set to zero.
Proposition~\ref{prop:aside-analytic-graph-theory}, applied to $(g,n)=(1,1)$,
supplies the analytic continuation of $dF_1^{\mathrm{GW}}$.
We write $F_1^{\mathrm{GW},{\rm an}}(\cX;\kappa)$ for a primitive on the
chosen branch.

Write $\hat R^B(z)=I+R_1^Bz+O(z^2)$.  Define the \emph{reduced calibrated
tau function}, up to a nonzero constant, by
\begin{equation}\label{eq:remodeling:reduced-tau}
  d\log\tau_B^-
  =\sum_\alpha
  \left(\Res_{p\to p_\alpha}{B^-(p,\bar p)\over d\hat x(p)}\right)
  d\bigl(\hat x(p_\alpha)\bigr)
  =-\sum_\alpha(R_1^B)_{\alpha\alpha}\,du_\alpha.
\end{equation}
These are ordinary residues on the $l+1$ representative disks, with
$\bar p$ the local conjugate $s_\alpha\mapsto-s_\alpha$.  Indeed, the
kernel expansion and Proposition~\ref{prop:prym-double-laplace} give
\[
  \Res_{p\to p_\alpha}{B^-(p,\bar p)\over d\hat x(p)}
  =-{1\over2}B^{\alpha\alpha}_{0,0}
  =-\check B^{\alpha\alpha}_{0,0}
  =(R_1^B)_{\alpha\alpha},
  \qquad d\bigl(\hat x(p_\alpha)\bigr)=-du_\alpha.
\]
For existence, symplecticity makes $R_1^B$ symmetric, and the normalized
Dubrovin equation of Proposition~\ref{prop:smooth-chamber-calibration} gives
\[
  d(R_1^B)_{\alpha\alpha}
  =\sum_\beta(R_1^B)_{\alpha\beta}^{\,2}(du_\beta-du_\alpha).
\]
Thus the one-form in \eqref{eq:remodeling:reduced-tau} is closed by pairwise
cancellation.  With logarithms on the chosen branch, set
(cf.~\cite[\S4.2.4]{eynard2007invariantsalgebraiccurvestopological})
\begin{equation}\label{eq:remodeling:genus-one-potential}
  F_1^B=-{1\over2}\log\tau_B^-
        -{1\over24}\sum_\alpha\log\check h_1^\alpha.
\end{equation}

\begin{proposition}[Genus-one free energy]\label{prop:genus-one-free-energy}
On the chosen branch,
\begin{equation}\label{eq:remodeling:genus-one-differential}
  dF_1^B=dF_1^{\mathrm{GW},{\rm an}}(\cX;\kappa)
  ={1\over48}\sum_\alpha d\log\Delta_\alpha^\cX
   +{1\over2}\sum_\alpha(R_1^B)_{\alpha\alpha}\,du_\alpha.
\end{equation}
\end{proposition}

\begin{proof}
The normalization \eqref{eq:bgraphsum:gaussian-normalization} gives
$d\log\check h_1^\alpha=-\frac12d\log\Delta_\alpha^\cX$, so
\eqref{eq:remodeling:reduced-tau}--\eqref{eq:remodeling:genus-one-potential}
give the B-side differential.

For the A-side twisted $B\Gamma$ theory
\cite[Thms.~4.3--4.4]{lan2023twistedequivariantgromovwittentheory}, with
calibration fixed in Proposition~\ref{prop:rmatrix}, differentiating the
primary series gives
$\partial_{x_a}F_1^{\mathrm{GW}}=\int_{\Mbar_{1,1}}\Omega^A_{1,1}(T_a)$.
There are three contributing graph types in
Proposition~\ref{prop:aside-graph-weights}
\cite{givental2001semisimplefrobeniusstructureshigher,teleman2012structure2dsemisimplefield}:
a genus-one vertex with a
height-one ordinary leaf, a genus-one vertex with one height-two dilaton
leaf, and a genus-zero vertex with a loop.  Using
$\int_{\Mbar_{1,1}}\psi_1=\int_{\Mbar_{1,2}}\psi_2^2=1/24$ and the loop
automorphism factor $1/2$, their contributions give
\[
  dF_1^{\mathrm{GW},{\rm an}}
  ={1\over24}\sum_{\alpha,\beta}{v_\beta\over v_\alpha}
     R_{1,\alpha\beta}(du_\alpha-du_\beta)
   +{1\over2}\sum_\alpha R_{1,\alpha\alpha}\,du_\alpha.
\]
Here $v_\alpha=(\Delta_\alpha^\cX)^{-1/2}$ and $R_1=[z]R(z)$ as in \eqref{eq:remodeling:shared-rmatrix}; the ordinary
height-one coefficient uses $-R_1$ from $R(-z)^{\mathsf T}$.
The column convention for $\Psi^B$ and flatness of the unit give
$dv_\alpha=\sum_\beta R_{1,\alpha\beta}v_\beta(du_\beta-du_\alpha)$.
The first sum is therefore $-\frac1{24}\sum_\alpha d\log v_\alpha$,
proving \eqref{eq:remodeling:genus-one-differential}
(cf.~\cite{givental1998ellipticgromovwitten,zong2016equivariantgkm}).
\end{proof}

The unit coordinate is $x_{l+1}=2\nu\log(K/K_0)$. Since
$\partial_{x_{l+1}}u_\alpha=1$ and the metric is unit-invariant,
\eqref{eq:remodeling:genus-one-differential} gives
$\partial_{x_{l+1}}F_1^B=\frac12\operatorname{tr}R_1^B$.

\begin{remark}
The Frobenius-manifold isomorphism of Theorem~\ref{thm:frob-iso} already identifies
the genus-zero prepotential $F_0$ locally, up to a polynomial of degree at most two
in flat coordinates.  We do not attempt to reconstruct the spectral-curve
free energy $F_0$ of Eynard--Orantin
\cite[\S4.2.4]{eynard2007invariantsalgebraiccurvestopological}.
\end{remark}
}

\appendix
\begingroup

\section{{Boundary labels and character formulas}}
\label{app:boundary-characters}

The Frobenius-algebra extension of Corollary~\ref{cor:frob-identification-extension}
identifies the boundary idempotents abstractly. In this appendix we make their representation labels
explicit and record the resulting character- and sector-basis formulas. The basis correspondence
uses the affine mirror map of Section~\ref{sec:frobenius}; the boundary $R$-matrix formulas are
consequences of Theorem~\ref{thm:anchor} and the A-side orbifold calculation in
Section~\ref{sec:amodel}. These results are supplementary to the proofs in
Sections~\ref{sec:anchor} and~\ref{sec:graphsum-remodeling}.

Fix the presentation
\[
  \Gamma=\langle a,b\mid a^{2(l-2)}=1,\ b^2=a^{l-2},\ bab^{-1}=a^{-1}\rangle,
\]
in which $a$ acts on the defining representation $V$ with eigenvalues
{$e^{\pm\pi\sqrt{-1}/(l-2)}$. For $1\leq s\leq l-3$, let $\rho_s$ be the
two-dimensional irreducible representation whose character satisfies
$\chi_{\rho_s}(a)=2\cos\bigl(\pi s/(l-2)\bigr)$, so that $\rho_1\simeq V$. The four remaining
irreducible representations are one-dimensional, two with character value $1$ on $a$ and two with
character value $-1$ on $a$. For a fixed point $r=\pm1$, let $\chi_{r,+}$ and $\chi_{r,-}$ be the
characters of the two one-dimensional representations satisfying}
\[
  \chi_{r,\pm}(a)=-r;
\]
{Let $e_{r,\pm}:=e_{\chi_{r,\pm}}$ be the corresponding primitive central idempotents of
\eqref{eq:amodel:primitive-idempotents}. For the normalized canonical vectors, put
\[
  \widetilde e_{r,\pm}:=\sqrt{\Delta^\cX_{r,\pm}}\,e_{r,\pm},
  \qquad
  \widetilde e_r^{\rm inv}:={\widetilde e_{r,+}+\widetilde e_{r,-}\over\sqrt2},
  \qquad
  \widetilde e_r^{\rm anti}:={\widetilde e_{r,+}-\widetilde e_{r,-}\over\sqrt2}.
\]
The first combination is invariant and the second anti-invariant under the exchange of the two
idempotents.}

\begin{lemma}[{Special-fiber basis correspondence}]\label{lem:orbifold-label-dictionary}
Let $m=l-2$. Post-composing the special-fiber mirror isomorphism
$\Phi_{\rm ext}:\cH_B(\kappa^{\rm orb})\to Z(\bC(\nu)[\Gamma])$ of
Corollary~\ref{cor:frob-identification-extension} with the class-algebra identification of
Lemma~\ref{lem:frob}, the class of the invariant coordinate $w=\mu+\mu^{-1}$ satisfies
\begin{equation}\label{eq:anchor:class-sum-dictionary}
  \Phi_{\rm ext}\bigl([\,w\,]\bigr)=-(a+a^{-1}).
\end{equation}
{Consequently, $\Phi_{\rm ext}$ maps the $w_0$-eigenspace of multiplication by $[w]$ to the
span of the primitive central idempotents $e_\alpha$ whose associated irreducible representations
satisfy $2\chi_\alpha(a)/d_\alpha=-w_0$:}
\begin{enumerate}
\item the free main orbit through $\mu_0=e^{\pi\sqrt{-1}\,s/m}$, with $w_0=2\cos(\pi s/m)$,
  {maps to the line spanned by $e_{\rho_{m-s}}$, where $\rho_{m-s}$ is the irreducible
  representation in which $a$ has eigenvalues $-\mu_0$ and $-\mu_0^{-1}$;}
\item {for the fixed point $r=\pm1$, the $w_0=2r$ eigenspace maps to the span of
  $e_{r,+}$ and $e_{r,-}$. These idempotents correspond to the two one-dimensional representations
  whose characters satisfy $\chi(a)=-r$.}
\end{enumerate}
{A choice of square-root branch for the pole-cancellation coordinate $\varepsilon_r$ orders the
fixed pair; the deck transformation $\varepsilon_r\mapsto-\varepsilon_r$ exchanges
$e_{r,+}$ and $e_{r,-}$.}
\end{lemma}

\begin{proof}
By Definition~\ref{def:mirror-map} the flat coordinate $x_{[g]}$ is dual to the age-normalized
sector class $\overline{\mathbf 1}_{[g]}$, and under the isomorphism of Lemma~\ref{lem:frob} this
class is the class sum $\sum_{h\in[g]}h\in Z(\bC(\nu)[\Gamma])$. Since
$\Phi_{\rm ext}=d\Phi\circ\widetilde{\mathrm{KS}}{}^{-1}$, the special-fiber Kodaira--Spencer class
of every sector field is sent to the corresponding class sum,
\[
  \Phi_{\rm ext}\bigl(\bigl[\,2\nu\,\partial_{x_{[g]}}\log\lambda\,\bigr]\bigr)
  =\sum_{h\in[g]}h .
\]
It therefore suffices to prove that, on the special fiber,
\begin{equation}\label{eq:anchor:ks-of-rotation-sector}
  \bigl[\,2\nu\,\partial_{x_{[a]}}\log\lambda\,\bigr]=-[\,w\,]
\end{equation}
for the rotation class $[a]=\{a,a^{-1}\}$.

We first record the McKay placement needed to read off the relevant column of the affine change of
variables. The mark computation in Remark~\ref{rem:hu-coordinate-check} {places the unique end of the chain with Dynkin mark $1$ at the node $j=1$ and the fork
at $j=l-1,l$.} Tensoring by $V$ gives
$V\otimes\chi'=\rho_1$ for the one-dimensional character $\chi'$ with $\chi'(a)=1$,
$V\otimes\rho_t=\rho_{t-1}\oplus\rho_{t+1}$ in the interior of the chain, and
$V\otimes\rho_{m-1}$ contains both characters with $\chi(a)=-1$. Hence for $l\geq5$ the McKay
correspondence forces $\chi_1=\chi'$, $\chi_j=\rho_{j-1}$ for $2\leq j\leq l-2$, and places the
two characters with $\chi(a)=-1$ at the fork; for $l=4$ we fix the same placement among the
triality-related choices. In particular
\[
  \chi_j(a)=
  \begin{cases}
    1, & j=1,\\
    2\cos\bigl(\pi(j-1)/m\bigr), & 2\leq j\leq l-2,\\
    -1, & j=l-1,\ l .
  \end{cases}
\]

{Let $\xi=(\xi_i)$ denote the coefficient vector obtained by inverting the $[a]$ column of the
affine change of variables. Equations~\eqref{eq:frob:raw-hu-map} and
\eqref{eq:frob:age-normalized-map} give
\[
  \partial_{x_{[a]}}
  =-{\sqrt{-1}\over\nu}\sum_{i=1}^{l}\xi_i\,\partial_{\btau_i},
  \qquad
  G^{\mathsf T}\xi=\mathsf L_{\bullet,[a]},
  \qquad
  \mathsf L_{j,[a]}={\sin(\pi/(2m))\over m}\chi_j(a).
\]
Put $\psi_k=(2k-1)\pi/(2m)$. At the orbifold point,
\[
  c_i=-2\cos\psi_{i-1},
  \qquad
  \kappa_i^{-1}-\kappa_i=2\sqrt{-1}\sin\psi_{i-1}
  \quad(2\le i\le l-1),
\]
and $\{2\cos\psi_k\}_{k=1}^m$ is the root set of
$\mathcal Q(w)=2T_m(w/2)$. The triangular system for $G^{\mathsf T}$ and the Dirichlet-kernel
identity give
\[
  \xi_1=\xi_l=0,
  \qquad
  \xi_i=-{\sin\psi_{i-1}\over m}
  \quad(2\le i\le l-1).
\]
Thus the pole-cancellation directions do not occur, and
\[
\begin{aligned}
  2\nu\,\partial_{x_{[a]}}\log\lambda
  &=-\sqrt{-1}\sum_{i=2}^{l-1}\xi_i
    {2(\kappa_i^{-1}-\kappa_i)\over w-c_i} \\
  &=-{4\over m}\sum_{k=1}^{m}{\sin^2\psi_k\over w+2\cos\psi_k}
   =-{4\over m}\sum_{k=1}^{m}{\sin^2\psi_k\over w-2\cos\psi_k}.
\end{aligned}
\]
Using $4\sin^2\psi_k=(4-w^2)+(w-2\cos\psi_k)(w+2\cos\psi_k)$ yields
\[
  {4\over m}\sum_{k=1}^{m}{\sin^2\psi_k\over w-2\cos\psi_k}
  =w-{(w^2-4)\mathcal Q'(w)\over m\mathcal Q(w)}.
\]
The correction term vanishes in the limiting algebra. On a residual summand,
$\mathcal Q'=0$ and $\mathcal Q\ne0$; on the summand over $w=2r$, the factor $w^2-4$ vanishes while
$\mathcal Q(2r)=2r^m\ne0$. More explicitly,
\[
  t=b_r+\varepsilon_r\mathfrak a_r(e_{r,+}-e_{r,-}),
\]
where $\mathfrak a_r$ is the Weierstrass unit, so both components of the correction tend to zero with
$b_r$ and $\varepsilon_r$. This proves \eqref{eq:anchor:ks-of-rotation-sector}.}

Finally, we match eigenvalues. The class $[\,w\,]$ acts by $2\cos(\pi s/m)$ on the residual summand
through $\mu_0=e^{\pi\sqrt{-1}\,s/m}$ and by $2r$ on both idempotents of the fixed-node summand
at $r$, while the class sum $a+a^{-1}$ acts on the $\alpha$-isotypic block by
$2\chi_\alpha(a)/d_\alpha$. Matching under \eqref{eq:anchor:class-sum-dictionary} gives
$2\chi_\alpha(a)/d_\alpha=-w_0$. The residual values select
$2\chi_\alpha(a)/d_\alpha=2\cos\bigl(\pi(m-s)/m\bigr)$, hence $\alpha=\rho_{m-s}$; the
{fixed-node values select the one-dimensional representations whose characters satisfy
$\chi(a)=-r$.} {The exchange within each pair under} $\varepsilon_r\mapsto-\varepsilon_r$ is the sign discussion in Step 2 of
Proposition~\ref{prop:bmodel-flat-algebra-family}.
\end{proof}

{
\begin{corollary}[Character- and sector-basis boundary formulas]\label{cor:postcomparison-support}
Let $\hat R^{B,\mathrm{bdry}}$ denote the coefficientwise boundary value of the removable analytic
continuation furnished by Theorem~\ref{thm:anchor}. In the normalized canonical basis indexed by
$\alpha,\beta\in\operatorname{Irr}(\Gamma)$,
\[
  \hat R^{B,\mathrm{bdry}}_{\alpha\beta}(z)
  =
  {E(z)\over|\Gamma|}\sum_{h\in\Gamma}
  \chi_\alpha(h^{-1})\chi_\beta(h)\Theta_h(z).
\]
In the normalized conjugacy-sector basis $\{\epsilon_{[h]}\}$ of
\eqref{eq:amodel:character-transform}, the same operator is diagonal:
\[
  \hat R^{B,\mathrm{bdry}}(z)\epsilon_{[h]}
  =\Theta_h(z)E(z)\epsilon_{[h]}.
\]
For a fixed node $r$, the normalized anti-invariant vector has the sector expansion
\[
  \widetilde e_r^{\rm anti}
  =
  {1\over\sqrt2}
  \sum_{[h]\in\operatorname{Conj}(\Gamma)}
  {\chi_{r,+}(h^{-1})-\chi_{r,-}(h^{-1})
   \over\sqrt{|C_\Gamma(h)|}}\,\epsilon_{[h]}.
\]
The two characters in the fixed pair agree on the rotation subgroup $\langle a\rangle$, so the
sector expansion of $\widetilde e_r^{\rm anti}$ is supported on the reflection conjugacy sectors.
\end{corollary}

\begin{proof}
By Theorem~\ref{thm:anchor}, the smooth B-model $R$-matrix is the analytic continuation of the A-side
formal calibration. Taking the A-side coefficientwise boundary value from
Proposition~\ref{prop:aside-chamber-calibration} gives
$\hat R^{B,\mathrm{bdry}}=\hat R^\cX_{\rm orb}=E\hat R^{A,\mathrm{surf}}$.
Proposition~\ref{prop:rmatrix} gives the character-basis formula. Equations
\eqref{eq:amodel:character-transform} and \eqref{eq:amodel:inverse-character-transform} give the
diagonal sector-basis formula. Applying the first of these transformations to
$(\widetilde e_{r,+}-\widetilde e_{r,-})/\sqrt2$ gives the displayed expansion of
$\widetilde e_r^{\rm anti}$; its rotation-sector coefficients vanish because
$\chi_{r,+}|_{\langle a\rangle}=\chi_{r,-}|_{\langle a\rangle}$.
\end{proof}
}

\endgroup

\bigskip
\noindent\textbf{Conflict of interest.}
The authors declare no financial or non-financial competing interests.

\medskip
\noindent\textbf{Data availability.}
Data sharing is not applicable to this article as no datasets were
generated or analysed during the current study.

\bibliographystyle{amsalpha-arxiv}
\bibliography{references}

@article{kontsevich1997relationscorrelatorstopological,
      title={Relations between the correlators of the topological sigma-model coupled to gravity},
      author={Maxim Kontsevich and Yuri I. Manin},
      journal={Communications in Mathematical Physics},
      volume={196},
      number={2},
      pages={385--398},
      year={1998},
      doi={10.1007/s002200050426},
      url={https://doi.org/10.1007/s002200050426},
}

@incollection{givental1998ellipticgromovwitten,
      title={Elliptic {Gromov--Witten} invariants and the generalized mirror conjecture},
      author={Alexander B. Givental},
      booktitle={Integrable Systems and Algebraic Geometry (Kobe/Kyoto, 1997)},
      pages={107--155},
      publisher={World Scientific},
      address={River Edge, NJ},
      year={1998},
      url={https://arxiv.org/abs/math/9803053},
}

@misc{zong2016equivariantgkm,
      title={Equivariant {Gromov--Witten} Theory of {GKM} Orbifolds},
      author={Zhengyu Zong},
      year={2016},
      eprint={1604.07270},
      archivePrefix={arXiv},
      primaryClass={math.AG},
      url={https://arxiv.org/abs/1604.07270},
}

@article{givental2001semisimplefrobeniusstructureshigher,
      title={Semisimple Frobenius structures at higher genus},
      author={Alexander B. Givental},
      journal={International Mathematics Research Notices},
      volume={2001},
      number={23},
      pages={1265--1286},
      year={2001},
      doi={10.1155/S1073792801000605},
      url={https://doi.org/10.1155/S1073792801000605},
}

@incollection{jarvis2002orbifoldquantumcohomologyclassifying,
      title={Orbifold quantum cohomology of the classifying space of a finite group}, 
      author={Tyler J. Jarvis and Takashi Kimura},
      booktitle={Orbifolds in Mathematics and Physics},
      series={Contemporary Mathematics},
      volume={310},
      pages={123--134},
      publisher={American Mathematical Society},
      address={Providence, RI},
      year={2002},
      doi={10.1090/conm/310/05401},
      url={https://doi.org/10.1090/conm/310/05401},
}

@article{tseng2009orbifoldquantumriemannrochlefschetz,
      title={Orbifold Quantum Riemann--Roch, Lefschetz and Serre},
      author={Hsian-Hua Tseng},
      journal={Geometry \& Topology},
      volume={14},
      number={1},
      pages={1--81},
      year={2010},
      doi={10.2140/gt.2010.14.1},
      url={https://doi.org/10.2140/gt.2010.14.1},
}

@article{bryan2007rootsystemsquantumcohomology,
      title={Root systems and the quantum cohomology of ADE resolutions},
      author={Jim Bryan and Amin Gholampour},
      journal={Algebra \& Number Theory},
      volume={2},
      number={4},
      pages={369--390},
      year={2008},
      doi={10.2140/ant.2008.2.369},
      url={https://doi.org/10.2140/ant.2008.2.369},
}

@article{bryan2008quantummckaycorrespondencepolyhedral,
      title={The quantum McKay correspondence for polyhedral singularities},
      author={Jim Bryan and Amin Gholampour},
      journal={Inventiones mathematicae},
      volume={178},
      number={3},
      pages={655--681},
      year={2009},
      doi={10.1007/s00222-009-0212-8},
      url={https://doi.org/10.1007/s00222-009-0212-8},
}

@article{eynard2007invariantsalgebraiccurvestopological,
      title={Invariants of algebraic curves and topological expansion}, 
      author={Bertrand Eynard and Nicolas Orantin},
      journal={Communications in Number Theory and Physics},
      volume={1},
      number={2},
      pages={347--452},
      year={2007},
      doi={10.4310/CNTP.2007.v1.n2.a4},
      url={https://doi.org/10.4310/CNTP.2007.v1.n2.a4},
}

@article{borot2015abstractloopequations,
      title={Abstract loop equations, topological recursion and new applications},
      author={Ga{\"e}tan Borot and Bertrand Eynard and Nicolas Orantin},
      journal={Communications in Number Theory and Physics},
      volume={9},
      number={1},
      pages={51--187},
      year={2015},
      doi={10.4310/CNTP.2015.v9.n1.a2},
      url={https://doi.org/10.4310/CNTP.2015.v9.n1.a2},
}

@article{brini2020e8spectralcurves,
      title={{E8} spectral curves},
      author={Andrea Brini},
      journal={Proceedings of the London Mathematical Society},
      volume={121},
      number={4},
      pages={954--1032},
      year={2020},
      doi={10.1112/plms.12331},
      url={https://doi.org/10.1112/plms.12331},
}

@article{teleman2012structure2dsemisimplefield,
      title={The structure of 2D semi-simple field theories}, 
      author={Constantin Teleman},
      journal={Inventiones mathematicae},
      volume={188},
      number={3},
      pages={525--588},
      year={2012},
      doi={10.1007/s00222-011-0352-5},
      url={https://doi.org/10.1007/s00222-011-0352-5},
}

@article{duninbarkowski2012identificationgiventalformulaspectral,
      title={Identification of the Givental Formula with the Spectral Curve Topological Recursion Procedure},
      author={Petr Dunin-Barkowski and Nicolas Orantin and Sergey Shadrin and Loek Spitz},
      journal={Communications in Mathematical Physics},
      volume={328},
      number={2},
      pages={669--700},
      year={2014},
      doi={10.1007/s00220-014-1887-2},
      url={https://doi.org/10.1007/s00220-014-1887-2},
}

@article{hu2012quantummckaycorrespondencesingularities,
      title={The quantum McKay correspondence for singularities of type D},
      author={Xiaowen Hu},
      journal={Advances in Mathematics},
      volume={247},
      pages={266--308},
      year={2013},
      doi={10.1016/j.aim.2013.07.013},
      url={https://doi.org/10.1016/j.aim.2013.07.013},
}

@article{eynard2012computationopengromovwitten,
      title={Computation of Open Gromov--Witten Invariants for Toric Calabi--Yau 3-Folds by Topological Recursion, a Proof of the BKMP Conjecture},
      author={Bertrand Eynard and Nicolas Orantin},
      journal={Communications in Mathematical Physics},
      volume={337},
      number={2},
      pages={483--567},
      year={2015},
      doi={10.1007/s00220-015-2361-5},
      url={https://doi.org/10.1007/s00220-015-2361-5},
}

@article{fang2013allgenusopenclosedmirror,
      title={All-genus open-closed mirror symmetry for affine toric Calabi--Yau 3-orbifolds},
      author={Bohan Fang and Chiu-Chu Melissa Liu and Zhengyu Zong},
      journal={Algebraic Geometry},
      volume={7},
      number={2},
      pages={192--239},
      year={2020},
      doi={10.14231/ag-2020-007},
      url={https://doi.org/10.14231/ag-2020-007},
}

@article{fang2016eynardorantinrecursionequivariantmirror,
      title={The Eynard--Orantin recursion and equivariant mirror symmetry for the projective line},
      author={Bohan Fang and Chiu-Chu Melissa Liu and Zhengyu Zong},
      journal={Geometry \& Topology},
      volume={21},
      number={4},
      pages={2049--2092},
      year={2017},
      doi={10.2140/gt.2017.21.2049},
      url={https://doi.org/10.2140/gt.2017.21.2049},
}

@article{fang2019remodelingconjecturetoriccalabiyau,
      title={On the remodeling conjecture for toric Calabi--Yau 3-orbifolds},
      author={Bohan Fang and Chiu-Chu Melissa Liu and Zhengyu Zong},
      journal={Journal of the American Mathematical Society},
      volume={33},
      number={1},
      pages={135--222},
      year={2020},
      doi={10.1090/jams/934},
      url={https://doi.org/10.1090/jams/934},
}

@misc{lan2023twistedequivariantgromovwittentheory,
      title={Twisted Equivariant Gromov-Witten Theory of the Classifying Space of a Finite Group}, 
      author={Zhuoming Lan and Zhengyu Zong},
      year={2023},
      eprint={2309.01473},
      archivePrefix={arXiv},
      primaryClass={math.AG},
      url={https://arxiv.org/abs/2309.01473}, 
}

@article{brini2025dubrovindualitymirrorsymmetry,
      title={Dubrovin duality and mirror symmetry for ADE resolutions}, 
      author={Andrea Brini and Jingxiang Ma and Ian A. B. Strachan},
      journal={Proceedings of the Royal Society A: Mathematical, Physical and Engineering Sciences},
      volume={481},
      number={2325},
      pages={20250047},
      year={2025},
      doi={10.1098/rspa.2025.0047},
      url={https://doi.org/10.1098/rspa.2025.0047},
}

@article{li2001degenerationstablemorphismsrelative,
      title={Stable morphisms to singular schemes and relative stable morphisms},
      author={Jun Li},
      journal={Journal of Differential Geometry},
      volume={57},
      number={3},
      pages={509--578},
      year={2001},
      doi={10.4310/jdg/1090348132},
      url={https://doi.org/10.4310/jdg/1090348132},
}

@misc{eynard2009geometricalinterpretationtopologicalrecursion,
      title={Geometrical interpretation of the topological recursion, and integrable string theories},
      author={Bertrand Eynard and Nicolas Orantin},
      year={2009},
      eprint={0911.5096},
      archivePrefix={arXiv},
      primaryClass={math-ph},
      url={https://arxiv.org/abs/0911.5096},
}

@article{kokotov2008hierarchyhurwitz,
      title={A new hierarchy of integrable systems associated to {Hurwitz} spaces},
      author={Alexey Kokotov and Dmitry Korotkin},
      journal={Philosophical Transactions of the Royal Society A},
      volume={366},
      number={1867},
      pages={1055--1088},
      year={2008},
      doi={10.1098/rsta.2007.2061},
      url={https://doi.org/10.1098/rsta.2007.2061},
}

@article{shramchenko2008riemannhilbertirregular,
      title={{Riemann--Hilbert} problem associated to {Frobenius} manifold structures on {Hurwitz} spaces: irregular singularity},
      author={Vasilisa Shramchenko},
      journal={Duke Mathematical Journal},
      volume={144},
      number={1},
      pages={1--52},
      year={2008},
      doi={10.1215/00127094-2008-030},
      url={https://doi.org/10.1215/00127094-2008-030},
      eprint={0707.0696},
      archivePrefix={arXiv},
      primaryClass={math-ph},
}

@incollection{duninbarkowski2016primaryinvariantshurwitz,
      title={Primary invariants of Hurwitz Frobenius manifolds},
      author={Petr Dunin-Barkowski and Paul Norbury and Nicolas Orantin and Alexandr Popolitov and Sergey Shadrin},
      booktitle={Topological recursion and its influence in analysis, geometry, and topology},
      series={Proceedings of Symposia in Pure Mathematics},
      volume={100},
      pages={297--331},
      publisher={American Mathematical Society},
      address={Providence, RI},
      year={2018},
      doi={10.1090/pspum/100/01768},
      url={https://doi.org/10.1090/pspum/100/01768},
}

@article{marino2006openstringamplitudeslarge,
      title={Open string amplitudes and large order behavior in topological string theory},
      author={Marcos Mari{\~n}o},
      journal={Journal of High Energy Physics},
      volume={2008},
      number={3},
      pages={060},
      year={2008},
      doi={10.1088/1126-6708/2008/03/060},
      url={https://doi.org/10.1088/1126-6708/2008/03/060},
}

@article{bouchard2007remodelingbmodel,
      title={Remodeling the B-Model},
      author={Vincent Bouchard and Albrecht Klemm and Marcos Mari{\~n}o and Sara Pasquetti},
      journal={Communications in Mathematical Physics},
      volume={287},
      number={1},
      pages={117--178},
      year={2009},
      doi={10.1007/s00220-008-0620-4},
      url={https://doi.org/10.1007/s00220-008-0620-4},
}

@article{bouchard2008topologicalopenstringsorbifolds,
      title={Topological Open Strings on Orbifolds},
      author={Vincent Bouchard and Albrecht Klemm and Marcos Mari{\~n}o and Sara Pasquetti},
      journal={Communications in Mathematical Physics},
      volume={296},
      number={3},
      pages={589--623},
      year={2010},
      doi={10.1007/s00220-010-1020-0},
      url={https://doi.org/10.1007/s00220-010-1020-0},
}

@article{chekhov2006freeenergytopologicalexpansion,
      title={Free energy topological expansion for the 2-matrix model},
      author={Leonid Chekhov and Bertrand Eynard and Nicolas Orantin},
      journal={Journal of High Energy Physics},
      volume={2006},
      number={12},
      pages={053},
      year={2006},
      doi={10.1088/1126-6708/2006/12/053},
      url={https://doi.org/10.1088/1126-6708/2006/12/053},
}

@misc{fang2025remodelingconjecturedescendants,
      title={Remodeling Conjecture with Descendants},
      author={Bohan Fang and Chiu-Chu Melissa Liu and Song Yu and Zhengyu Zong},
      year={2025},
      eprint={2504.15696},
      archivePrefix={arXiv},
      primaryClass={math.AG},
      url={https://arxiv.org/abs/2504.15696},
}

@article{borot2015chernsimonstheorysphericalseifert,
      title={Chern-Simons theory on spherical Seifert manifolds, topological strings and integrable systems},
      author={Ga{\"e}tan Borot and Andrea Brini},
      journal={Advances in Theoretical and Mathematical Physics},
      volume={22},
      number={2},
      pages={305--394},
      year={2018},
      doi={10.4310/ATMP.2018.v22.n2.a2},
      url={https://doi.org/10.4310/ATMP.2018.v22.n2.a2},
}

@article{giacchetto2024newspinhurwitztheory,
      title={A new spin on Hurwitz theory and ELSV via theta characteristics},
      author={Alessandro Giacchetto and Reinier Kramer and Danilo Lewa{\'n}ski},
      journal={Selecta Mathematica},
      volume={31},
      number={5},
      pages={90},
      year={2025},
      doi={10.1007/s00029-025-01077-y},
      url={https://doi.org/10.1007/s00029-025-01077-y},
}

@incollection{dubrovin1994geometry2dtopologicalfield,
      title={Geometry of 2D topological field theories},
      author={Boris Dubrovin},
      booktitle={Integrable Systems and Quantum Groups},
      series={Lecture Notes in Mathematics},
      volume={1620},
      pages={120--348},
      publisher={Springer},
      address={Berlin, Heidelberg},
      year={1996},
      doi={10.1007/BFb0094793},
      url={https://doi.org/10.1007/BFb0094793},
}

@article{dubrovin1996extendedaffineweylgroups,
      title={Extended affine Weyl groups and Frobenius manifolds},
      author={Boris Dubrovin and Youjin Zhang},
      journal={Compositio Mathematica},
      volume={111},
      number={2},
      pages={167--219},
      year={1998},
      doi={10.1023/A:1000258122329},
      url={https://doi.org/10.1023/A:1000258122329},
}

@article{dubrovin2015extendedaffineweylbcd,
      title={Extended affine Weyl groups of BCD-type: their Frobenius manifolds and Landau--Ginzburg superpotentials},
      author={Boris Dubrovin and Ian A. B. Strachan and Youjin Zhang and Dafeng Zuo},
      journal={Advances in Mathematics},
      volume={351},
      pages={897--946},
      year={2019},
      doi={10.1016/j.aim.2019.05.030},
      url={https://doi.org/10.1016/j.aim.2019.05.030},
}

@article{brini2023mirrorsymmetryextendedaffine,
      title={Mirror symmetry for extended affine Weyl groups},
      author={Andrea Brini and Karoline van Gemst},
      journal={Journal de l'{\'E}cole polytechnique -- Math{\'e}matiques},
      volume={9},
      pages={907--957},
      year={2022},
      doi={10.5802/jep.197},
      url={https://doi.org/10.5802/jep.197},
}

@misc{ju2026automatedconjectureresolutionformal,
      title={Automated Conjecture Resolution with Formal Verification},
      author={Haocheng Ju and Guoxiong Gao and Jiedong Jiang and Bin Wu and Zeming Sun and Shurui Liu and Leheng Chen and Yutong Wang and Yuefeng Wang and Zichen Wang and Wanyi He and Peihao Wu and Liang Xiao and Ruochuan Liu and Bryan Dai and Bin Dong},
      year={2026},
      eprint={2604.03789},
      archivePrefix={arXiv},
      primaryClass={cs.LG},
      url={https://arxiv.org/abs/2604.03789},
}

@misc{liu2026danusorchestratingmathematicalreasoning,
      title={Danus: Orchestrating Mathematical Reasoning Agents with Fact-Graph Memory},
      author={Jihao Liu and Guoxiong Gao and Zeming Sun and Bin Wu and Shurui Liu and Jiedong Jiang and Haocheng Ju and Leheng Chen and Ronnie Cheng and Xiping Zhang and Bin Dong},
      year={2026},
      eprint={2607.06447},
      archivePrefix={arXiv},
      primaryClass={cs.AI},
      url={https://arxiv.org/abs/2607.06447},
}

\end{document}